\documentclass[11pt,letterpaper]{amsart}
\usepackage{amsmath,amssymb,amsthm,eucal,multirow,enumerate,placeins}

\renewcommand{\phi}{\varphi}
\renewcommand{\epsilon}{\varepsilon}
\DeclareMathOperator{\even}{even}
\DeclareMathOperator{\odd}{odd}
\DeclareMathOperator{\disc}{disc}
\DeclareMathOperator{\Hyp}{Hyp}
\DeclareMathOperator{\Ell}{Ell}
\DeclareMathOperator{\Ig}{Ig}
\DeclareMathOperator{\norm}{norm}
\DeclareMathOperator{\Pl}{Pl}
\DeclareMathOperator{\Planes}{Planes}
\DeclareMathOperator{\rank}{rank}
\DeclareMathOperator{\RealPart}{Re}
\DeclareMathOperator{\Sq}{Sq}
\DeclareMathOperator{\Tr}{Tr}
\DeclareMathOperator{\Vol}{Vol}
\newcommand{\C}{{\mathbb C}}
\newcommand{\F}{{\mathbb F}}
\newcommand{\N}{{\mathbb N}}
\newcommand{\Q}{{\mathbb Q}}
\newcommand{\Z}{{\mathbb Z}}
\newcommand{\Ni}{\N\cup\{\infty\}}
\newcommand{\card}[1]{\#{#1}}
\newcommand{\sums}[1]{\sum_{\substack{#1}}}
\newcommand{\sumls}[1]{\sum\limits_{{\substack{#1}}}}
\newcommand{\prods}[1]{\prod_{\substack{#1}}}
\newcommand{\prodls}[1]{\prod\limits_{{\substack{#1}}}}
\newcommand{\bigopluss}[1]{\bigoplus_{\substack{#1}}}
\newcommand{\bigoplusls}[1]{\bigoplus\limits_{{\substack{#1}}}}
\newcommand{\ceil}[1]{\lceil{#1}\rceil}
\newcommand{\field}{K}
\newcommand{\val}{v_\unif}
\newcommand{\twoval}{v_2}
\newcommand{\resfield}{\F_q}
\newcommand{\resfieldunits}{\F_q^*}
\newcommand{\resfieldsquareunits}{\F_q^{*2}}
\newcommand{\unif}{\pi}
\newcommand{\ring}{R}
\newcommand{\ringunits}{\ring^*}
\newcommand{\ringsquareunits}{\ring^{*2}}
\newcommand{\unitsmodsquares}{\ringunits/\ringsquareunits}
\newcommand{\teich}{T}
\newcommand{\teichu}{\teich^*}

\newcommand{\quotring}[1]{R_{#1}}
\newcommand{\coset}[2]{{#1}+\unif^{#2}\ring}
\newcommand{\ncoset}[1]{\unif^{#1}\ring}
\newcommand{\gncoset}[1]{\gamma^{\ncoset{#1}}}
\newcommand{\zcoset}[2]{z^{\coset{#1}{#2}}}
\newcommand{\zncoset}[1]{z^{\ncoset{#1}}}
\newcommand{\gfgroup}[1]{\Gamma_{#1}}
\newcommand{\groupring}[1]{{\mathcal G}_{#1}}
\newcommand{\limgroupring}{{\mathcal G}}

\newcommand{\igr}{\frac{1-1/q}{1-t/q}}
\newcommand{\igrp}{\left(\igr\right)}
\newcommand{\Planar}[1]{\Planes(\pm_{#1})}

\newtheorem{theorem}{Theorem}[section]
\newtheorem{proposition}[theorem]{Proposition}
\newtheorem{lemma}[theorem]{Lemma}
\newtheorem{corollary}[theorem]{Corollary}
\newtheorem{remark}[theorem]{Remark}
\theoremstyle{remark}
\title[A New Generating Function for Calculating the Igusa Local Zeta Function]{A New Generating Function for Calculating \\ the Igusa Local Zeta Function}

\author{Raemeon A.~Cowan}

\author{Daniel J.~Katz}
\address{Department of Mathematics, California State University, Northridge, United States}

\author{Lauren M.~White}

\date{first version: 25 June 2015; this version: 31 August 2016}
\thanks{All three authors were partially supported by the National Science Foundation grant DMS 1247679.
The second author was partially supported by a Probationary Faculty Grant from California State University, Northridge.}

\begin{document}

\begin{abstract}
A new method is devised for calculating the Igusa local zeta function $Z_f$ of a polynomial $f(x_1,\dots,x_n)$ over a $p$-adic field.  This involves a new kind of generating function $G_f$ that is the projective limit of a family of generating functions, and contains more data than $Z_f$.  This $G_f$ resides in an algebra whose structure is naturally compatible with operations on the underlying polynomials, facilitating calculation of local zeta functions.  This new technique is used to expand significantly the set of quadratic polynomials whose local zeta functions have been calculated explicitly.  Local zeta functions for arbitrary quadratic polynomials over $p$-adic fields with $p$ odd are presented, as well as for polynomials over unramified $2$-adic fields of the form $Q+L$ where $Q$ is a quadratic form and $L$ is a linear form where $Q$ and $L$ have disjoint variables.  For a quadratic form over an arbitrary $p$-adic field with odd $p$, this new technique makes clear precisely which of the three candidate poles are actual poles.
\end{abstract}

\maketitle
\vspace{-4mm}
\section{Introduction}

In this paper, we devise a new method for calculating the Igusa local zeta function of a polynomial over a $p$-adic field.
Throughout this paper, $p$ is a prime, $\field$ is a $p$-adic field with finite residue field $\resfield$ of order $q$ and valuation ring $\ring$.  We let $\pi$ be a uniformizing parameter for $\ring$, and we have the $\pi$-adic valuation $\val$, and define the absolute value of an element $a \in \field$ to be $|a|_{\field}=q^{-\val(a)}$.
We let $\ringunits$ be the group of units in $\ring$ and let $\ringsquareunits$ be the group of square units in $\ring$.
We use $\N$ to denote the set of nonnegative integers.
We say that $K$ (or $R$) is unramified when the prime $p$ does not ramify in $R$, and in that case we always choose $\pi=p$.

The {\it Igusa local zeta function} of a polynomial $f(x_1,\ldots,x_n) \in \field[x_1,\ldots,x_n]$ is
\begin{equation}\label{James}
Z_f(s)=\int_{\ring^n} |f(x_1,\ldots,x_n)|_{\field}^s d x_1 \cdots d x_n,
\end{equation}
where $s \in \C$ with $\RealPart(s) > 0$ and $d x_1\cdots d x_n$ is a volume element for the Haar measure.
This kind of zeta function was introduced by Weil \cite{Weil-1965}, and studied extensively by Igusa \cite{Igusa-1974,Igusa-1975,Igusa-1978}.
See the monograph by Igusa \cite{Igusa-2000}, and also the report by Denef \cite{Denef-1991-Local} for an extensive survey.

Throughout this paper, we set $t=q^{-s}$ because Igusa \cite{Igusa-1974,Igusa-1975,Igusa-1978} showed that $Z_f(s)$ is a rational function of $t$, and by abuse of notation, we shall call this rational function $Z_f(t)$.
When $f(x_1,\ldots,x_n) \in R[x_1,\ldots,x_n]$, the local zeta function carries all information about the number $N_k(f)$ of zeroes that $f(x_1,\ldots,x_n) \pmod{\pi^k}$ has in $(\ring/\unif^k\ring)^n$ for every $k \in \N$.
Indeed, if we define the Poincar\'e series for $f$ to be
\[
P_f(t) = \sum_{i \in \N} \frac{N_i(f)}{q^{n i}} t^i,
\]
then one can show that 
\[
P_f(t) = \frac{1-t Z_f(t)}{1-t}.
\]
This relation makes the local zeta function interesting in arithmetic geometry.  For example, the real parts of the poles of $Z_f(s)$ are connected with the $p$-divisibility of the numbers $N_k(f)$ of zeroes modulo $\pi^k$.
Segers \cite{Segers-2006-Lower} proves that no pole of $Z_f(s)$ has real part less than $-n/2$ by noting that $q^{\ceil{(n/2)(i-1)}} \mid N_i(f)$ for each $i\in\N$ when $n>1$.  The sharpest possible lower bound on $p$-divisibility is given in \cite{Marshall-Ramage-1975}, which is extended to general algebraic sets in \cite{Katz-2009}.
Numerous other works study the poles of the local zeta function \cite{Meuser-1983,Strauss-1983,Lichtin-Meuser-1985,Veys-1990,Denef-1991-Local,Veys-1993,Denef-1995,Denef-Veys-1995,Denef-Hoornaert-2001,Zuniga-Galindo-2003-Local,Zuniga-Galindo-2004-Poles,Saia-Zuniga-Galindo-2005,Zuniga-Galindo-2005,Veys-2006,Veys-Zuniga-Galindo-2008,Melle-Hernandez-Torrelli-Veys-2009,Melle-Hernandez-Torrelli-Veys-2010}.

Many authors have labored on the calculation of local zeta functions in various situations \cite{Robinson-1996,Veys-1997,Albis-Zuniga-Galindo-1999,Denef-Hoornaert-2001,Zuniga-Galindo-2001,Zuniga-Galindo-2003-Computing,Marko-Riedl-2005,Rodriguez-Vega-2005,Saia-Zuniga-Galindo-2005,Ibadula-2006}, and many works either use local zeta functions or else apply the methods developed for obtaining them \cite{duSatoy-Grunewald-2000,Zuniga-Galindo-2004-Pseudo,Zuniga-Galindo-2006,Sakellaridis-2008,Klopsch-Voll-2009-Igusa,Klopsch-Voll-2009-Zeta,Zuniga-Galindo-2009,Voll-2010,Segers-Zuniga-Galindo-2011,Zuniga-Galindo-2011}.
Nonetheless, certain classes of polynomials have proved forbidding to those who wish to obtain general results.
One such example would be quadratic polynomials over $2$-adic fields, the local zeta functions of which are known to be useful \cite{Klopsch-Voll-2009-Zeta}, but which are considerably more challenging to obtain in general than in the case of odd $p$.

We propose a new method in this paper to improve the situation.
We now sketch the basic philosophy.
If $f_1,\ldots,f_k$ are polynomials, then we write $f_1\oplus\cdots\oplus f_k$ to denote the sum $f_1+\cdots+f_k$ and at the same time assert that no indeterminate appears in more than one of the $f_i$; we then say that $f$ is the {\it direct sum} of $f_1,\ldots,f_k$.
If a polynomial can be expressed as the direct sum of many polynomials with only a few indeterminates each, one has a better chance of being able to calculate its local zeta function.
Extreme cases of this would be diagonal forms such as Fermat varieties, whose local zeta functions have been studied in \cite{Marko-Riedl-2005}.
Even if $f$ and $g$ have distinct indeterminates, it is not possible to calculate the local zeta function of $f(x) \oplus g(y)$ from the local zeta functions of $f$ and $g$.
For a trivial example, suppose that $p$ is odd and $\alpha$ is a nonsquare unit in $\ring$, and that $f(x)=x^2$, $g(y)=\pi y-1$, and $h(y)=\pi y-\alpha$.
Then clearly $Z_g=Z_h=1$, and yet $Z_{f\oplus g}\not=1$ while $Z_{f\oplus h}=1$.
This is simply a manifestation of the fact that one cannot deduce the valuation of a sum $a+b$ from the individual valuations of $a$ and $b$.
Therefore, we invent a new object, called the {\it $p$-adic generating function}, that carries enough information so that the generating function for a direct sum can be deduced from the generating functions of the summands.
This $p$-adic generating function, which is the inverse limit of a family of generating functions, contains as much data about the polynomial as the collection of $Z_{f-c}$ for all $c \in \ring$.
Such generating functions cannot be expressed as simple polynomials: their terms have ``exponents'' that are sets, and yet because they reside in a group ring whose algebraic structure is naturally compatible with operations on the underlying polynomials, we obtain a straightforward calculus of generating functions that enables us to build up the generating function for a particular polynomial from smaller elements.
There is then a simple map from the $p$-adic generating functions to local zeta functions that forgets the extra information contained in the former.
This provides a new method for calculating local zeta functions that would have been very difficult to calculate with existing methods.

To demonstrate the use of this method, we calculate the local zeta function for an arbitrary quadratic polynomial when $p$ is odd, and for an arbitrary polynomial of the form $Q\oplus L$ with $Q$ a quadratic form and $L$ a linear form over an unramified $2$-adic field.  The results are given in Theorems \ref{Odilia} and \ref{Ursula} below.
For odd $p$, this extends the work of Igusa \cite{Igusa-1994-Local}, who considered polynomials of the form $Q\oplus L$, where $Q$ is a quadratic form and $L$ is a linear form.
The problem over $2$-adic fields, even the unramified ones, is far more difficult, and until now we lacked general theorems like the ones that Igusa had proved for the case of odd $p$.
The new method we present here has rendered the calculations of the old results far easier, has obtained new results, and should also enable the calculation of local zeta functions for many other non-quadratic polynomials.

Our new method also makes transparent the fact that the local zeta function of a quadratic form over any $p$-adic field has at most three poles (Remark \ref{Percy} below).  When $p$ is odd, it enables us to indicate the poles precisely in Theorem \ref{Arthur}.

We present our results for quadratic polynomials in the next section.  The rest of this paper is organized as follows.  In Section \ref{Henry}, we develop the general theory of the $p$-adic generating function.  In Section \ref{Elizabeth}, we apply the theory to quadratic polynomials to prove the main results presented in Section \ref{Martha}.  Section \ref{Elizabeth} in turn relies on calculations for unimodular quadratic forms that are presented in Section \ref{Carlos}, with technical proofs in the Appendix.

\section{Results: Local Zeta Functions for Quadratic Polynomials}\label{Martha}

In this section, we present our results on the local zeta functions of quadratic polynomials, after some introductory material on quadratic polynomials.

\subsection{General Remarks on Quadratic Polynomials}\label{Gordon}

To any $n\times n$ symmetric matrix $M$ with entries in $R$ we associate the form $Q(x_1,\ldots,x_n)=(x_1,\ldots,x_n) M (x_1,\ldots,x_n)^T$.  When we speak of a {\it quadratic form over $R$}, we mean one that can be obtained in this manner.  When $p=2$, one can scale by $\frac{1}{2}$ (and use diagonal entries in $2\ring$) to be able to obtain an arbitrary homogeneous polynomial of degree $2$ in $\ring[x_1,\ldots,x_n]$, and note that this simply scales the local zeta function by the $\pi$-adic absolute value of $1/2$.
A {\it unimodular matrix} is a matrix with entries in $\ring$ and determinant in $\ringunits$, and a quadratic form $Q$ is said to be {\it unimodular} if its associated matrix is unimodular.

We say that two quadratic forms $Q_1$ and $Q_2$ are {\it equivalent} and write $Q_1\cong Q_2$ if one can be obtained from the other by an invertible $\ring$-linear change of coordinates; these clearly have the same local zeta function.
Equivalence preserves the following three invariants:
\begin{itemize}
\item The {\it rank} of a quadratic form, written $\rank(Q)$, is the rank of its associated matrix.
\item The {\it discriminant} of a quadratic form $Q$, written $\disc(Q)$, is the element of the set $(\field^*/\ringsquareunits) \cup \{0\}$ that contains the determinant of the matrix associated to $Q$.  By a common abuse of terminology, we often say that some $a \in \field$ is the discriminant of a quadratic form to mean that the discriminant is $a\ringsquareunits$.
\item The {\it norm} of a quadratic form $Q(x_1,\ldots,x_n)$ is the ideal of $\ring$ generated by the set $Q(\ring^n)$ of all elements represented by the form.  For unimodular quadratic forms over $p$-adic fields with $p$ odd, the norm is always $\ring$ itself or $0$ (for the zero form).
When $p=2$, a unimodular form has a norm which is an ideal $I$ with $2\ring \subseteq I \subseteq R$ or $I=0$ (for the zero form).
\end{itemize}
Note that the zero form is considered to be unimodular with rank $0$, norm $0$, and discriminant $1$.

It is a fact \cite[\S 91C]{O'Meara-1963} that every quadratic form over $\ring$ can be brought by some equivalence to a form $\bigoplus_{i=0}^\infty \pi^i Q_i$, where $Q_0, Q_1,\ldots$ are unimodular quadratic forms, almost all of which are zero.  (See \cite[\S 94]{O'Meara-1963} for procedures.)
When $p$ is odd, two such decompositions of the same form must be identical up to equivalence of their $i$th components for every $i$ (see \cite[\S 92:2]{O'Meara-1963}).  When $p=2$, this is not the case (cf.~\cite[\S 93:28]{O'Meara-1963} and \cite{Jones-1944}).
For the purposes of our calculations, we assume that our quadratic forms have been transformed into such a decomposition.

Now we should understand the unimodular forms $Q_i$ that make up our decomposition.  For all $p$, each unimodular form is the direct sum of unimodular forms of rank $1$ and $2$ (see \cite[\S 91C, 92:1]{O'Meara-1963}).
If $p$ is odd, we may use solely rank $1$ forms, but we actually find it more convenient to decompose our unimodular forms into both rank $1$ and $2$ forms.  We now describe some basic unimodular forms of interest for all $p$.
\begin{itemize}
\item For $a \in \ringunits$, we use $a \Sq$ to denote the quadratic form $a x^2$.
\item We write $\Hyp$ for the {\it hyperbolic plane} $2 x y$, and we write $\Hyp^n$ for the direct sum of $n$ hyperbolic planes.
\item We write $\Ell$ for the {\it elliptic plane}, which is the rank $2$ form $2 y^2 f(x/y)$ with $f(X)$ a quadratic polynomial over $\ring$ whose reduction modulo $\pi$ is quadratic and irreducible over $\resfield$.  All such forms are equivalent, regardless of the choice of polynomial.
\end{itemize}
In Sections \ref{Olivia} (for $p$ odd) and \ref{Caesar} (for $p=2$ and unramified), we shall use these rank $1$ and $2$ forms to build arbitrary unimodular forms (up to equivalence).

We now consider arbitrary quadratic polynomials, that is, ones that may have linear and constant terms.
We say that two polynomials $f$ and $g$ over $\ring$ are {\it isospectral} if they have the same Igusa local zeta function, which (via the Poincar\'e series) is equivalent to saying that $f$ and $g$ have the same number of zeroes modulo $\unif^k$ for every $k \in \N$.  We say that $f$ and $g$ are {\it strongly isospectral} if $f$ and $g$ modulo $\pi^k$ represent each value in $\ring/\pi^k\ring$ the same number of times for every $k \in \N$, or equivalently, $f$ and $g$ are strongly isospectral if $f-c$ is isospectral to $g-c$ for every constant $c \in\ring$.
When $p$ is odd (and also when $\ring=\Z_2$) it turns out (see Proposition \ref{Theresa}) that every quadratic polynomial is strongly isospectral to a polynomial of the form
\begin{equation}\label{Angelo}
\bigoplus_{i=0}^{\omega} \pi^i Q_i \oplus \pi^\lambda L + c,
\end{equation}
where $Q_0, \ldots, Q_\omega$ are unimodular quadratic forms, and $L$ is a linear form involving at most one variable with $\lambda > \omega$, and $c$ is a constant in $\ring$.
Thus when $p$ is odd or when $R=\Z_2$ we may always assume that a quadratic polynomial whose local zeta function we are seeking has been replaced with a polynomial in form \eqref{Angelo}.
In fact, in our computation of the Igusa local zeta function when $p$ is odd (Theorem \ref{Odilia} below), we do not need to assume that $\lambda > \omega$, but even if we had some nonzero terms $Q_i$ with $i \geq\lambda$, the Igusa local zeta function would not be changed if we replaced each such $Q_i$ with $0$ (this is is a consequence of Lemma \ref{Ludwig} below).
We note that in computing the local zeta function for polynomials of the form $\bigoplus_{i=0}^{\omega} \pi^i Q_i \oplus \pi^\lambda L$ with $p$ odd, Igusa employed the assumption $\lambda \geq \omega$ ($e \geq e_t$ in his notation in the Corollary to Theorem 2 in \cite{Igusa-1994-Local}), and so we see that any quadratic polynomial in our standard form \eqref{Angelo} with $c=0$ is indeed strongly isospectral to one of those considered by Igusa.

\subsection{Quadratic Polynomials for Odd $p$}\label{Olivia}

When $p$ is odd, $\unitsmodsquares$ is a group of order $2$, and we let $\eta$ be the nontrivial character from this group to $\{\pm 1\}$ and extend $\eta$ so that we also write $\eta(b)$ when $b \in \ringunits$ to mean $\eta(b\ringsquareunits)$, and take $\eta(b)=0$ when $\pi\mid b$.
We fix $\alpha \in \ringunits\smallsetminus\ringsquareunits$.
On Table \ref{Wendy} we list the unimodular quadratic forms up to equivalence: every unimodular quadratic form is equivalent to {\it precisely one} form on our table (see \cite[\S 92:1a]{O'Meara-1963}).  Thus rank and discriminant classify the unimodular forms.  We may now compute local zeta functions.
\begin{table}[ht]
\caption{Possible Unimodular Forms for Odd $p$ (up to equivalence)}\label{Wendy}
\begin{center}
\vspace{-3mm}
\begin{tabular}{|c|c|c|}
\multicolumn{3}{c}{$\alpha \in \ringunits\smallsetminus\ringsquareunits$} \\
\hline
rank $r$ & discriminant & possible forms \\
\hline
\hline
\multirow{2}{*}{even} & $(-1)^{r/2}$ & $\Hyp^{r/2}$ \\
& $(-1)^{r/2} \alpha$ & $\Ell\oplus\Hyp^{(r-2)/2}$ \\
\hline
\multirow{2}{*}{odd} & $(-1)^{(r-1)/2}$ & $\Sq \oplus \Hyp^{(r-1)/2}$ \\
& $(-1)^{(r-1)/2}\alpha$ & $\alpha\Sq \oplus \Hyp^{(r-1)/2}$ \\
\hline
\end{tabular}
\end{center}
\end{table}
\FloatBarrier
\begin{theorem}\label{Odilia}
Let $p$ be odd.
Consider the quadratic polynomial $Q=\bigoplus_{i\in\N} \unif^i Q_i \oplus L + c$, where each $Q_i$ is a unimodular quadratic form of rank $r_i$ and discriminant $d_i$ over $\ring$, where there is a positive integer $\omega$ such that $Q_i=0$ for $i > \omega$, where $L$ is a linear form involving at most one indeterminate, and $c \in \ring$ is a constant.
For each $j \in \N$, we let
\begin{center}
\begin{tabular}{lll}
& &  \\
$Q_{(j)} = \!\!\!\!\! \bigoplusls{0 \leq i \leq j \\ i \equiv j \pmod{2}} \!\!\!\!\! Q_i,$ & & $d_{(j)} =\disc(Q_{(j)})=\!\!\!\!\! \prodls{0 \leq i \leq j \\ i \equiv j \pmod{2}} \!\!\!\!\! d_i$, \\
 & & \\
$r_{(j)} =\rank(Q_{(j)})=\!\!\!\!\! \sumls{0 \leq i \leq j \\ i \equiv j \pmod{2}} \!\!\!\!\! r_i$, & & $q_{(j)}  = q^{\,\sum_{0 \leq i < j} r_{(i)}}$. \\
\end{tabular}
\end{center}
We have the following forms for the Igusa local zeta function $Z_Q(t)$ of $Q$ expressed using the terms $I_a(r_{(i)},d_{(i)})$ from Table \ref{Gary}.
\begin{itemize}
\item If $L=0$ and $c=0$, let $r=\sum_{i \in \N} r_i$, and then
\begin{multline*}
\quad\quad\quad Z_Q(t)=\sum_{0 \leq i < \omega-1} \frac{t^i}{q_{(i)}} I_0(r_{(i)},d_{(i)}) \\ + \left(\frac{t^{\omega-1}}{q_{(\omega-1)}} I_0(r_{(\omega-1)},d_{(\omega-1)}) + \frac{t^\omega}{q_{(\omega)}} I_0(r_{(\omega)},d_{(\omega)})\right) \left(1-\frac{t^2}{q^r}\right)^{-1}.
\end{multline*}
\item If $L(x)=b x$ for some $b$ with $\val(b)=\lambda < \infty$, and if $\val(c) \geq \val(b)$, then 
\[
Z_Q(t)=\sum_{0 \leq i < \lambda} \frac{t^i}{q_{(i)}} I_0(r_{(i)},d_{(i)}) + \frac{t^\lambda}{q_{(\lambda)}} \igrp.
\]
\item If $L(x)=b x$ for some $b$ with $\val(b) > \val(c)$ (this includes the case where $L=0$, $c\not=0$), let $\kappa=\val(c)$, and then
\[
Z_Q(t)=\sum_{0 \leq i \leq \kappa} \frac{t^i}{q_{(i)}} I_{c/\pi^i}(r_{(i)},d_{(i)}) + \frac{t^{\kappa}}{q_{(\kappa+1)}}.
\]
\end{itemize}
\end{theorem}
Theorem \ref{Odilia} is proved in Section \ref{Winston}.
\begin{table}[ht]
\caption{$I_a(r,d)$ for Theorem \ref{Odilia}}\label{Gary}
\begin{center}
\vspace{-3mm}
\begin{tabular}{|c|c|c|}
\multicolumn{3}{c}{$\eta(b)=\begin{cases} 1 & \text{if $b \in \ringsquareunits$}, \\ -1 & \text{if $b \in \ringunits\smallsetminus\ringsquareunits$} \end{cases}$} \\
\hline
\multirow{4}{*}{$r$ odd} & \multirow{2}{*}{$\pi\mid a$} & \multirow{2}{*}{$\left(1-\frac{t}{q^r} \right) \igrp$} \\
& & \\
& \multirow{2}{*}{$\pi\nmid a$} & \multirow{2}{*}{$\left(1 + \frac{\eta(a(-1)^{(r+1)/2} d)}{q^{(r+1)/2}} t\right) \igrp -\frac{1}{q^r} -\frac{\eta(a(-1)^{(r+1)/2} d)}{q^{(r+1)/2}}$} \\
& & \\
\hline
\multirow{4}{*}{$r$ even} & \multirow{2}{*}{$\pi\mid a$} & \multirow{2}{*}{$\left(1 - \frac{\eta((-1)^{r/2} d)}{q^{r/2}}\right) \left(1+\frac{\eta((-1)^{r/2} d)}{q^{r/2}} t\right) \igrp$} \\
& & \\
& \multirow{2}{*}{$\pi\nmid a$} & \multirow{2}{*}{$\left(1 - \frac{\eta((-1)^{r/2} d)}{q^{r/2}}\right) \left(\igrp + \frac{\eta((-1)^{r/2} d)}{q^{r/2}}\right)$} \\
& & \\
\hline
\end{tabular}
\end{center}
\end{table}

\subsection{Quadratic Polynomials over Unramified $2$-Adic Fields}\label{Caesar}

We now consider the case where the prime $p$ is $2$ and does not ramify in $\ring$.
We will need to use the absolute Galois-theoretic trace $\Tr \colon \field \to \Q_2$ in some of our formulations, and use the convention that if $a \in \Z_2$, then $(-1)^a=(-1)^{(a \bmod{2})}$.
We also need to know something about the squares in $\ring$.
Because the residue field is perfect and of characteristic $2$, every element $a$ of $\ring$ is a square modulo $2$, that is, $a\equiv b^2 \pmod{2}$ for some $b \in \ring$.
Thus, without loss of generality, every form $a\Sq$ with $a \in \ringunits$ is equivalent to a form $b\Sq$ where $b\equiv 1 \pmod{2}$, so we shall always insist that such coefficients are $1$ modulo $2$ in this section.
Furthermore, every element of the form $1+8 a$ with $a \in R$ is a square (so $\unitsmodsquares$ is finite), but there is a unit $\xi\in R$ such that $1+4 \xi$ is not a square (see \cite[\S 63:1 and 64:4]{O'Meara-1963} for proofs of these facts).
In fact, it is an easy consequence of Hilbert's Theorem 90 that the units $\xi$ such that $1+4\xi$ is not square are precisely those for which $\Tr(\xi) \equiv 1 \pmod{2}$.

On Table \ref{Melissa} we list the possible unimodular quadratic forms up to equivalence (see Corollary \ref{Frank} below for a proof).  A unimodular quadratic form is equivalent to {\it at least} one form on our table: there is some duplication, which would bring no convenience to us to eliminate.
\begin{table}[ht]
\caption{Possible Unimodular Forms over Unramified $2$-Adic Fields (up to equivalence)}\label{Melissa}
\begin{center}
\vspace{-3mm}
\begin{tabular}{|c|c|c|c|}
\multicolumn{4}{c}{$\xi \in \ringunits$ with $\Tr(\xi) \equiv 1 \pmod{2}$} \\
\multicolumn{4}{c}{$a\equiv b \equiv 1 \pmod{2}$} \\
\hline
rank $r$ & norm & discriminant & possible forms \\
\hline
\hline
\multirow{2}{*}{even} & $2\ring$ (for $r\not=0$) & $(-1)^{r/2}$ & $\Hyp^{r/2}$ \\
& or $0$ (for $r=0$) &  $(-1)^{r/2} (1+4\xi)$ & $\Ell\oplus\Hyp^{(r-2)/2}$ \\
\hline
\multirow{2}{*}{even} & \multirow{2}{*}{$\ring$} & $(-1)^{(r-2)/2} a b$ & $a\Sq\oplus b\Sq\oplus\Hyp^{(r-2)/2}$ \\
& &  $(-1)^{(r-2)/2} (1+4\xi) a b$ & $a\Sq\oplus b\Sq \oplus \Ell\oplus\Hyp^{(r-4)/2}$ \\
\hline
\multirow{2}{*}{odd} & \multirow{2}{*}{$\ring$} & $(-1)^{(r-1)/2} a$ & $a\Sq\oplus\Hyp^{(r-1)/2}$ \\
& &  $(-1)^{(r-1)/2} (1+4\xi) a$ & $a\Sq\oplus \Ell\oplus\Hyp^{(r-3)/2}$ \\
\hline
\end{tabular}
\end{center}
\end{table}
\FloatBarrier

In order to use our theorem below to calculate the local zeta function, one needs to know how to express the direct sum of any pair of unimodular forms from Table \ref{Melissa} as another form on Table \ref{Melissa}.  If $Q_1$ and $Q_2$ are two such forms, then $Q_1 \oplus Q_2$ is another unimodular form whose rank, norm, and discriminant are respectively the sum, the sum (as ideals), and the product of those invariants for the two constituent forms.  However, these three invariants are not always enough to identify a unimodular form.  To precisely determine $Q_1 \oplus Q_2$, one can make use of the following ``addition rules.''
\begin{lemma}\label{Eric}
Suppose that $p=2$ and does not ramify in $\ring$.  Let $a,b,c \in \ringunits$.  Then
\begin{enumerate}[(i).]
\item $\Ell\oplus\Ell\cong \Hyp\oplus\Hyp$, and
\item\label{Jane} $a\Sq \oplus b\Sq \oplus c\Sq \cong (-a b c)\Sq \oplus \Hyp$ if there are $r,s,t \in \ring$ with $a r^2 + b s^2 + c t^2 \equiv - a b c \pmod{8}$.  Otherwise $a\Sq \oplus b\Sq \oplus c\Sq \cong (-a b c)(1+4\xi)\Sq \oplus \Ell$.
\end{enumerate}
\end{lemma}
This is proved in Section \ref{Leonard}.  For rule \eqref{Jane}, since there are only finitely many elements modulo $8$, one can work out which case one obtains with finitely many trials.  As an example, Table \ref{Roderick} shows how the addition rules work when $\ring=\Z_2$, the ring of $2$-adic integers.

\begin{table}[ht]
\caption{Addition Rules Unimodular Quadratic Forms over $\Z_2$}\label{Roderick}
\begin{center}
\vspace{-3mm}
{\renewcommand{\arraystretch}{1.1}
\begin{tabular}{|c|c|}
\hline
\multicolumn{2}{|c|}{$\Ell\oplus\Ell\cong\Hyp\oplus\Hyp$} \\
\hline
\multicolumn{2}{c}{} \\
\hline
\multicolumn{2}{|c|}{$a x^2+b y^2+c z^2 \cong d w^2\oplus P$} \\
\hline
$(a,b,c)$ & $(d,P)$ \\
\hline
$(1,3,5)$ \quad $(1,1,7)$ \quad $(5,5,7)$ & $(1,\Hyp)$ \\
\hline
$(1,3,7)$ \quad $(3,3,5)$ \quad $(5,7,7)$ & $(3,\Hyp)$ \\
\hline
$(1,5,7)$ \quad $(3,5,5)$ \quad $(1,1,3)$ & $(5,\Hyp)$ \\
\hline
$(3,5,7)$ \quad $(1,7,7)$ \quad $(1,3,3)$ & $(7,\Hyp)$ \\
\hline
$(3,3,3)$ \quad $(3,7,7)$ & $(1,\Ell)$ \\
\hline
$(1,1,1)$ \quad $(1,5,5)$ & $(3,\Ell)$ \\
\hline
$(7,7,7)$ \quad $(3,3,7)$ & $(5,\Ell)$ \\
\hline
$(5,5,5)$ \quad $(1,1,5)$ & $(7,\Ell)$ \\
\hline
\end{tabular}
}
\end{center}
\end{table}

Now we are ready to compute the local zeta function with the following theorem, which uses Table \ref{Violet}.  We use a shorthand for quadratic forms in that table: $\Planes(+)$ means a direct sum of hyperbolic planes (with the correct number needed to achieve a specified rank) and $\Planes(-)$ is the direct sum of a single elliptic plane and the correct number of hyperbolic planes to achieve a specified rank.  For example, if we say that $\rank(Q_i)=r_i$ and $Q_i=a\Sq\oplus\Planar{i}$, then we mean that either $Q_i=a\Sq\oplus \Hyp^{(r_i-1)/2}$ (if $\pm_i=+$) or $Q_i=a\Sq\oplus \Ell\oplus\Hyp^{(r_i-3)/2}$ (if $\pm_i=-$).
\newpage
\begin{theorem}\label{Ursula}
Suppose that $p=2$ and does not ramify in $\ring$.
Consider the quadratic polynomial $Q=\bigoplus_{i\in\N} \unif^i Q_i \oplus L$, where each $Q_i$ is a unimodular quadratic form of rank $r_i$ over $\ring$, where there is a positive integer $\omega$ such that $Q_i=0$ for $i > \omega$, and where $L$ is a linear form involving at most one indeterminate.
For each $j \in \N$, we let
\[
Q_{(j)}  = \!\!\!\!\!\!\!\! \bigopluss{0 \leq i \leq j \\ i \equiv j \!\!\!\! \pmod{2}} Q_i,  \quad\quad\quad r_{(j)} =\rank(Q_{(j)})=\!\!\!\!\!\!\!\! \sums{0 \leq i \leq j \\ i \equiv j \!\!\!\! \pmod{2}} \!\!\!\!\! r_i, \quad\quad\text{and}\quad\quad q_{(j)}  = q^{\,\sum_{0 \leq i < j} r_{(i)}}.
\]
We have the following forms for the Igusa local zeta function $Z_Q(t)$ of $Q$ expressed using the terms $I(P_0,P_1,P_2)$ from Table \ref{Violet}.
\begin{itemize}
\item If $L=0$, let $r=\sum_{i \in \N} r_i$, and then 
\begin{multline*}
\quad\quad\quad Z_Q(t)=\sum_{0 \leq i < \omega-1} \frac{t^i}{q_{(i)}} I(Q_{(i)},Q_{(i+1)},Q_{i+2}) \\
+ \left(\frac{t^{\omega-1}}{q_{(\omega-1)}} I(Q_{(\omega-1)},Q_{(\omega)},0) + \frac{t^\omega}{q_{(\omega)}} I(Q_{(\omega)},Q_{(\omega-1)},0)\right) \left(1-\frac{t^2}{q^r}\right)^{-1}.
\end{multline*}
\item If $L=a x$ with $2 \leq \twoval(a)<\infty$, write $\lambda=\twoval(a)$, and then 
\begin{multline*}
\quad\quad\quad Z_Q(t)=\sum_{0 \leq i < \lambda-2} \frac{t^i}{q_{(i)}} I(Q_{(i)},Q_{(i+1)},Q_{i+2}) + \frac{t^{\lambda-2}}{q_{(\lambda-2)}} I(Q_{(\lambda-2)},Q_{(\lambda-1)},\Sq) \\ + \frac{t^{\lambda-1}}{q_{(\lambda-1)}} I(Q_{(\lambda-1)},\Sq,\Sq) + \frac{t^\lambda}{q_{(\lambda)}} \igrp.
\end{multline*}
\item If $L=a x$ with $\twoval(a)=1$, then 
\[
Z_Q(t)=I(Q_{(0)},\Sq,\Sq)  + \frac{t}{q^{r_0}} \igrp.
\]
\item If $L=a x$ with $\twoval(a)=0$, then $Z_Q(t)=\igrp$.
\end{itemize}
\end{theorem}
Theorem \ref{Ursula} is proved in Section \ref{Wilbur}.
\begin{table}[ht]
\caption{$I(Q_0,Q_1,Q_2)$ for Theorem \ref{Ursula}}\label{Violet}
\begin{center}
\vspace{-3mm}
\begin{tabular}{|c|c|}
\multicolumn{2}{c}{$r_i=\rank(Q_i)$ for $i=0,1,2$} \\
\multicolumn{2}{c}{$a \equiv b \equiv c \equiv d \equiv 1 \pmod{2}$}\\
\multicolumn{2}{c}{$\Planes(\pm)$ notation described before Theorem \ref{Ursula}}\\
\hline
\multicolumn{2}{|c|}{$Q_0 = \Planar{0}$} \\
\hline
$\norm(Q_1)\not=\ring$ & $\left(1\mp_0 \frac{1}{q^{r_0/2}}\right) \left(t \pm_0 \frac{t^2}{q^{r_0/2}}\right) \igrp$ \\
\hline
$\norm(Q_1)=\ring$ & $\left(1- \frac{1}{q^{r_0}}\right) t \igrp$ \\
\hline
\hline
\multicolumn{2}{|c|}{$Q_0 = a\Sq\oplus \Planar{0}$} \\
\hline
$\norm(Q_1)\not=\ring$ & $\left(1-\frac{t^2}{q^{r_0}} \pm_0 \frac{t^2-t}{q^{(r_0+1)/2}}\right) \igrp$ \\
\hline
$\norm(Q_1)=\ring$ & $\left(1- \frac{t}{q^{r_0}}\right) \igrp$ \\
\hline
\hline
\multicolumn{2}{|c|}{$Q_0 = a\Sq\oplus b\Sq \Planar{0}$ with $4 \mid a+b$ and $\sigma=(-1)^{\Tr((a+b)/(4 a))}$} \\
\hline
$Q_1=\Planar{1}$, $\norm(Q_2)\not=\ring$ & $\left(1 -\frac{t^2}{q^{r_0}} \pm_0 \frac{t^2-t}{q^{r_0/2}} \pm_1 \frac{\sigma (t^3-t^2)}{q^{r_0+r_1/2}}\right) \igrp$ \\
\hline
$\norm(Q_1)\not=\ring$, $\norm(Q_2)=\ring$ & $\left(1 -\frac{t^2}{q^{r_0}} \pm_0 \frac{t^2-t}{q^{r_0/2}}\right) \igrp$ \\
\hline
$Q_1=c\Sq\oplus\Planar{1}$ & \multirow{2}{*}{$\left(1- \frac{t}{q^{r_0}} \pm_1 \frac{\sigma (t^3-t^2)}{q^{r_0+(r_1+1)/2}} \right) \igrp$} \\
$\norm(Q_2)\not=\ring$ & \\
\hline
$Q_1=c\Sq\oplus d\Sq\oplus\Planar{1}$ & \multirow{3}{*}{$\left(1- \frac{t}{q^{r_0}} \pm_1 \frac{\sigma(t^3-t^2)}{q^{r_0+r_1/2}}\right) \igrp$} \\
$4 \mid a+b+c+d$ & \\
$\norm(Q_2)\not=\ring$  & \\
\hline
$Q_1=c\Sq\oplus d\Sq\oplus\Planar{1}$ & \multirow{3}{*}{$\left(1- \frac{t}{q^{r_0}}\right) \igrp$} \\
$4 \nmid a+b+c+d$ & \\
$\norm(Q_2)\not=\ring$  & \\
\hline
$\norm(Q_1)=\norm(Q_2)=\ring$ & $\left(1- \frac{t}{q^{r_0}}\right) \igrp$ \\
\hline
\hline
\multicolumn{2}{|c|}{$Q_0 = a\Sq\oplus b\Sq \Planar{0}$ with $4 \nmid a+b$} \\
\hline
$\norm(Q_1)\not=\ring$ & $\left(1-\frac{t^2}{q^{r_0}}\right) \igrp$ \\
\hline
$Q_1=c\Sq\oplus\Planar{1}$ & \multirow{3}{*}{$\left(1 - \frac{t}{q^{r_0}} \pm_1 \frac{\phi(t^3-t^2)}{q^{r_0+(r_1+1)/2}} \right)\igrp$} \\
$\phi=(-1)^{\Tr\left(\frac{c}{2} \left[\left(\frac{a+b}{2 c}\right)^q+\frac{a+b}{2 c}\right]\right)}$ & \\
$\norm(Q_2)\not=\ring$ & \\
\hline
$Q_1=c \Sq \oplus d \Sq \oplus \Planar{1}$ & \multirow{4}{*}{$\left(1-\frac{t}{q^{r_0}} \pm_1 \frac{\psi(t^3-t^2)}{q^{r_0+r_1/2}}\right) \igrp$} \\ 
$4\mid a+b+c+d$ & \\
$\psi=(-1)^{\Tr\left(\frac{c}{2}\left[\left(\frac{c+d}{2}\right)^q+\frac{a+b}{2 c}\right]\right)}$ & \\
$\norm(Q_2)\not=\ring$ & \\
\hline
$Q_1=c\Sq\oplus d\Sq\oplus\Planar{1}$ & \multirow{3}{*}{$\left(1- \frac{t}{q^{r_0}}\right) \igrp$}\\
$4 \nmid a+b+c+d$ & \\
$\norm(Q_2)\not=\ring$  & \\
\hline
$\norm(Q_1)=\norm(Q_2)=\ring$ & $\left(1- \frac{t}{q^{r_0}}\right) \igrp$\\
\hline
\end{tabular}
\end{center}
\end{table}
\FloatBarrier

\subsection{Location of Poles}

When $p$ is odd, if we specialize the quadratic polynomial in Theorem \ref{Odilia} to have constant term $c=0$, then the theorem reduces to the results of Igusa \cite[Theorem 1 and Corollary to Theorem 2]{Igusa-1994-Local}.
Let us compare our expression to Igusa's for $Z_Q(t)$ when $Q$ is a quadratic form.
If we adopt the notation of Theorem \ref{Odilia} here, then Igusa's expression for $Z_Q(t)$ is a sum of terms, one for each nonzero $Q_i$, such that Igusa's $i$th term is expressed as a polynomial divided by
\[
\left(1-\frac{t}{q}\right)\left(1-\frac{t^2}{q^{\sum_{0 \leq j < i} r_j}}\right)\left(1-\frac{t^2}{q^{\sum_{0 \leq j \leq i} r_j}}\right).
\]
Peter \cite[Proof of Lemma 3.7]{Peter-2000} notes that Igusa's expression makes it appear that the local zeta function might have many more poles than it actually has.  The advantage of our expression is that it makes clear that there are at most three poles.  Indeed, since every $I_a(r,d)$ in Table \ref{Gary} is expressible as $f(t)/(1-t/q)$ for some polynomial $f(t) \in \Q[t]$, our formula implies the following.
\begin{remark}\label{Percy}
The Igusa local zeta function of a quadratic form $Q$ of rank $r$ can be written as $\frac{f(t)}{\left(1-t/q\right)\left(1-t^2/q^r\right)}$ for some polynomial $f(t) \in \Q[t]$.
Note that we do not stipulate that $p$ be odd, for in fact, this is also true when $p=2$ (in both the unramified and ramified cases), as is shown in Section \ref{Michael}.
This fact also follows directly from Igusa's proof of the rationality of the local zeta function via resolution of singularities.
\end{remark}
When $p$ is odd, we can go further, and determine precisely the denominator of $Z_Q(t)$ when it is written in reduced form.
\begin{theorem}\label{Arthur}
Let $p$ be odd, let $\alpha \in \ringunits\smallsetminus\ringsquareunits$, and let $\eta(b)=1$ for $b \in \ringsquareunits$ and $\eta(b)=-1$ for $b\in\ringunits\smallsetminus\ringsquareunits$.
Let $Q$ be a quadratic form over $\ring$, equivalent to the form $\bigoplus_{i \in \N} \pi^i Q_i$, where each $Q_i$ is a unimodular quadratic form of rank $r_i$ and discriminant $d_i$, and almost all $Q_i$ are zero.  Let
\begin{align*}
r & = \sum_{i \in\N} r_i, \text{\quad\quad} r_{\even} = \sum_{i \in \N} r_{2 i}, \text{\quad\quad} r_{\odd} = \sum_{i \in \N} r_{2 i+1}, \\
d & = \prod_{i \in \N} d_i, \text{\quad\quad} d_{\even} = \prod_{i \in \N} d_{2 i}, \text{\quad\quad} d_{\odd} = \prod_{i \in \N} d_{2 i+1}.
\end{align*}
Write the Igusa local zeta function $Z_Q(t)$ of $Q$ as $f(t)/g(t)$, where $f(t), g(t) \in \Q[t]$ with $\gcd(f(t),g(t))=1$ and with the constant coefficient of $g(t)$ equal to $1$.
\begin{enumerate}[(a).]
\item If $r=0$, then $g(t)=1$.
\item If $r>0$ and $\{(r_{\even},d_{\even}), (r_{\odd},d_{\odd})\} \subseteq \{(0,1),(1,1),(1,\alpha),(2,-\alpha)\}$, then
\begin{enumerate}[(i).]
\item if $r_{\even}=r_{\odd}$, then $g(t)=1-t/q^{r/2}$, and
\item if $r_{\even}\not=r_{\odd}$, then $g(t)=1-t^2/q^r$.
\end{enumerate}
\item If $r > 0$ and $\{(r_{\even},d_{\even}),(r_{\odd},d_{\odd})\} \not\subseteq \{(0,1),(1,1),(1,\alpha),(2,-\alpha)\}$, then
\begin{enumerate}[(i).]
\item if $r_{\even} \equiv r_{\odd} \equiv 1\pmod{2}$, then $g(t)=1-t/q$,
\item if $r_{\even} \equiv r_{\odd} \equiv 0\pmod{2}$, then $g(t)=(1-t/q)(1-\eta((-1)^{r/2} d) t/q^{r/2})$, and
\item if $r_{\even} \not\equiv r_{\odd} \pmod{2}$, then $g(t)=(1-t/q)(1-t^2/q^r)$.
\end{enumerate}
\end{enumerate}
\end{theorem}
This is proved in Section \ref{Morton}.

\section{The Theory of the $p$-Adic Generating Function}\label{Henry}

Now we proceed to the general theory of the $p$-adic generating function, which we then use (in Section \ref{Elizabeth}) to obtain all the results of the previous section.
\subsection{Modular Generating Functions}

We let $\quotring{k} = \ring/\unif^{k}\ring$ and we let $\val(0)=\infty$ and formally regard $\unif^\infty$ as $0$, so that $\quotring{\infty}\cong R$.
Consider a polynomial $f(x_1,\ldots,x_n) \in \ring[x_1,\ldots,x_n]$.
We are interested in how many times $f(A_1,\ldots,A_n)$ assumes each value in $\quotring{k}$ as $(A_1,\ldots,A_n)$ runs through the $q^{n k}$ values of $\quotring{k}^n$.
To this end, we introduce the group $\Gamma_k=\{\gamma^A: A \in \quotring{k}\}$, with the multiplication rule $\gamma^A \gamma^B = \gamma^{A+B}$, so that $\Gamma_k$ is just a multiplicative version of the additive group of $\quotring{k}$.
Then we introduce the group $\C$-algebra\footnote{It would suffice for the purposes of this paper to define the group algebra over $\Q$, but sometimes it is useful to have the coefficients in $\C$ if we want to perform manipulations on expressions in the same manner that one factors polynomials over $\Q$ into factors over $\C$.} $\groupring{k}=\C[\gfgroup{k}]$, and we say that $F=\sum_{B \in \quotring{k}} F_B \gamma^B$ in $\groupring{k}$ is the {\it modulo $\pi^k$ generating function for $f$} if $f(A_1,\ldots,A_n)=B$ for precisely $q^{n k} F_B$ values of $(A_1,\ldots,A_n) \in (\quotring{k})^n$.
More explicitly, the modulo $\pi^k$ generating function for $f(x_1,\ldots,x_n)$ is $\frac{1}{q^{n k}} \sum_{(A_1,\ldots,A_n)\in (\quotring{k})^n} \gamma^{f(A_1,\ldots,A_n)}$.
Note that the coefficients of our generating function are rational numbers scaled so that $\sum_{B \in \quotring{k}} F_B=1$ rather than the more customary integer counts.
This has the advantage that if $f(x_1,\ldots,x_n)$ is conceived of as a polynomial in a larger polynomial ring $\ring[x_1,\ldots,x_{n^\prime}]$ where $n^\prime \geq n$, where the indeterminates $x_{n+1},\ldots,x_{n^\prime}$ do not appear in $f$, then the modulo $\pi^k$ generating function for $f$ remains unchanged.

The following easily verified rule is what makes generating functions useful.
\begin{remark}[Sum-Product Rule for Modular Generating Functions]\label{Alexander}
If $f(x_1,\ldots,x_n)$ and $g(y_1,\ldots,y_m)$ are two polynomials over $\ring$ involving distinct indeterminates $x_1,\ldots,x_n$, $y_1,\ldots,y_m$, and if $F$ and $G$ are their respective modulo $\pi^k$ generating functions, then the modulo $\pi^k$ generating function for $f\oplus g$ is $F G$.
\end{remark}

\subsection{$p$-Adic Generating Functions}\label{Lawrence}

When $j \leq k$, we have an epimorphism of $\C$-algebras $\phi_{k,j}\colon \groupring{k} \to \groupring{j}$ with $\phi(\gamma^B)=\gamma^C$, where $C$ is the unique element of $\quotring{j}$ (which is a coset of $\ncoset{j}$ in $\ring$) that contains $B$ (a coset of $\unif^k$ in $\ring$). 
Furthermore, if $f(x_1,\ldots,x_n) \in \ring[x_1,\ldots,x_n]$ and $G$ and $H$ are respectively the modulo $\pi^k$ and modulo $\pi^j$ generating functions for $f$, then one readily shows that $\phi_{k,j}(G)=H$.

We set $\limgroupring=\varprojlim\groupring{k}$, where for each $k \in \N$, we let $\phi_k\colon \limgroupring \to \groupring{k}$ be the $k$th coordinate map, an epimorphism of $\C$-algebras.
If $f(x_1,\ldots,x_n) \in \ring[x_1,\ldots,x_n]$, then the element $F \in \limgroupring$ such that $\phi_k(F)$ is the modulo $\pi^k$ generating function of $f$ for each $k \in \N$ is called the {\it $p$-adic generating function of $f$}.
The sum-product rule for modular generating functions (Remark \ref{Alexander}) immediately implies a sum-product rule for $p$-adic generating functions.
\begin{remark}[Sum-Product Rule for $p$-Adic Generating Functions]
Let $f(x_1,\ldots,x_n)$ and $g(y_1,\ldots,y_m)$ be polynomials over $\ring$ involving distinct indeterminates $x_1,\ldots,x_n,y_1,\ldots,y_m$.  If $F$ and $G$ are the $p$-adic respective generating functions for $f$ and $g$, then the $p$-adic generating function for $f\oplus g$ is $F G$.
\end{remark}
We now give an interpretation of the $p$-adic generating function $F$ for a polynomial $f(x_1,\ldots,x_n) \in \ring[x_1,\ldots,x_n]$.
Write
\begin{equation}\label{Madeline}
\phi_k(F) = \sum_{A \in \quotring{k}} F_A \gamma^A,
\end{equation}
for each $k \in \N$.
If ${\mathcal C}$ is a collection of pairwise disjoint cosets drawn from $\bigcup_{k \in \N} \quotring{k}$, and if $S=\bigcup_{C \in {\mathcal C}} C$, 
then $\sum_{C \in {\mathcal C}} F_C$ is the Haar volume of the preimage $f^{-1}(S)$ of $S$ by the polynomial $f$ (regarded as a function from $\ring^n$ to $\ring$).

\subsection{$p$-Adic Generating Functions for Constant and Linear Polynomials}

Consider the polynomial $f(x)=a x+b$ in a single variable $x$, where $a, b \in \ring$.
The modulo $\pi^k$ generating function for $f$ is
\[
\frac{1}{\#\{C \in \quotring{k}: C \cap (a R+b)\not=\emptyset\}} \sums{C \in \quotring{k} \\ C \cap (a R+b) \not=\emptyset} \gamma^C.
\]
We introduce the following notation: if $A$ is any coset of any $\quotring{j}$ (or if $A$ is a singleton set), we write $z^A$ for the $p$-adic generating function that has
\begin{equation}\label{Nancy}
\phi_k(z^A)=\frac{1}{\card{\{C \in \quotring{k}: C\cap A \not=\emptyset\}}} \sums{C \in \quotring{k} \\ C \cap A \not=\emptyset} \gamma^C
\end{equation}
for all $k \in \N$.
\begin{remark}\label{Otto}
The $p$-adic generating function of $f(x)=a x+b$ is $z^{a R+b}$.
\end{remark}
Note that $a R+b$ is the singleton set $\{b\}$ when $a=0$ (i.e., when $f(x)=a x+b$ is a constant polynomial).  We often write $z^b$ in place of $z^{\{b\}}$.
Recall our convention that $\unif^\infty=0$.
We can consider our singleton subsets of $\ring$ as elements of $\quotring{\infty}=\ring/\unif^\infty\ring$.
We sum cosets of the various $\quotring{k}$ in the usual group-theoretic sense: $A+B=\{a+b: a \in A, b \in B\}$.
So if $0 \leq j \leq k \leq \infty$, the sum of an element of $\quotring{j}$ and $\quotring{k}$ is an element of $\quotring{j}$.
The benefit of our notation $z^A$ is that it gives the following convenient arithmetic for $p$-adic generating functions for constant and linear polynomials, which is easy to check.
\begin{remark}
If $A \in \quotring{j}$ and $B \in \quotring{k}$, then $z^A z^B = z^{A+B}$ in $\limgroupring$.
\end{remark}
\begin{remark}[Coalescence]\label{Anne}
We note that if $A \in\quotring{j}$, $k\geq j$, and $W$ is a set of $q^{k-j}$ representatives for the cosets of $\quotring{k}$ that lie in $A$ (so $A=\bigsqcup_{w \in W} \coset{w}{k}$), then $\sum_{w \in W} \zcoset{w}{k} = q^{k-j} z^A$.  We say that we {\it coalesce cosets} when we simplify sums in this manner.
\end{remark}

\subsection{Relation to the Igusa Zeta Function}

Since $|a|_{\field}=q^{-\val(a)}$ for $a \in \field$, and since $t=q^{-s}$, the Igusa local zeta function \eqref{James} becomes 
\[
Z_f(t)=\int_{\ring^n} t^{\val(f(x_1,\ldots,x_n))} d x_1 \cdots d x_n.
\]
For $k \in \N$, let $U_k=\{(r_1,\ldots,r_n) \in \ring^n: \val(f(r_1,\ldots,r_n))=k\}$.
Then
\[
Z_f(t)=\sum_{k \in \N} t^k \Vol(U_k),
\]
where $\Vol(U)$ is the volume of $U$.
At the end of Section \ref{Lawrence}, we showed that if $F(z)$ is the $p$-adic generating function for $f$, and if the modulo $\pi^k$ generating function for $f$ is as given in \eqref{Madeline}, then
\[
\Vol(U_k)=F_{\ncoset{k}}-F_{\ncoset{k+1}}.
\]
We define a map $\Ig$ from $\limgroupring$ to the ring of power series $\C[[t]]$ as follows: if $F(z) \in\limgroupring$ with $\phi_k(F)$ as given in \eqref{Madeline}, then
\begin{equation}\label{Raphael}
\Ig(F(z))=\sum_{k \in \N} (F_{\ncoset{k}}-F_{\ncoset{k+1}}) t^k,
\end{equation}
so that $\Ig$ takes the $p$-adic generating function of $f$ to the Igusa zeta function of $f$.
$\Ig$ is clearly a $\C$-linear map because the maps $\phi_k$ are $\C$-linear.

\subsection{Igusa Local Zeta Functions for Constant and Linear Polynomials}

Now we can compute the local zeta functions for constant and linear polynomials from their $p$-adic generating functions.
\begin{lemma}\label{Priscilla}
Let $f(x)=a x+b$.  Then the Igusa local zeta function of $f$ is
\[
\begin{cases}
t^{\val(a)} \igrp & \text{if $\val(a) \leq \val(b)$,} \\
t^{\val(b)} & \text{if $\val(a) > \val(b)$,}
\end{cases}
\]
where $t^\infty$ is interpreted as $0$. 
\end{lemma}
\begin{proof}
The $p$-adic generating function for $f$ is $z^{a R+b}$ by Remark \ref{Otto}.
We calculate $\Ig(z^{a R+b})$ in Lemma \ref{Stevie} below.
\end{proof}
\begin{remark}
Of course, it is easy to calculate the zeta functions in Lemma \ref{Priscilla} directly from the definition \eqref{James}.
\end{remark}
\begin{lemma}\label{Stevie}
Let $j \in \Ni$.  If $A=\ncoset{j}$, then
\[
\Ig(z^A) = t^j \igrp,
\]
where $t^\infty$ is interpreted as $0$.
If $A=\coset{a}{j}$ where $a\not\in\ncoset{j}$, then
\[
\Ig(z^A) = t^{\val(a)}.
\]
\end{lemma}
\begin{proof}
Let $F(z)=z^A$.

First consider the case $A=\coset{a}{j}$ with $a\not\in \ncoset{j}$.  Consider \eqref{Nancy}, and note that if we write $\phi_k(F)$ as in \eqref{Madeline}, so that
\[
F_{\unif^k R}=\begin{cases} 
1 & \text{if $k \leq \val(a)$,} \\
0 & \text{if $k > \val(a)$.}
\end{cases}
\]
So
\[
F_{\unif^k \ring}-F_{\unif^{k+1} \ring} = \begin{cases}
1 & \text{if $k=\val(a)$,} \\
0 & \text{otherwise,}
\end{cases}
\]
and thus $\Ig(F(z))=t^{\val(a)}$ by \eqref{Raphael}.

Now suppose that $A=\ncoset{j}$.  Again, consider \eqref{Nancy} and write $\phi_k(F)$ as in \eqref{Madeline}, so that
\[
F_{\unif^k R}=\begin{cases} 
1 & \text{if $k \leq j$,} \\
\frac{1}{q^{k-j}} & \text{if $k > j$,}
\end{cases}
\]
and so for $k \in \N$, we have
\[
F_{\unif^k \ring}-F_{\unif^{k+1} \ring} = \begin{cases}
0 & \text{if $k<j$,} \\
\left(1-\frac{1}{q}\right) \frac{1}{q^{k-j}} & \text{if $k \geq j$,}
\end{cases}
\]
and thus \eqref{Raphael} gives
\[
\Ig(F(z))=\left(1-\frac{1}{q}\right) \sum_{k\geq j} \frac{t^k}{q^{k-j}}.\qedhere
\]
\end{proof}

\subsection{Normalization}

We often write an element $F$ of $\groupring{k}$ as $F(\gamma)$, inasmuch as it resembles a polynomial in $\gamma$.
If $F(\gamma)=\sum_{A \in \quotring{k}} F_A \gamma^A \in \groupring{k}$, we use the notation $F(1)$, which we call the {\it normalization of $F$}, to denote the sum of the coefficients, that is, $F(1)=\sum_{A \in \quotring{k}} F_A$.  If $F$ is a modular generating function, then $F(1)=1$.

Similarly, we often write an element $F$ of $\limgroupring$ as $F(z)$, especially if we are expressing it as a linear combination of terms of the form $z^A$.
By the nature of the homomorphisms $\phi_{k,j}$, we see that $(\phi_k(F))(1)$ is independent of $k$, and we let $F(1)$ denote the common value, which we call the {\it normalization of $F$}.
Indeed, if $F$ is expressed as a linear combination $F(z)=\sum_{C \in {\mathcal C}} F_C z^C$, then $F(1)=\sum_{C \in {\mathcal C}} F_C$.
If $F$ is a $p$-adic generating function, then $F(1)=1$.

We note that multiplying by the $p$-adic generating function $z^\ring$ for the linear polynomial $f(x)=x$ extracts the normalization of any other element of $\limgroupring$.
\begin{lemma}\label{Albert}
If $F(z)$ is an element of $\limgroupring$, then $F(z) z^\ring=F(1) z^\ring$.
\end{lemma}
\begin{proof}
For each $k \in \N$, we have $\phi_k(z^\ring)=\frac{1}{\card{\quotring{k}}} \sum_{A \in \quotring{k}} \gamma^A$, and write $\phi_k(F)=\sum_{B \in \quotring{k}} F_B \gamma^B$, so that $\phi_k(z^\ring F(z)) = \phi_k(z^\ring) \phi_k(F(z))$ becomes $\sum_{B \in \quotring{k}} F_B \cdot \frac{1}{\card{\quotring{k}}} \sum_{A \in \quotring{k}} \gamma^{A+B}$, which is $F(1) \phi_k(z^\ring)  = \phi_k(F(1) z^\ring)$.
\end{proof}

\subsection{Scaling}

Let $f$ be a polynomial and let $s \in R$.  If $k \leq \val(s)$, then the modulo $\pi^k$ generating function for $s f$ is $\gncoset{k}$.  Now suppose $k \geq \val(s)$.  If the modulo $\pi^{k-\val(s)}$ generating function for $f$ is
\[
\sum_{A \in \quotring{k-\val(s)}} F_A \gamma^A,
\]
then the modulo $\pi^k$ generating function for $s f$ is
\[
\sum_{A \in \quotring{k-\val(s)}} F_A \gamma^{s A}.
\]
Accordingly, we introduce the notation that if $F(z) \in \limgroupring$ with $\phi_k(F)$ as presented in \eqref{Madeline}, then $F(z^s)$ denotes the element of $\limgroupring$ given by
\begin{equation}\label{William}
\phi_k(F(z^s)) = \begin{cases}
F(1) \gncoset{k} & \text{for $k \leq \val(s)$,} \\
\sum_{A \in \quotring{k-\val(s)}} F_A \gamma^{s A}  & \text{for $k \geq \val(s)$.} \\
\end{cases}
\end{equation}
\begin{remark}
Note that if $F(z)=G(z) H(z)$, then $F(z^s)=G(z^s) H(z^s)$.
\end{remark}
\begin{remark}
If $F(z)$ is the $p$-adic generating function of $f$, then $F(z^s)$ is the $p$-adic generating function of $s f$.
\end{remark}
Now Lemma \ref{Albert} generalizes as follows.
\begin{remark}\label{Zachary}
If $j \in \Ni$ and $F(z)$ is an element of $\limgroupring$, then for $k \leq j$ we have $F(z^{\unif^j}) \zncoset{k}=F(1) \zncoset{k}$.
\end{remark}
We now show how scaling influences the Igusa local zeta function.
\begin{lemma}\label{Tina}
If $s \in \ring$ and $F(z) \in \limgroupring$, then $\Ig(F(z^s))=t^{\val(s)} \Ig(F(z))$, where $t^\infty$ is interpreted as $0$.
\end{lemma}
\begin{proof}
Let $G(z)=F(z^s)$, and for each $k \in \N$ write
\begin{align*}
\phi_k(F) & = \sum_{A \in \quotring{k}} F_A \gamma^A, \text{ and} \\
\phi_k(G) & = \sum_{A \in \quotring{k}} G_A \gamma^A.
\end{align*}

If $s=0$, then $G(z)=F(1) z^0$.
For each $k \in \N$, we have $G_{\ncoset{k}}=F(1)$, so then by \eqref{Raphael}, we have $\Ig(G(z))=0=t^\infty \Ig(F(z))=t^{\val(s)} \Ig(F(z))$.

Now suppose that $\val(s)=j < \infty$.
Then \eqref{William} shows us that
\[
G_{\unif^k \ring}= \begin{cases}
F(1) & \text{if $k \leq j$,} \\
F_{\unif^{k-j} \ring} & \text{if $k \geq j$,}
\end{cases}
\]
so then
\[
G_{\unif^k\ring}-G_{\unif^{k+1}\ring} = \begin{cases}
0 & \text{if $k < j$,} \\
F_{\unif^{k-j}\ring}-F_{\unif^{k-j+1}\ring} & \text{if $k \geq j$,}
\end{cases}
\]
and so by \eqref{Raphael}, we have $\Ig(G(z)) = \sum_{k \geq j} (F_{\unif^{k-j}\ring}-F_{\unif^{k-j+1}\ring}) t^k = t^j \Ig(F(z))$.
\end{proof}

\subsection{Uniformization}

For $j \in \N$, we say that $F(z) \in \limgroupring$ is {\it $\pi^j$-uniform} if it can be written as a $\C$-linear combination of terms of the form $z^A$ with $A \in \quotring{j}$.  If $F(z)$ is a generating function for some polynomial $f$, saying that $F$ is $\pi^j$-uniform is equivalent to saying that for any $A \in \quotring{j}$, and $k \geq j$, the polynomial $f \pmod{\pi^k}$ represents all $B \in \quotring{k}$ with $B\subseteq A$ equally often.  Thus we say that the polynomial $f$ is {\it $\pi^j$-uniform} if its $p$-adic generating function is $\pi^j$-uniform.
\begin{remark}
If $i \leq j$, then a $\pi^i$-uniform element of $\limgroupring$ is $\pi^j$-uniform.
\end{remark}
\begin{remark}\label{Bertram}
The product of a $\pi^i$-uniform element of $\limgroupring$ and a $\pi^j$-uniform element of $\limgroupring$ is a $\pi^{\min(i,j)}$-uniform element of $\limgroupring$.
\end{remark}

If $F(z)$ is an arbitrary element of $\limgroupring$ and $j \in \N$, then there is a unique $\pi^j$-uniform element $G(z) \in \limgroupring$ called the {\it $\pi^j$-uniformization of $F(z)$}, written $F(z) \pmod{\pi^j}$,  such that $\phi_j(F(z))=\phi_j(G(z))$.
If $k \geq j$ and $\phi_k(F(z))$ is written as in \eqref{Madeline}, then we must have
\begin{equation}\label{Reginald}
\phi_k(F(z) \!\!\!\! \pmod{\pi^j}) = \sum_{A \in \quotring{j}} \frac{F_A}{q^{k-j}} \sums{B \in\quotring{k} \\ B \subseteq A} z^B.
\end{equation}
This is what one gets if one applies $\phi_j$ to $F(z)$ to get an element $\sum_{A \in \quotring{j}} F_A \gamma^A \in \groupring{j}$, and then replaces each instance of $\gamma$ with $z$ to produce the $\pi^j$-uniform element $\sum_{A \in \quotring{j}} F_A z^A$ of $\limgroupring$.
\begin{remark}
If $i \leq j$, then the $\pi^i$-uniformization of the $\pi^j$-uniformization of an element $F(z) \in \limgroupring$ is just the $\pi^i$-uniformization of $F(z)$.
\end{remark}
\begin{remark}
An element of $\limgroupring$ is $\pi^j$-uniform if and only if it is equal to its own $\pi^j$-uniformization.
\end{remark}
Uniformizations simplify generating functions, and so they are useful in calculations via the following principle.
\begin{lemma}\label{Una}
Suppose that $j \geq i$, and that $F(z), G(z) \in \limgroupring$ and $F(z)$ is $\pi^i$-uniform.  Then $F(z) G(z) = F(z) (G(z) \pmod{\unif^j})$, and $F(z) G(z)$ is $\pi^i$-uniform.
\end{lemma}
\begin{proof}
By linearity, it suffices to show this when $F(z)=z^A$ for a single coset $A\in\quotring{i}$.  Let $k \in \N$ with $k \geq j$ be given.  Write $\phi_k(G(z))=\sum_{B \in \quotring{k}} G_B \gamma^B$, and use \eqref{Nancy} for $\phi_k(z^A)$ to get
\begin{align*}
\phi_k(z^A) \phi_k(G(z)) & = \frac{1}{q^{k-i}} \sums{B, C \in \quotring{k} \\ C \subseteq A} G_B \gamma^{B+C} = \frac{1}{q^{k-i}} \sum_{D \in \quotring{j}} \sums{B,C \in \quotring{k} \\ B \subseteq D \\ C \subseteq A} G_B \gamma^{B+C} \\
& = \frac{1}{q^{k-i}} \sum_{D\in\quotring{j}} \sums{B,E \in \quotring{k} \\ B \subseteq D \\ E \subseteq A+D} G_B \gamma^E = \frac{1}{q^{k-i}} \sum_{D \in \quotring{j}} G_D \sums{E \in \quotring{k} \\ E \subseteq A+D} \gamma^E
\\ & = \frac{1}{q^{k-i}} \sum_{D\in\quotring{j}} \frac{G_D}{q^{k-j}} \sums{B,C \in \quotring{k} \\ B\subseteq D \\ C \subseteq A} \gamma^{B+C} =\phi_k(z^A) \phi_k(G(z)\!\!\!\! \pmod{\unif^j}),
\end{align*}
where we use \eqref{Reginald} to recognize $\phi_k(G(z) \pmod{\unif^j})$ in the final step.  That $z^A (G(z) \pmod{\unif^j})$ is $\unif^i$-uniform follows from Remark \ref{Bertram}.
\end{proof}
\begin{corollary}\label{Philip}
Let $i \in \N$, and suppose that $F(z) \in\limgroupring$ is $\pi^i$-uniform.
If $j \geq i$ and if $G(z) \in \limgroupring$, then $F(z) G(z^{\pi^j})=F(z) G(1)$.
\end{corollary}
\begin{proof}
Apply Lemma \ref{Una} and note that $G(z^{\unif^j}) \pmod{\unif^j}=G(1) z^{\unif^j}$.
\end{proof}

\subsection{Partial Generating Functions, Heads, and Homogeneous Polynomials}

For $i \in \N$, a subset of $\ring^n$ is said to be {\it regular modulo $\pi^i$}, or just {\it $\pi^i$-regular}, if it is a union of sets of the form $A_1\times\cdots\times A_n$ where each $A_h$ is a coset of the form $\coset{a_h}{i}$ for $a_h \in \ring$.

Let $f(x_1,\ldots,x_n) \in \ring[x_1,\ldots,x_n]$ be homogeneous of degree $d$.
For each $k \in \N$, let $G_k(\gamma)$ be the modulo $\pi^k$ generating function for $f$, and let $G(z)$ be the $p$-adic generating function of $f$.
Recall that
\[
G_k(\gamma)=\frac{1}{q^{n k}} \sum_{(A_1,\ldots,A_n)\in (\quotring{k})^n} \gamma^{f(A_1,\ldots,A_n)}.
\]

If $S$ is a regular subset of $R^n$ modulo $\pi^i$ and $k \geq i$, then the {\it partial modulo $\pi^k$ generating function of $f$ on $S$} is defined to be
\[
\frac{1}{q^{n k}} \sums{(A_1,\ldots,A_n) \in (\quotring{k})^n \\ A_1 \times \cdots\times A_n \subseteq S} \gamma^{f(A_1,\ldots,A_n)}.
\]
Note that the $\pi^i$-regularity makes each $(A_1,\ldots,A_n) \in (\quotring{k})^n$ either contained in $S$ or disjoint from $S$.

For each $k \geq i$, we let $H_k(\gamma)$ denote the partial modulo $\pi^k$ generating function of $f$ on $S$.
It is straightforward to see that if $k \geq j \geq i$, then the standard projection map $\phi_{k,j}\colon \groupring{k} \to \groupring{j}$ carries $H_k$ to $H_j$.
Thus there is an element $H$ of $\limgroupring$ such that for each $k \geq i$, the projection $\phi_k(H)$ in $\groupring{k}$ is the partial modulo $\pi^k$ generating function of $f$ on $S$.
We call this $H$ the {\it partial $p$-adic generating function of $f$ on $S$}.

The partial modulo $\pi^k$ generating functions (for $k \geq 1$) and the partial $p$-adic generating function for $f$ associated to the $\pi$-regular region $S=\ring^n\smallsetminus (\unif\ring)^n$ are of special interest, and are called the {\it head of the modulo $\pi^k$ generating function of $f$} and the {\it head of the $p$-adic generating function of $f$}, respectively.  The head of the modulo $\pi^k$ generating function of $f$ is then
\[
\frac{1}{q^{n k}} \sums{(A_1,\ldots,A_n)\in (\quotring{k})^n \\ (A_1,\ldots,A_n) \not\equiv (0,\ldots,0) \!\!\!\! \pmod{\pi}} \gamma^{f(A_1,\ldots,A_n)}.
\]
This keeps track of how many times $f$ represents a given value modulo $\pi^k$ when we restrict the inputs so that they cannot all simultaneously be nonunits.
\begin{lemma}\label{Heinrich}
Suppose that there is a degree $d$ homogeneous polynomial $f(x_1,\ldots,x_n) \in \ring[x_1,\ldots,x_n]$, whose $p$-adic generating function $G(z)$ has head $H(z)$.  Then $G(z)=H(z)+\frac{1}{q^n} G(z^{\pi^d})$, and the Igusa local zeta function of $f$ is $\Ig(G(z))= \Ig(H(z))/(1-t^d/q^n)$.
\end{lemma}
\begin{proof}
Consider the behavior of $f$ for inputs $(A_1,\ldots,A_n) \in (\quotring{k})^n$ such that $(A_1,\ldots,A_n) \equiv (0,\ldots,0) \pmod{\pi}$.
To see how many times $f(x_1,\ldots,x_n)$ represents elements in $\quotring{k}$ for such inputs, one could reparameterize $(x_1,\ldots,x_n)=\pi(y_1,\ldots,y_n)$, and then note that $f(\pi y_1,\ldots,\pi y_n)=\pi^d f(y_1,\ldots,y_n)$ by homogeneity.
If $f(y_1,\ldots,y_n)$ represents a particular element $B$ of $\quotring{k-d}$ for $N$ distinct inputs $(C_1,\ldots,C_n) \in \quotring{k-d}$, then $f(\pi y_1,\ldots,\pi y_n)$ represents $\pi^d B \in \quotring{k}$ for $N$ distinct inputs $(C_1,\ldots,C_n)$ of $\quotring{k-d}$, and so represents $\pi^d B \in \quotring{k}$ for $q^{n(d-1)} N$ distinct inputs $(D_1,\ldots,D_n) \in \quotring{k-1}$, so that $f(x_1,\ldots,x_n)$ represents $\pi^d B \in \quotring{k}$ for $q^{n(d-1)} N$ distinct inputs $(A_1,\ldots,A_n) \in \quotring{k}$ that vanish modulo $\pi$.
Keeping in mind our normalization, this makes a contribution of $q^{-n(k-d+1)} N \gamma^{\pi^d B}$ to the modulo $\pi^k$ generating function, and the coefficient for this contribution is $q^{-n}$ times the coefficient for $\gamma^B$ in the modulo $\pi^{k-d}$ generating function of $f$.
Thus when we account for the parts of the modulo $\pi^k$ generating function excluded from the head, we see that
\[
G_k(\gamma)-H_k(\gamma)=\frac{1}{q^n} G_{k-d}(\gamma^{\pi^d}),
\]
where $G_{k-d}(\gamma^{\pi^d})$ indicates the function obtained from $G_{k-d}(\gamma)$ by replacing each instance of $\gamma^A$ (with $A \in \quotring{k-d}$) with $\gamma^{\pi^d A}$ (in $\quotring{k}$).
If we take the inverse limit of both sides of our equation and recall the definition of $G(z^s)$ for $s \in \ring$, we obtain $G(z) - H(z) =\frac{1}{q^n} G(z^{\pi^d})$.  Apply $\Ig$ to both sides, and apply Lemma \ref{Tina} and rearrange to obtain $(1-t^d/q^n) \Ig(G(z))=\Ig(H(z))$.
\end{proof}
\begin{remark}
The relation between $\Ig(G(z))$ and $\Ig(H(z))$ in Lemma \ref{Heinrich} is well-known and can be derived directly from the definition of the zeta function via an elementary calculation.  For example, see \cite[eq.~(1)]{Denef-Meuser}.
\end{remark}

\subsection{Hensel's Lemma}

If $f(x_1,\ldots,x_n) \in \ring[x_1,\ldots,x_n]$ and $a=(a_1,\ldots,a_n) \in \ring^n$, then the {\it derivative of $f$ at $a$} is the vector of partial derivatives $(\partial f/\partial x_1,\ldots,\partial f/\partial x_n)$ evaluated at $(a_1,\ldots,a_n)$.  The {\it valuation} of the derivative of $f$ at $(a_1,\ldots,a_n)$ is the least valuation of the $n$ partial derivatives at $a$.
We now state a form of Hensel's Lemma in terms of our theory of $p$-adic generating functions: if the valuation of the derivative is constant on a certain region, then the partial $p$-adic generating function on that region is uniform.
\begin{lemma}\label{Gretchen}
Let $f(x_1,\ldots,x_n) \in \ring[x_1,\ldots,x_n]$, and let $j \in \N$.
Let $S$ be a $\pi^{j+1}$-regular subset of $\ring^n$ upon which the derivative of $f$ always has valuation $j$.
Then the partial $p$-adic generating function $H(z)$ of $f$ on $S$ is $\pi^{2 j+1}$-uniform, so it may be obtained by replacing instances of $\gamma$ with $z$ in the partial modulo $\pi^{2 j+1}$ generating function of $f$ on $S$.
If $W_{j+1}$ is a set of representatives for the cosets of $\pi^{j+1} R$ in $R$, then
\[
H(z)=\frac{1}{q^{n (j+1)}} \sum_{(a_1,\ldots,a_n) \in (W_{j+1})^n \cap S} z^{\coset{f(a_1,\ldots,a_n)}{2 j+1}}.
\]
\end{lemma}
\begin{proof}
Fix $a=(a_1,\ldots,a_n) \in W_{j+1}^n \cap S$.
For any $b=(b_1,\ldots,b_n) \in R^n$, use the Taylor expansion of $f$ around $a$ to see that the first-order term of $f(a+\pi^{j+1} b)-f(a)$ has valuation at least $2 j+1$, with equality when $\pi^{j+1}$ does not divide the inner product of $b$ with the derivative of $f$ at $a$.
The second-order and higher-order terms of the Taylor expansion all have valuation at least $2 j+2$, so $f(a+\pi^{j+1} b)-f(a)$ is linear in $b$ modulo $\pi^{2 j+2}$.
Then one shows that for each $c \in \ring$ with $c \equiv f(a_1,\ldots,a_n) \pmod{\pi^{2 j+1}}$, there exist $\widetilde{a} \in \ring^n$ such that $f(\widetilde{a}) \equiv c \pmod{\pi^{2 j+2}}$ and $\widetilde{a}\equiv a \pmod{\pi^{j+1}}$, and the number of such $\widetilde{a}$ that are distinct modulo $\pi^{j+2}$ is precisely $q^{n-1}$ by linear algebra.

Thus by iterative amelioration of solutions, one shows that for each $k \geq 2 j+1$ and each $c \in \ring$ with $c \equiv f(a_1,\ldots,a_n) \pmod{\pi^{2 j+1}}$, there exist $\widetilde{a} \in \ring^n$ such that $f(\widetilde{a}) \equiv c \pmod{\pi^k}$ and $\widetilde{a}\equiv a \pmod{\pi^{j+1}}$, and the number of such $\widetilde{a}$ that are distinct modulo $\pi^{k-j}$ is precisely $q^{(n-1)(k-(2j+1))}$.

If we are thinking of the partial modulo $\pi^k$ generating function for $f$ on $S$, then the number of $(A_1,\ldots,A_n) \in (\quotring{k})^n$ with every $A_i \subseteq \coset{a_i}{j+1}$ and such that $f(A_1,\ldots,A_n)=\coset{c}{k}$ is equal to $q^{(n-1)(k-(2 j+1))} \cdot q^{n j}=q^{n(k-j-1)-(k-(2 j+1))}$.

As we let $a$ run through $(W_{j+1})^n\cap S$, we represent all the $n$-fold products of cosets of $\pi^{j+1}\ring$ that make up $S$, and so the partial modulo $\pi^k$ generating function of $f$ on $S$ is
\begin{align*}
H_k(\gamma)
& = \frac{1}{q^{n k}} \sum_{(a_1,\ldots,a_n) \in (W_{j+1})^n\cap S} \sums{\coset{c}{k} \in \quotring{k} \\ \coset{c}{k} \subseteq \coset{f(a_1,\ldots,a_n)}{2 j+1}} q^{n(k-j-1)-(k-(2 j+1))} \gamma^{\coset{c}{k}} \\
& = \frac{1}{q^{n (j+1)}} \sum_{(a_1,\ldots,a_n) \in (W_{j+1})^n\cap S} \sums{\coset{c}{k} \in \quotring{k} \\ \coset{c}{k} \subseteq \coset{f(a_1,\ldots,a_n)}{2 j+1}} q^{-(k-(2 j+1))} \gamma^{\coset{c}{k}},
\end{align*}
which by coalescence of cosets (see Remark \ref{Anne}) shows that the partial $p$-adic generating function of $f$ on $S$ is
\[
H(z)= \frac{1}{q^{n (j+1)}} \sum_{(a_1,\ldots,a_n) \in (W_{j+1})^n\cap S} z^{\coset{f(a_1,\ldots,a_n)}{2 j+1}}.\qedhere
\]
\end{proof}
This has a very useful application to unimodular quadratic forms that we use later.
\begin{corollary}\label{Hortense}
Let $\val(2)=\ell$, and let $W$ be a set of $q^{\ell+1}$ representatives for the cosets of $\pi^{\ell+1}\ring$ in $\ring$ .
Let $Q(x_1,\ldots,x_n)$ be a unimodular quadratic form of rank $n$ over $\ring$.
Then the head of the $p$-adic generating function of $Q$ is
\[
H(z)=\frac{1}{q^{n(\ell+1)}} \sums{(a_1,\ldots,a_n) \in W^n \\ (a_1,\ldots,a_n)\not\equiv (0,\ldots,0)\!\!\!\! \pmod{\pi}} z^{\coset{Q(a_1,\ldots,a_n)}{2 \ell+1}}.
\]
\end{corollary}
\begin{proof}
If $M$ is the matrix associated to $Q$, then the derivative of $Q$ at $(a_1,\ldots,a_n)$ is $2 M (a_1,\ldots,a_n)^T$.  If $(a_1,\ldots,a_n) \in \ring^n\smallsetminus (\pi\ring)^n$, then since $M$ is unimodular, we see that the derivative has valuation $\val(2)=\ell$, so we may apply Lemma \ref{Gretchen}.
\end{proof}

\section{Quadratic Polynomials over $p$-Adic Fields}\label{Elizabeth}

This section assumes the basic facts about quadratic forms related in Section \ref{Martha}, one of which is that every quadratic form over $\ring$ is equivalent to $\bigoplus_{i=0}^\infty \pi^i Q_i$ for some unimodular quadratic forms $Q_0, Q_1,\ldots$, almost all of which are zero.  We first show how the $p$-adic generating function of $Q=\bigoplus_{i=0}^\infty \pi^i Q_i$ relates to those of the constituent unimodular forms $Q_i$.

\subsection{Generating Function for an Arbitrary Quadratic Form}

\begin{proposition}\label{Katherine}
Suppose $\val(2)=\ell$.
Consider the quadratic form $Q=\bigoplus_{i\in\N} \unif^i Q_i$, where each $Q_i$ is a unimodular quadratic form of rank $r_i$ over $\ring$, and let $\omega$ be a positive integer such that $Q_i=0$ for $i > \omega$.
For each $j \in \N$, we let
\[
Q_{(j)}  = \!\!\!\!\!\!\!\! \bigopluss{0 \leq i \leq j \\ i \equiv j \pmod{2}} Q_i,  \quad\quad\quad r_{(j)} =\rank(Q_{(j)})=\!\!\!\!\!\!\!\! \sums{0 \leq i \leq j \\ i \equiv j \pmod{2}} \!\!\!\!\! r_i, \quad\quad\quad q_{(j)}  = q^{\, \sum_{0 \leq i < j} r_{(i)}},
\]
and for any quadratic form $P$, use the term $G_P$ to denote the $p$-adic generating function of $P$, and $H_P$ for the head of $G_P$.

Then the $p$-adic generating function $G_Q(z)$ of $Q$ is 
\[
\sum_{0 \leq i < \omega-1} \frac{1}{q_{(i)}} H_{Q_{(i)}}(z^{\unif^i}) G_{Q_{(i+1)}}(z^{\unif^{i+1}}) \prod_{2 \leq j \leq 2 \ell} G_{Q_{i+j}}(z^{\unif^{i+j}}) + \frac{1}{q_{(\omega-1)}} G_{Q_{(\omega-1)}}(z^{\unif^{\omega-1}}) G_{Q_{(\omega)}}(z^{\unif^{\omega}}).
\]
\end{proposition}
\begin{proof}
We proceed by induction on $\omega$.  The $\omega=1$ case is just an application of the sum-product rule and scaling: the generating function of $Q_0\oplus \unif Q_1$ is the product of $G_{Q_0}(z)$ and $G_{Q_1}(z^\unif)$.

If $\omega > 1$, then $Q$ is the direct sum of the quadratic forms $\widetilde{Q}=\bigoplus_{i=0}^{\omega-1} \unif^i Q_i$ and $\unif^\omega Q_\omega$.  Thus by the sum-product rule, the generating function $G_Q(z)$ of $Q$ is the product of the generating function $G_{\widetilde{Q}}(z)$ of $\widetilde{Q}$ and the generating function $G_{Q_{\omega}}(z^{\unif^\omega})$ of $\unif^\omega Q_\omega$.
For each $i \in \N$, we let
\[
\widetilde{Q}_i = \begin{cases}
Q_i & \text{for $i\not=\omega$,} \\
0 & \text{for $i=\omega$,}
\end{cases}
\quad\quad\quad\quad
\widetilde{r}_i= \rank(\widetilde{Q}_i) = \begin{cases}
r_i & \text{for $i\not=\omega$,} \\
0 & \text{for $i=\omega$,}
\end{cases}
\]
and then for $j \in \N$, we set
\[
\widetilde{Q}_{(j)}  = \!\!\!\!\!\!\!\! \bigopluss{0 \leq i \leq j \\ i \equiv j \!\!\!\! \pmod{2}} \widetilde{Q}_i,  \quad\quad\quad \widetilde{r}_{(j)} =\rank(\widetilde{Q}_{(j)})=\!\!\!\!\!\!\!\! \sums{0 \leq i \leq j \\ i \equiv j \!\!\!\! \pmod{2}} \!\!\!\!\! \widetilde{r}_i, \quad\quad\quad \widetilde{q}_{(j)}  = q^{\,\sum_{0 \leq i < j} \widetilde{r}_{(i)}},
\]
so that induction shows that $G_{\widetilde{Q}}(z)$ is
\[
\sum_{0 \leq i < \omega-1} \frac{1}{\widetilde{q}_{(i)}} H_{\widetilde{Q}_{(i)}}(z^{\unif^i}) G_{\widetilde{Q}_{(i+1)}}(z^{\unif^{i+1}}) \prod_{2 \leq j \leq 2 \ell} G_{\widetilde{Q}_{i+j}}(z^{\unif^{i+j}}) + \frac{1}{\widetilde{q}_{(\omega-1)}} G_{\widetilde{Q}_{(\omega-1)}}(z^{\unif^{\omega-1}}) G_{\widetilde{Q}_{(\omega)}}(z^{\unif^{\omega}}).
\]
Now note that
\[
Q_{(i)} =\begin{cases}
\widetilde{Q}_{(i)} & \text{for $i < \omega$,} \\
\widetilde{Q}_{(i)} + Q_\omega & \text{for $i=\omega$,}
\end{cases}
\quad\,
r_{(i)} = \begin{cases}
\widetilde{r}_{(i)} & \text{for $i < \omega$,} \\
\widetilde{r}_{(i)}+r_\omega & \text{for $i=\omega$,}
\end{cases}
\quad\text{and}\quad
\widetilde{q}_{(i)} = q_{(i)} \text{ for $i \leq \omega$},
\]
so that $G_{\widetilde{Q}}(z)$ is
\[
\sum_{0 \leq i < \omega-1} \frac{1}{q_{(i)}} H_{Q_{(i)}}(z^{\unif^i}) G_{Q_{(i+1)}}(z^{\unif^{i+1}}) \prods{2 \leq j \leq 2 \ell \\ j\not=\omega-i} G_{Q_{i+j}}(z^{\unif^{i+j}}) \\
+ \frac{1}{q_{(\omega-1)}} G_{Q_{(\omega-1)}}(z^{\unif^{\omega-1}}) G_{\widetilde{Q}_{(\omega)}}(z^{\unif^{\omega}}).
\]
Recall that the generating function $G_Q(z)$ of $Q$ is the product of $G_{\widetilde{Q}}(z)$ and $G_{Q_{\omega}}(z^{\unif^\omega})$.
By the sum-product rule we know that
\[
G_{\widetilde{Q}_{(\omega)}}(z) G_{Q_\omega}(z)=G_{Q_{(\omega)}}(z),
\]
because $Q_{(\omega)}=\widetilde{Q}_{(\omega)} \oplus Q_\omega$.
When we multiply  $G_{\widetilde{Q}}(z)$ and $G_{Q_\omega}(z^{\unif^\omega})$ and use this principle, we get
\begin{multline*}
G_Q(z)=\sum_{0 \leq i < \omega-1} \frac{1}{q_{(i)}} H_{Q_{(i)}}(z^{\unif^i}) G_{Q_{(i+1)}}(z^{\unif^{i+1}}) G_{Q_{\omega}}(z^{\unif^\omega}) \prods{2 \leq j \leq 2 \ell \\ j\not=\omega-i} G_{Q_{i+j}}(z^{\unif^{i+j}}) \\
+ \frac{1}{q_{(\omega-1)}} G_{Q_{(\omega-1)}}(z^{\unif^{\omega-1}}) G_{Q_{(\omega)}}(z^{\unif^{\omega}}).
\end{multline*}
Now look at the first sum in the last expression.
We note that Corollary \ref{Hortense} shows that $H_{Q_{(i)}}(z^{\unif^i})$ is $\pi^{2 \ell+1+i}$-uniform, so that if $i < \omega-2\ell$, Corollary \ref{Philip} then shows that $H_{Q_{(i)}}(z^{\unif^i}) G_{Q_\omega}(z^{\pi^\omega})=H_{Q_{(i)}}(z^{\unif^i})$, so that we can drop the $G_{Q_\omega}(z^{\pi^\omega})$ term.  For $i < \omega-2\ell$, we can also drop the restriction $j\not=\omega-i$ in the product, since it has no effect.  On the other hand, when $\omega-2\ell \leq i < \omega-1$ in the first sum, the $G_{Q_\omega}(z^{\unif^\omega})$ term supplies precisely the term that the $j\not=\omega-i$ restriction in the product omits.  Thus we obtain precisely the expression for $G_Q(z)$ that we were to prove.
\end{proof}
The following result will allow us to apply $\Ig$ to the last term in the expression for the generating function given in the above proposition.
\begin{lemma}\label{Paul}
Let $Q_0$ and $Q_1$ be quadratic forms of ranks $r_0$ and $r_1$.  For each quadratic form $P$, let $G_P$ and $H_P$ be respectively the $p$-adic generating function and the head of the $p$-adic generating function of $P$.
Then
\[
G_{Q_0\oplus \pi Q_1}(z) = H_{Q_0}(z) G_{Q_1}(z^\unif)+\frac{1}{q^{r_0}} G_{Q_1}(z^{\unif}) G_{Q_0}(z^{\unif^2}),
\]
and
\[
G_{Q_0\oplus \pi Q_1}(z) = H_{Q_0}(z) G_{Q_1}(z^\unif)+\frac{1}{q^{r_0}} H_{Q_1}(z^\unif)G_{Q_0}(z^{\unif^2})+\frac{1}{q^{r_0 + r_1}}G_{Q_0\oplus\pi Q_1}(z^{\unif^2}).
\]
The Igusa local zeta function $\Ig(G_{Q_0\oplus \pi Q_1}(z))$ is then
\[
\left(\Ig(H_{Q_0}(z) G_{Q_1}(z^\unif))+\frac{t}{q^{r_0}} \Ig(H_{Q_1}(z)G_{Q_0}(z^{\unif}))\right) \left(1-\frac{t^2}{q^{r_0+r_1}}\right)^{-1}.
\]
\end{lemma}
\begin{proof}
From Lemma \ref{Heinrich}, we have $G_{Q_0}(z) =H_{Q_0}(z)+ \frac{1}{q^{r_0}} G_{Q_0}(z^{\unif^2})$, and we use the sum-product rule to multiply this by the generating function $G_{Q_1}(z^\unif)$ of $\pi Q_1$ to obtain
\begin{equation}\label{Quentin}
G_{Q_0\oplus \pi Q_1}(z) = H_{Q_0}(z)G_{Q_1}(z^{\unif})+ \frac{1}{q^{r_0}} G_{Q_1}(z^\unif) G_{Q_0}(z^{\unif^2})
\end{equation}
the first relation we were to prove.  Then we use Lemma \ref{Heinrich} again to obtain $G_{Q_1}(z^\unif) =H_{Q_1}(z^\unif)+ \frac{1}{q^{r_1}} G_{Q_1}(z^{\unif^3})$, and substitute this into \eqref{Quentin} to obtain the second relation we were to prove.  When one applies $\Ig$ to both sides of the second relation, and uses Lemma \ref{Tina}, one obtains
\[
\Ig(G_{Q_0\oplus \pi Q_1}(z)) = \Ig(H_{Q_0}(z) G_{Q_1}(z^\unif)) +\frac{1}{q^{r_0}} \Ig(H_{Q_1}(z^\unif)G_{Q_0}(z^{\unif^2}))+\frac{t^2}{q^{r_0 + r_1}} \Ig(G_{Q_0\oplus\pi Q_1}(z)),
\]
whence one obtains the expression that was claimed for the Igusa local zeta function.
\end{proof}

\subsection{Generating Function for Quadratic Polynomials}

In this section, we consider what happens when we add linear and constant terms to a quadratic form (as defined in previous section) over a $p$-adic field.
We shall consider polynomials of the form
\begin{equation}\label{Boris}
\bigoplus_{i=0}^\infty \unif^i Q_i \oplus L + c,
\end{equation}
where each $Q_i$ is a unimodular quadratic form, and almost all $Q_i$ are zero, $L$ is a linear form with at most one variable ($L=0$ if zero variables), and $c \in \ring$ is a constant.
In Section \ref{Armand} below, we show that when $p$ is odd or when our ring $\ring$ is $\Z_2$, then any quadratic polynomial over $\ring$ is strongly isospectral to a polynomial of the form \eqref{Boris}.

\begin{theorem}\label{Victor}
Suppose $\val(2)=\ell$.
Consider the quadratic polynomial $Q=\bigoplus_{i\in\N} \unif^i Q_i \oplus L + c$, where each $Q_i$ is a unimodular quadratic form of rank $r_i$ over $\ring$, where there is a positive integer $\omega$ such that $Q_i=0$ for $i > \omega$, where $L$ is a linear form involving at most one indeterminate, and $c \in \ring$ is a constant.
For each $j \in \N$, we let
\[
Q_{(j)}  = \!\!\!\!\!\!\!\! \bigopluss{0 \leq i \leq j \\ i \equiv j \!\!\!\! \pmod{2}} Q_i,  \quad\quad\quad r_{(j)} =\rank(Q_{(j)})=\!\!\!\!\!\!\!\! \sums{0 \leq i \leq j \\ i \equiv j \!\!\!\! \pmod{2}} \!\!\!\!\! r_i, \quad\quad\quad q_{(j)}  = q^{\, \sum_{0 \leq i < j} r_{(i)}},
\]
and for any quadratic form $P$, use the term $G_P$ to denote the $p$-adic generating function of $P$, and $H_P$ for the head of $G_P$.
\begin{itemize}
\item If $L=0$, then the $p$-adic generating function of $Q$ is
\begin{multline*}
\quad\quad G_Q(z)=\sum_{0 \leq i < \omega-1} \frac{z^c}{q_{(i)}} H_{Q_{(i)}}(z^{\unif^i}) G_{Q_{(i+1)}}(z^{\unif^{i+1}}) \prod_{2 \leq j \leq 2 \ell} G_{Q_{i+j}}(z^{\unif^{i+j}}) \\
+ \frac{z^c}{q_{(\omega-1)}} G_{Q_{(\omega-1)}}(z^{\unif^{\omega-1}}) G_{Q_{(\omega)}}(z^{\unif^{\omega}}).
\end{multline*}
\item If $L=a x$ with $\val(a)=\lambda < \infty$, then the $p$-adic generating function of $Q$ is
\begin{multline*}
\quad\quad G_Q(z)=\sum_{0 \leq i < \lambda-2\ell} \frac{z^c}{q_{(i)}} H_{Q_{(i)}}(z^{\unif^i}) G_{Q_{(i+1)}}(z^{\unif^{i+1}}) \prod_{2 \leq j \leq 2 \ell} G_{Q_{i+j}}(z^{\unif^{i+j}}) \\
+ \sum_{\max\{0,\lambda-2\ell\} \leq i < \lambda} \frac{\zcoset{c}{\lambda}}{q_{(i)}} H_{Q_{(i)}}(z^{\unif^i}) G_{Q_{(i+1)}}(z^{\unif^{i+1}}) \prod_{2 \leq j \leq 2 \ell} G_{Q_{i+j}}(z^{\unif^{i+j}}) + \frac{\zcoset{c}{\lambda}}{q_{(\lambda)}}.
\end{multline*}
\end{itemize}
\end{theorem}
\begin{proof}
We may assume $c=0$: the $p$-adic generating function of $c$ is $z^c$ by Remark \ref{Otto}, and the general case follows from the $c=0$ case by the sum-product rule.  The $L=0$ case is Proposition \ref{Katherine}.

So we suppose $L=a x$ with $\val(a)=\lambda < \infty$ henceforth.
By the sum-product rule, the $p$-adic generating function we seek is the product of the $p$-adic generating function of $\bigoplus_{i\in\N} \unif^i Q_i$ furnished by Proposition \ref{Katherine}, and the $p$-adic generating function of $L(x)=a x$, which is $\zncoset{\lambda}$, as shown in Remark \ref{Otto}.
When using Proposition \ref{Katherine}, we make sure to use $\omega$ large enough that $\omega > \lambda$.
When we multiply $\zncoset{\lambda}$ with terms from Proposition \ref{Katherine} of the form  $H_{Q_{(i)}}(z^{\unif^i}) G_{Q_{(i+1)}}(z^{\unif^{i+1}}) \prod_{2 \leq j \leq 2 \ell} G_{Q_{i+j}}(z^{\unif^{i+j}})$ with $i \geq \lambda$, we get $H_{Q_{(i)}}(1)\zncoset{\lambda}$ by Remark \ref{Zachary}, which equals $(1-1/q^{r_{(i)}})\zncoset{\lambda}$ because the head of a generating function for a form with $r_{(i)}$ variables records the form's values on a set of Haar volume $1-1/q^{r_{(i)}}$.
And when we multiply $\zncoset{\lambda}$ with the term $G_{Q_{(\omega-1)}}(z^{\unif^{\omega-1}}) G_{Q_{(\omega)}}(z^{\unif^{\omega}})$, we simply get $\zncoset{\lambda}$.
Thus the generating function of $Q$ is
\begin{multline*}
G_Q(z)=\sum_{0 \leq i < \lambda} \frac{\zncoset{\lambda}}{q_{(i)}} H_{Q_{(i)}}(z^{\unif^i}) G_{Q_{(i+1)}}(z^{\unif^{i+1}}) \prod_{2 \leq j \leq 2 \ell} G_{Q_{i+j}}(z^{\unif^{i+j}}) \\
+\sum_{\lambda \leq i < \omega-1} \frac{\zncoset{\lambda}}{q_{(i)}} \left(1-\frac{1}{q^{r_{(i)}}}\right)+ \frac{\zncoset{\lambda}}{q_{(\omega-1)}},
\end{multline*}
and then use the definition of $q_{(i)}$ and $r_{(i)}$ to see that the second sum and the final term telescope to give $\zncoset{\lambda}/q_{(\lambda)}$.

We then note that Corollary \ref{Hortense} shows that $H_{Q_{(i)}}(z^{\unif^i})$ is $\pi^{2 \ell+1+i}$-uniform, so that if $i < \lambda-2\ell$, Remark \ref{Zachary} shows that $H_{Q_{(i)}}(z^{\unif^i}) \zncoset{\lambda}=H_{Q_{(i)}}(z^{\unif^i})$, so that we can drop the $\zncoset{\lambda}$ term in these cases.
\end{proof}
\begin{corollary}\label{Matilda}
Let $p$ be odd.
Consider the quadratic polynomial $Q=\bigoplus_{i\in\N} \unif^i Q_i \oplus L + c$, where each $Q_i$ is a unimodular quadratic form of rank $r_i$ and discriminant $d_i$ over $\ring$, where there is a positive integer $\omega$ such that $Q_i=0$ for $i > \omega$, where $L$ is a linear form involving at most one indeterminate, and $c \in \ring$ is a constant.
For each $j \in \N$, we let
\begin{center}
\begin{tabular}{lll}
 & & \\
$Q_{(j)} = \!\!\!\!\! \bigoplusls{0 \leq i \leq j \\ i \equiv j \pmod{2}} \!\!\!\!\! Q_i,$ & & $d_{(j)} =\disc(Q_{(j)})=\!\!\!\!\! \prodls{0 \leq i \leq j \\ i \equiv j \pmod{2}} \!\!\!\!\! d_i$, \\
 & & \\
$r_{(j)} =\rank(Q_{(j)})=\!\!\!\!\! \sumls{0 \leq i \leq j \\ i \equiv j \pmod{2}} \!\!\!\!\! r_i$, & & $q_{(j)}  = q^{\,\sum_{0 \leq i < j} r_{(i)}}$, \\
& &  \\
\end{tabular}
\end{center}
and for any quadratic form $P$, use the term $G_P$ to denote the $p$-adic generating function of $P$, and $H_P$ for the head of $G_P$.
\begin{itemize}
\item If $L=0$, then the $p$-adic generating function of $Q$ is
\[
G_Q(z)=\sum_{0 \leq i < \omega-1} \frac{z^c}{q_{(i)}} H_{Q_{(i)}}(z^{\unif^i}) + \frac{z^c}{q_{(\omega-1)}} G_{Q_{(\omega-1)}}(z^{\unif^{\omega-1}}) G_{Q_{(\omega)}}(z^{\unif^{\omega}}).
\]
\item If $L=a x$ with $\val(a)=\lambda < \infty$, then the $p$-adic generating function of $Q$ is
\[
G_Q(z)=\sum_{0 \leq i < \lambda} \frac{z^c}{q_{(i)}} H_{Q_{(i)}}(z^{\unif^i}) + \frac{\zcoset{c}{\lambda}}{q_{(\lambda)}}.
\]
\end{itemize}
\end{corollary}
\begin{proof}
This follows from Theorem \ref{Victor}, where we have $\ell=\val(2)=0$.
Corollary \ref{Hortense} says that if $P_0$ is a unimodular quadratic form, then $H_{P_0}(z)$ is $\pi$-uniform, so if $P_1$ is another unimodular quadratic form, then Corollary \ref{Philip} shows that $H_{P_0}(z) G_{P_1}(z^\unif)=H_{P_0}(z)$.
\end{proof}

\subsection{Local Zeta Function for a Quadratic Polynomial}
We now use what we know about the $p$-adic generating function of a quadratic polynomial to determine its local zeta function.
\begin{theorem}\label{John}
Suppose $\val(2)=\ell$.
Consider the quadratic polynomial $Q=\bigoplus_{i\in\N} \unif^i Q_i \oplus L + c$, where each $Q_i$ is a unimodular quadratic form of rank $r_i$ over $\ring$, where there is a positive integer $\omega$ such that $Q_i=0$ for $i > \omega$, where $L$ is a linear form involving at most one indeterminate, and $c \in \ring$ is a constant.
For each $j \in \N$, we let
\[
Q_{(j)}  = \!\!\!\!\!\!\!\! \bigopluss{0 \leq i \leq j \\ i \equiv j \!\!\!\! \pmod{2}} Q_i,  \quad\quad\quad r_{(j)} =\rank(Q_{(j)})=\!\!\!\!\!\!\!\! \sums{0 \leq i \leq j \\ i \equiv j \!\!\!\! \pmod{2}} \!\!\!\!\! r_i, \quad\quad\quad q_{(j)}  = q^{\, \sum_{0 \leq i < j} r_{(i)}},
\]
and for any quadratic form $P$, use the term $G_P$ to denote the $p$-adic generating function of $P$, and $H_P$ to denote the head of $G_P$.  For $a \in \ring$, $\mu\in\N$, and quadratic forms $P_0,P_1,\ldots,P_{2\ell}$, let $I_a(P_0,P_1,\ldots,P_{2\ell})=\Ig(z^a H_{P_0}(z) G_{P_1}(z^\unif)\ldots G_{P_{2\ell}}(z^{\unif^{2\ell}}))$, and let $I_a^\mu=\Ig(\zcoset{a}{\mu} H_{P_0}(z) G_{P_1}(z^\unif)\ldots G_{P_{2\ell}}(z^{\unif^{2\ell}}))$.
\begin{itemize}
\item If $L=0$ and $c=0$, let $r=\sum_{i \in \N} r_i$, and then the Igusa local zeta function for $Q$ is
\begin{multline*}
\quad\enspace Z_Q(t)=\sum_{0 \leq i < \omega-1} \frac{t^i}{q_{(i)}} I_0(Q_{(i)},Q_{(i+1)},Q_{i+2},Q_{i+3},\ldots,Q_{i+2\ell}) \\
\quad\quad\enspace + \left(\frac{t^{\omega-1}}{q_{(\omega-1)}} I_0(Q_{(\omega-1)},Q_{(\omega)},0,0,\ldots,0)  + \frac{t^\omega}{q_{(\omega)}} I_0(Q_{(\omega)},Q_{(\omega-1)},0,0,\ldots,0)\!\right)\!\! \left(1-\frac{t^2}{q^r}\right)^{-1}\!\!\!\!\!.
\end{multline*}
\item If $L(x)=b x$ for some $b$ with $\val(b)=\lambda < \infty$, and if $\val(b) \leq \val(c)$, then the Igusa local zeta function for $Q$ is
\begin{multline*}
\quad\quad\quad Z_Q(t)=\sum_{0 \leq i < \lambda-2\ell} \frac{t^i}{q_{(i)}} I_0(Q_{(i)},Q_{(i+1)},Q_{i+2},Q_{i+3},\ldots,Q_{i+2\ell}) \\
\quad\quad\quad\quad\quad\enspace + \!\!\!\!\!\!\!\! \sum_{\max\{0,\lambda-2\ell\} \leq i < \lambda} \frac{t^i}{q_{(i)}} I_0^{\lambda-i}(Q_{(i)},Q_{(i+1)},Q_{i+2},Q_{i+3},\ldots,Q_{i+2\ell}) 
+ \frac{t^\lambda}{q_{(\lambda)}} \igrp.
\end{multline*}
\item If $L(x)=b x$ for some $b$ with $\val(b)=\lambda < \infty$, and if $\val(c) < \val(b) \leq \val(c)+2\ell$, then let $\kappa=\val(c)$, and then the Igusa local zeta function for $Q$ is
\begin{multline*}
\quad\quad\quad Z_Q(t)=\sum_{0 \leq i < \lambda-2\ell} \frac{t^i}{q_{(i)}} I_{c/\pi^i}(Q_{(i)},Q_{(i+1)},Q_{i+2},Q_{i+3},\ldots,Q_{i+2\ell}) \\
+ \sum_{\max\{0,\lambda-2\ell\} \leq i \leq \kappa} \frac{t^i}{q_{(i)}} I_{c/\pi^i}^{\lambda-i}(Q_{(i)},Q_{(i+1)},Q_{i+2},Q_{i+3},\ldots,Q_{i+2\ell})
+ \frac{t^\kappa}{q_{(\kappa+1)}}.
\end{multline*}
\item If $L(x)=b x$ for some $b$ with $\val(b) > \val(c)+2 \ell$ (this includes the case where $L=0$, $c\not=0$), let $\kappa=\val(c)$, and then
\[
Z_Q(t)=\sum_{0 \leq i \leq \kappa} \frac{t^i}{q_{(i)}} I_{c/\pi^i}(Q_{(i)},Q_{(i+1)},Q_{i+2},Q_{i+3},\ldots,Q_{i+2\ell}) + \frac{t^\kappa}{q_{(\kappa+1)}}.
\]
\end{itemize}
\end{theorem}
\begin{proof}
We apply $\Ig$ to the $p$-adic generating functions supplied by Theorem \ref{Victor} (when $p=2$) or Corollary \ref{Matilda} (when $p$ is odd).

We note that when $a \in \ring$ and $F(z) \in \limgroupring$, Lemma \ref{Tina} tells us that the term $\Ig(z^a F(z^{\unif^{i}}))$ becomes $t^i \Ig(z^{a/\pi^i} F(z))$ when $\val(a) \geq i$, and it becomes $F(1) t^{\val(a)}$ when $\val(a) < i$.  And similarly, if $\mu \in \N$, then $\Ig(z^{a+\pi^\mu\ring} F(z^{\unif^{i}}))$ becomes $t^i \Ig(z^{a/\pi^i + \pi^{\mu-i} \ring} F(z))$ when $\val(a),\mu \geq i$, and it becomes $F(1) t^{\val(a)}$ when $\val(a) < \mu, i$.

When $L=0$ and $c=0$, we use Lemmata \ref{Paul} and \ref{Tina} to apply $\Ig$ to the last term $G_{Q_{(\omega-1)}}(z^{\unif^{\omega-1}}) G_{Q_{(\omega)}}(z^{\unif^\omega})$, and note that $r_{(\omega-1)}+r_{(\omega)}=r$.

When $L= b x$ with for some $b$ with $\val(b)=\lambda < \infty$, and when $\val(c) \geq \val(b)$, then a reparameterization with $x'=(x-c/b)$ eliminates the constant term, so we may take $c=0$ in this case.  Then we apply $\Ig$ and use Lemma \ref{Stevie} for the last term.

When $L(x)=b x$ for some $b$ with $\val(b)=\lambda < \infty$, and when $\val(c) < \lambda$, and we let $\kappa=\val(c)$, then we again use the fact that $\Ig(z^c F(z^{\unif^{i}}))=F(1) t^\kappa$ whenever $i > \kappa$ to obtain either
\begin{multline*}
\!\!\!\! \Ig(G_Q(z))=\sum_{0 \leq i < \lambda-2\ell} \frac{t^i}{q_{(i)}} I_{c/\pi^i}(Q_{(i)},Q_{(i+1)},Q_{i+2},Q_{i+3},\ldots,Q_{i+2\ell}) \\
+ \!\!\!\!\!\!\!\!\!\! \sum_{\max\{0,\lambda-2\ell\} \leq i \leq \kappa} \frac{t^i}{q_{(i)}} I_{c/\pi^i}^{\lambda-i}(Q_{(i)},Q_{(i+1)},Q_{i+2},Q_{i+3},\ldots,Q_{i+2\ell}) 
+ \sum_{\kappa < i < \lambda} \frac{t^\kappa}{q_{(i)}} H_{Q_{(i)}}(1) + \frac{t^\kappa}{q_{(\lambda)}},
\end{multline*}
if $\lambda - 2\ell \leq \kappa < \lambda$, or else
\[
\Ig(G_Q(z))=\sum_{0 \leq i \leq \kappa} \frac{t^i}{q_{(i)}} I_{c/\pi^i}(Q_{(i)},Q_{(i+1)},Q_{i+2},Q_{i+3},\ldots,Q_{i+2\ell}) + \sum_{\kappa < i < \lambda} \frac{t^\kappa}{q_{(i)}} H_{Q_{(i)}}(1) + \frac{t^\kappa}{q_{(\lambda)}},
\]
if $\kappa < \lambda-2\ell$, and again note that $H_{Q_{(i)}}(1)=1-\frac{1}{q^{r_{(i)}}}$, and the last sum and final term of each form telescope to give the desired form.

When $L=0$ and $c\not=0$, we make sure to use $\omega > \kappa+1$, so that we obtain
\begin{multline*}
\Ig(G_Q(z))=\sum_{0 \leq i \leq \kappa} \frac{t^i}{q_{(i)}} I_{c/\pi^i}(Q_{(i)},Q_{(i+1)},Q_{i+2},Q_{i+3},\ldots,Q_{i+2\ell}) \\ + \sum_{\kappa < i < \omega-1} \frac{1}{q_{(i)}} H_{Q_{(i)}}(1) t^\kappa + \frac{1}{q_{(\omega-1)}} t^\kappa,
\end{multline*}
and then note that $H_{Q_{(i)}}(1)=1-\frac{1}{q^{r_{(i)}}}$ because the head of a generating function for a form with $r_{(i)}$ variables records the form's values on a set of Haar volume $1-1/q^{r_{(i)}}$.  Then one sees that the second sum and last term telescope to give the desired form.
\end{proof}

\subsection{Proof of Theorem \ref{Odilia}}\label{Winston}

This follows from Theorem \ref{John}, where we have $\ell=\val(2)=0$, so no term of the form $I_a^\mu(P_0)$ occurs, and the term $I_a(P_0)$ is just $\Ig(z^a H_{P_0}(z))$.  If $P_0$ is of rank $n$ and discriminant $e$, then we write this as $I_a(n,e)$.  The calculation of the values $I_a(n,e)$ in Table \ref{Gary} is in Lemma \ref{Ophelia} in the next section.

\subsection{Proof of Theorem \ref{Ursula}}\label{Wilbur}

We apply Theorem \ref{John} with $\ell=1$, $\pi=2$, and $c=0$.
For the cases where $L(x)\not=0$, the terms $I_0^2(Q_{(\lambda-2)},Q_{(\lambda-1)},Q_\lambda)$ and $I_0^1(Q_{(\lambda-1)},Q_{(\lambda)},Q_{\lambda+1})$ are dealt with specially.
We note that if $P_0$, $P_1$, and $P_2$ are unimodular quadratic forms, then $H_{P_0}(z) G_{P_1}(z^2) G_{P_2}(z^4) z^{4 \ring}=H_{P_0}(z) G_{P_1}(z^2) z^{4\ring}$ by Remark \ref{Zachary}, which is the same as $H_{P_0}(z) G_{P_1}(z^2) G_{\Sq}(z^4)$ by Lemma \ref{Robert} of Appendix \ref{Mordecai}, so that $I_0^2(P_0,P_1,P_2)=I_0(P_0,P_1,\Sq)$.
And similarly, $H_{P_0}(z) G_{P_1}(z^2) G_{P_2}(z^4) z^{2 \ring}=H_{P_0}(z) z^{2\ring}$ by Remark \ref{Zachary}, which is the same as $H_{P_0}(z) G_{\Sq}(z^2) G_{\Sq}(z^4)$ by Lemma \ref{Andrew}, so that $I_0^1(P_0,P_1,P_2)=I_0(P_0,\Sq,\Sq)$.
So all our terms $I_0^\mu(P_0,P_1,P_2)$ can be replaced with $I_0(P_0,A,B)$ for some unimodular quadratic forms $A$ and $B$, and since the subscript on $I$ is always zero for Theorem \ref{Ursula} (since $c=0$), we suppress it.
The formulae in Theorem \ref{Ursula} for $\lambda=0$ and $1$ are then easy obtained.  The values of $I(P_0,P_1,P_2)$ in Table \ref{Violet} are calculated in Appendix \ref{Mordecai} below.

\subsection{Location of Candidate Poles}\label{Michael}

Remark \ref{Percy} follows if we can show that terms of the form $I_a(P_0,\ldots,P_{2\ell})$ and $I^\mu_a(P_0,\ldots,F_{2\ell})$ as defined in Theorem \ref{John} are equal to polynomials in $\Q[t]$ divided by $1-t/q$.  This will follow from Lemma \ref{Stevie} if we show that for any unimodular quadratic forms $P_0,\ldots,P_{2\ell}$, the term $H_{P_0}(z) G_{P_1}(z^\unif) \cdots G_{P_{2\ell}}(z^{\unif^{2\ell}})$ is $\pi^k$-uniform for some $k \in \N$.  But this follows from Corollary \ref{Hortense} (applied to $P_0$) and Lemma \ref{Una}.

\subsection{Proof of Theorem \ref{Arthur}}\label{Morton}

We use Theorem \ref{Odilia}, which shows that the local zeta function is a $\Q$-linear combination (with positive coefficients) of terms of the form $t^i I_0(n,e)$ with $I_0(n,e)$ from Table \ref{Gary}, plus one final term involving a division by $(1-t^2/q^r)$ that we analyze now.

\begin{lemma}\label{Deborah}
With the assumptions and notations of Theorem \ref{Odilia}, let $r=\sum_{i \in \N} r_i$ and $d=\prod_{i \in \N} d_i$, and let
\[
M=\left(\frac{t^{\omega-1}}{q_{(\omega-1)}} I_0(r_{(\omega-1)},d_{(\omega-1)}) + \frac{t^\omega}{q_{(\omega)}} I_0(r_{(\omega)},d_{(\omega)})\right) \left(1-\frac{t^2}{q^r}\right)^{-1}.
\]
\begin{itemize}
\item If $r_{(\omega-1)}$ and $r_{(\omega)}$ are both odd, then
\[
M=\left(\frac{t^{\omega-1}}{q_{(\omega-1)}}\right)\igrp.
\]
\item If $r_{(\omega-1)}$ and $r_{(\omega)}$ are both even, then
\[
M=\left(\frac{t^{\omega-1}}{q_{(\omega-1)}}\right) \left[1+ \frac{\eta((-1)^{r_{(\omega-1)}/2} d_{(\omega-1)}) (t-1)}{q^{r_{(\omega-1)}/2}\left(1- \frac{\eta((-1)^{r/2} d) t}{q^{r/2}}\right)} \right]\igrp.
\]
\item If $r_{(\omega-1)}$ is even and $r_{(\omega)}$ is odd, then
\[
M=\left(\frac{t^{\omega-1}}{q_{(\omega-1)}}\right) \left[1+ \frac{\eta((-1)^{r_{(\omega-1)}/2} d_{(\omega-1)}) (t-1)}{q^{r_{(\omega-1)}/2}  \left(1-\frac{t^2}{q^r}\right)}  \right]\igrp.
\]
\item If $r_{(\omega-1)}$ is odd and $r_{(\omega)}$ is even, then
\[
M=\left(\frac{t^{\omega-1}}{q_{(\omega-1)}}\right) \left[1+ \frac{\eta((-1)^{r_{(\omega)}/2} d_{(\omega)}) t(t-1)}{q^{(r_{(\omega-1)}+r)/2} \left(1-\frac{t^2}{q^r}\right)}  \right]\igrp.
\]
\end{itemize}
\end{lemma}
\begin{proof}
Substitute the values of $I_0(r_{(\omega-1)},d_{(\omega-1)})$ and $I_0(r_{(\omega)},d_{(\omega)})$ from Table \ref{Gary} and simplify.
\end{proof}

{\it Proof of Theorem \ref{Arthur}:}  The rank $0$ case is trivial, so we assume positive rank henceforth.  As mentioned above Theorem \ref{Odilia} shows that the local zeta function is a $\Q$-linear combination (with positive coefficients) of terms of the form $t^i I_0(n,e)$ with $I_0(n,e)$ from Table \ref{Gary}, plus the final term involving a division by $(1-t^2/q^r)$ that we just analyzed in Lemma \ref{Deborah}.  Perusal of Table \ref{Gary} shows that $I_0(n,e)$ is zero when $n=0$, is a positive constant when $(n,e) \in \{(1,1),(1,\alpha),(2,-\alpha)\}$, and otherwise $(1-t/q) I_0(n,e)$ is a polynomial whose value at $t=q$ is positive.

First suppose that $\{(r_{\even},d_{\even}), (r_{\odd},d_{\odd})\} \subseteq \{(0,1),(1,1),(1,\alpha),(2,-\alpha)\}$.  Then Theorem \ref{Odilia} shows that the local zeta function is a polynomial plus the final term analyzed in Lemma \ref{Deborah}.  So the denominator $g(t)$ we seek is the denominator of that term when (written in reduced form), and this can readily be deduced for the finitely many possibilities for $(r_{\even},d_{\even})$ and $(r_{\odd},d_{\odd})$ under consideration here by consulting Table \ref{Gary}.

So henceforth assume that $\{(r_{\even},d_{\even}), (r_{\odd},d_{\odd})\} \not\subseteq \{(0,1),(1,1),(1,\alpha),(2,-\alpha)\}$.  This makes $r \geq 2$.  If $r=2$, we must have $\{(r_{\even},d_{\even}),(r_{\odd},d_{\odd})\}=\{(0,1),(2,-1)\}$, and we let $\omega$ be the greatest index such that $Q_\omega\not=0$.  Then Theorem \ref{Odilia} and the fact that $I_0(r_{(i)},d_{(i)})$ is a constant for $i < \omega$ (and is zero when $i\not\equiv \omega \pmod{2}$) show that the local zeta function is a polynomial plus a positive constant times $t^\omega I_0(r_{(\omega)},d_{(\omega)})/(1-t^2/q^2)$, and since $(r_{(\omega)},d_{(\omega)})=(2,-1)$, Table \ref{Gary} tells us that this last term (in reduced form) has denominator $(1-t/q)^2$.

So we may assume that $r \geq 3$ henceforth.  The factors $1\pm t/q^{r/2}$ can no longer occur in the denominator of any $I_0(n,e)$, so they will be present in the local zeta function for $Q$ if and only if they are present in the denominator of the final term analyzed in Lemma \ref{Deborah}.  Since $I_0(r_{(i)},d_{(i)})$ is not a constant for some $i$, when we multiply the local zeta function by $1-t/q$ and evaluate at $t=q$, this term gives a positive contribution, and the other terms give a nonnegative one (we are including the final term, whose denominator is positive at $t=q$ in view of the rank).  So the local zeta function must have $1-t/q$ in its denominator when written in reduced form.

\subsection{Comparison with Igusa's Results}

Igusa calculated the local zeta function for a polynomial of the form $Q\oplus L$ where $Q$ is a quadratic form, $L$ is a linear form, and $p$ is odd in Theorem 1 and Corollary to Theorem 2 of \cite{Igusa-1994-Local}, which is the special case of Theorem \ref{Odilia} when $c=0$.
He also has some calculations in his monograph \cite[Corollary 10.2.1]{Igusa-2000}, which give a special case of his results for $p$ odd in \cite{Igusa-1994-Local}, but also give local zeta functions for a restricted family of quadratic forms over $2$-adic fields (including ramified ones) that behave similarly to unimodular quadratic forms over $p$-adic fields with $p$ odd.  In our notation, when $p=2$, Igusa restricts his attention to quadratic forms equal to $1/2$ times one of the following: $\Hyp^n$, $\Ell\oplus \Hyp^{n-1}$, $\Hyp^n \oplus 2 u\Sq$, or $\Ell\oplus \Hyp^{n-1}\oplus 2 u\Sq$, where $u$ is a unit in $\ring$.  The factor of $1/2$ merely causes the zeta function to be scaled by $t^{-\val(2)}$.  When there is no $2 u\Sq$ term present, Igusa's result follows directly from our calculations of zeta functions for unimodular forms in Lemma \ref{Rebecca} below.  When the $2 u\Sq$ term is present, write $Q_+=\Hyp^{(r-1)/2}$ and $Q_-=\Ell\oplus\Hyp^{(r-3)/2}$ so that $Q=Q_\pm \oplus 2 u \Sq$.  Then multiply the $p$-adic generating functions $G_{Q_\pm}(z)$ and $G_{u\Sq}(z^2)$ of $Q_\pm$ and $2u\Sq$ from Lemma \ref{Rebecca} and Corollary \ref{Hortense} and coalesce cosets (see Remark \ref{Anne}) to show that the $p$-adic generating function for $Q$ is
\[
G_Q(z)=\left(1-\frac{1}{q^{r-1}}\right) z^{2\ring} + \frac{1}{q^{r+\ell}} G_{Q_\pm}(z^{\unif^2}) \sum_{\tau \in T_\ell^*} z^{2u \tau^2+8\unif \ring} + \frac{1}{q^r} G_Q(z^{\unif^2}),
\]
where $\ell=\val(2)$, $T$ is a set of Teichm\"uller representatives for $\F_q$ in $K$ (i.e., $T$ contains all $(q-1)$th roots of unity and zero), and $T_\ell^*=\{t_0+t_1 \pi+\cdots+t_\ell\pi^\ell: t_0,\ldots,t_\ell \in \teich, t_0\not=0\}$.  Apply $\Ig$ to both sides, recognizing that $G_{Q_\pm}(z^{\unif^2})=G(z^{2\unif^2})$ for some $G(z) \in \limgroupring$, so that $\Ig$ applied to the second term gives $(q-1) t^\ell/q^r$, and also note that $\Ig$ applied to the last term gives $t^2 \Ig(G_Q(z))/q^r$ by Lemma \ref{Tina}, and so we rearrange to get $\Ig(G_Q(z))=\frac{t^\ell (1-1/q)(1-t/q^r)}{(1-t/q)(1-t^2/q^r)}$, which matches Igusa's result when we divide by $t^\ell$ to account for his difference in scaling.

\subsection{Strong Isospectrality for $p$ odd and $\Z_2$}\label{Armand}

Recall from Section \ref{Gordon} that two polynomials $f$ and $g$ over $\ring$ are said to be {\it strongly isospectral} if for each $k \in \N$, the reductions of $f$ and $g$ modulo $\pi^k$ represent each value in $\ring/\pi^k\ring$ the same number of times.  This is equivalent to saying that $f$ and $g$ have the same $p$-adic generating function.
We prove that in certain cases, quadratic polynomials over $\ring$ are isospectral to ones where no variable appears in both the linear and quadratic part.
\begin{proposition}\label{Theresa}
Suppose that $p$ is odd, or else that $p=2$ and $\ring=\Z_2$.  Let $Q$ be a quadratic polynomial over $\ring$.  Then $f$ is strongly isospectral to a polynomial of the form $\bigoplus_{i=0}^{\omega} \pi^i Q_i \oplus \pi^{\lambda} L + c$, for some unimodular quadratic forms $Q_0, Q_1,\ldots,Q_\omega$, where $L$ is a linear form involving at most one variable, $\lambda > \omega$, and $c$ is a constant in $\ring$.
\end{proposition}
\begin{proof}
In \cite[\S 91C]{O'Meara-1963} it is shown that for any quadratic form, there is an invertible $\ring$-linear change of variables that transforms it to a form $\bigoplus_{i=0}^{\omega} \pi^i P_i$, where each $P_i$ is a unimodular quadratic form.  In \cite[\S 91C, 92:1]{O'Meara-1963}) it is shown that one can arrange that each $P_i$ be a direct sum of unimodular quadratic forms of rank $1$ (when $p$ is odd) or ranks $1$ and $2$ (when $p=2$).
We apply such a transformation to our quadratic polynomial, thus bringing the quadratic portion into this convenient form.
This change of variables transforms the linear portion of our polynomial to another linear form $M$.

Now for each rank $1$ or $2$ direct summand of the quadratic portion, say $f(x)$ or $g(x,y)$, consider $f(x)+a x$ or $g(x,y)+a x + b y$, where the $a x$ or $a x+b y$ is the portion of $M$ involving $x$ or $x$ and $y$.
Since a rank $2$ summand that cannot be decomposed into two rank $1$ summands must be $\Hyp$ or $\Ell$ by Corollary \ref{Frank}, Lemmata \ref{Loren}--\ref{Leslie} below show that $f(x)+a x$ or $g(x,y)+a x+b y$ is strongly isospectral to a linear form plus a constant or a quadratic form plus a constant.

Thus we can assume that there are no variables in common between the linear and quadratic portions of our polynomial.  Then note that any nonzero linear form is isospectral to a linear form with a single variable by an invertible $\ring$-linear change of variables.

If our single-variable linear term is $a x$, and there is a quadratic term of the form $\pi^\mu P$ with $\mu \geq \val(a)$, Lemma \ref{Ludwig} shows that we can remove the $\pi^\mu P$ term.
\end{proof}
We conclude with some technical lemmata used above.
\begin{lemma}\label{Ludwig}
Let $f(x_1,\ldots,x_n)=a x_1+ b P(x_2,\ldots,x_n)$ with $a,b \in \ring$, and $P(x_2,\ldots,x_n) \in \ring[x_2,\ldots,x_n]$.
If $\val(a) \leq \val(b)$, then $f(x_1,\ldots,x_n)$ is strongly isospectral to $a x_1$.
\end{lemma}
\begin{proof}
By scaling, we may assume that $a=1$.
Let $k \in \N$.  Since $x_1$ does not occur in $b P(x_2,\ldots,x_n)$, both the polynomial $f(x_1,\ldots,x_n)$ and the polynomial $g(x_1,\ldots,x_n)=x_1$, when taken modulo $\pi^k$, represent each element of $R/\pi^k R$ equally often.
\end{proof}
\begin{lemma}\label{Lester}
Let $f(x_1,\ldots,x_n)=a x_1+ b L(x_2,\ldots,x_n) + c P(x_1,\ldots,x_n)$ with $a,b, c\in \ring$, and $L(x_2,\ldots,x_n)$, $P(x_1,\ldots,x_n) \in \ring[x_1,\ldots,x_n]$ with $L(x_2,\ldots,x_n)$ a linear form.
If $\val(a) \leq \val(b)$ and $\val(a) < \val(c)$, then $f(x_1,\ldots,x_n)$ is strongly isospectral to $a x_1$.
\end{lemma}
\begin{proof}
By scaling, we may assume that $a=1$ and $\pi \mid c$.
Then the partial derivative of $f$ at $x_1$ never vanishes modulo $\pi$, so Lemma \ref{Gretchen} shows that the $p$-adic generating function for $f$ is the same that of $x_1+ b L(x_2,\ldots,x_n)$, which by an invertible $\ring$-linear change of variables is the same as that of $x_1$.
\end{proof}
\begin{lemma}\label{Loren}
Let $f(x)=a x^2+b x$ for some $a,b \in \ring$.
If $\val(b) < \val(a)$, then $f(x)$ is strongly isospectral to $b x$.
If $\val(b)=\val(a)$ and $\ring=\Z_2$, then $f(x)$ is strongly isospectral to $2 b x$.
If $\val(b) \geq \val(2 a)$, then $f(x)$ is strongly isospectral to $a x^2 - \frac{b^2}{4 a}$.
\end{lemma}
\begin{proof}
Lemma \ref{Lester} handles the first case, and the third case is obtained by completing the square.
In the second case, we reparameterize with $x= b y/a$ to obtain $(b^2/a) (y^2+y)$.  It suffices to show that $y^2+y$ is strongly isospectral to $2 y$, for then $f$ will be strongly isospectral to $2 b (b/a) y$, which is isospectral to $2 b y$, since $b/a$ is a unit.  Note that the derivative of $y+y^2$ is always $1$ modulo $2$, and so by Lemma \ref{Gretchen}, the $p$-adic generating function of $y^2+y$ is $z^{2\ring}$, which is the $p$-adic generating function of $2 x$ by Remark \ref{Otto}.
\end{proof}
\begin{lemma}
Let $f(x)=a x+ b y + c x y$ for some $a, b, c \in \ring$.
If $\val(a), \val(b) \geq \val(c)$, then $f(x,y)$ is strongly isospectral to $c x y-\frac{a b}{c}$.
Otherwise, $f(x,y)$ is strongly isospectral to $a x + b y$.
\end{lemma}
\begin{proof}
In the first case, use the change of variables $x=u-b/c$ and $y=v-a/c$, and handle the residual cases with Lemma \ref{Lester}.
\end{proof}
\begin{lemma}\label{Leslie}
Let $\sigma \in \ring$ such that $z^2+ z + \sigma \pmod{\pi}$ is irreducible in $\resfield$ and set $e(x,y)=(x^2+x y+\sigma y^2)$, so that $2 e(x)$ is an elliptic plane.  Let $f(x,y) = a x + b y + c e(x,y)$.
If $\val(a), \val(b) \geq \val(c)$, then $f(x,y)$ is strongly isospectral to $c e(x,y) + \frac{\sigma a^2-a b + b^2}{c(1-4\sigma)}$.
Otherwise, $f(x,y)$ is strongly isospectral to $a x + b y$.
\end{lemma}
\begin{proof}
In the first case, use the change of variables $x=u-(b-2\sigma a)/(c(1-4\sigma))$ and $y=v-(a-2 b)/(c(1-4\sigma))$, and handle the residual cases with Lemma \ref{Lester}
\end{proof}

\section{Unimodular Quadratic Forms}\label{Carlos}

Since Theorem \ref{Victor} expresses the $p$-adic generating function and Igusa zeta function of a quadratic polynomial in terms of the $p$-adic generating functions (and their heads) of unimodular quadratic forms, we now analyze what unimodular forms look like.  This will eventually enable us to calculate the entries of Tables \ref{Gary} and \ref{Violet} used by Theorems \ref{Odilia} and \ref{Ursula}.

We set down some basic assumptions and notations that shall hold in Sections \ref{George}--\ref{Leonard}.
We assume the basic facts about quadratic forms presented in Section \ref{Gordon}.  We also recall from Section \ref{Olivia} that when $p$ is odd, $\unitsmodsquares$ is a group of order $2$, and we fix $\alpha\in\ringunits\smallsetminus\ringsquareunits$, and we use the extended character $\eta$ defined in in Section \ref{Olivia}.

When $p=2$, the residue field is perfect and of characteristic $2$, so every element $a$ of $\ring$ is a square modulo $2$, that is, $a\equiv b^2 \pmod{2}$ for some $b \in \ring$.  So $a\Sq$ is always equivalent to $b\Sq$ for some $b\in\ring$ with $b\equiv 1 \pmod{2}$. 
Furthermore, every element of the form $1+4\pi a$ with $a \in R$ is a square (so $\unitsmodsquares$ is finite), but we can (and do) fix a unit $\xi\in R$ such that $1+4 \xi$ is not a square (see \cite[\S 63:1 and 64:4]{O'Meara-1963} for proofs of these facts).
In fact, it is an easy consequence of Hilbert's Theorem 90 that the elements $a \in \ring$ such that $1+4 a$ is square are precisely those such that if $\bar{a} \in \resfield$ is the reduction of $a$ modulo $\pi$, then the absolute trace of $\bar{a}$ is $0$.
When $p=2$ and is unramified in $\ring$, we shall use $\Tr$ to denote the absolute trace $\Tr\colon \field \to \Q_2$, so if $a \in \ring$, then $1+4 a$ is a square if and only if $\Tr(a) \equiv 0 \pmod{2}$.
We use the convention that if $a \in \Z_2$, then $(-1)^a=(-1)^{(a \bmod{2})}$.

Throughout this section, we let $\teich$ be a set of Teichm\"uller representatives for $\resfield$ in $\field$, that is, $T$ contains all the $(q-1)$th roots of unity and $0$.
We let $\teichu=T\smallsetminus\{0\}$.
When $p=2$ and does not ramify in $\ring$, we let $S=\{\tau \in \teich: \Tr(\tau) \equiv 0 \pmod{2}\}$.  Note that $\card{S}=q/2$.

\subsection{Unimodular Quadratic Forms of Rank $1$}\label{George}

A unimodular quadratic form of rank $1$ is $u \Sq$ for some $u \in\ringunits$.
Let us examine the $p$-adic generating functions for these quadratic forms when $p$ is odd and in the unramified $2$-adic case.
\begin{lemma}\label{Beatrice}
For $p$ odd and $a \in \ringunits$, the head of the $p$-adic generating function of $a\Sq$ is 
\[
H_{a \Sq}(z)=z^\ring- \frac{1}{q} z^{\unif\ring} + \frac{1}{q} \sum_{\tau \in \teich} \eta(a \tau) z^{\tau+\unif\ring},
\]
and the $p$-adic generating function satisfies $G_{a \Sq}(z)=H_{a \Sq}(z) + \frac{1}{q} G_{a \Sq}(z^{\unif^2})$.
\end{lemma}
\begin{proof}
From Corollary \ref{Hortense} and the fact that $a x^2 \pmod{\pi}$ represents each element of $a \resfieldsquareunits$ twice as $x$ runs through $\teichu$, we deduce that the head is $(1/q)\sum_{\tau \in \teichu} (1+\eta(a\tau)) z^{\tau+\pi\ring}$, and coalesce cosets (see Remark \ref{Anne}) to obtain the desired form.  Then Lemma \ref{Heinrich} gives the relation for $G_{a\Sq}$.
\end{proof}
\begin{lemma}\label{Diana}
Suppose that $p=2$ and that $\field$ is unramified.  If $a \in \ringunits$, then the head of the $p$-adic generating function of $a \Sq$ is
\[
H_{a \Sq}(z) = \frac{2}{q^2} \sums{\tau \in \teichu \\ s \in S} z^{a \tau (1+4 s)+ 8 \ring}.
\]
The $p$-adic generating function satisfies $G_{a \Sq}(z)=H_{a \Sq}(z) + \frac{1}{q} G_{a \Sq}(z^{\unif^2})$.
\end{lemma}
\begin{proof}
By Corollary \ref{Hortense}, we have
\[
H_{a \Sq}(z)=\frac{1}{q^2} \sums{\tau_0 \in \teichu \\ \tau_1 \in \teich} z^{a(\tau_0+2 \tau_1)^2+8\ring},
\]
and then note that $(\tau_0+2 \tau_1)^2 = \tau_0^2(1+4(r+r^2))$ where $r=\tau_1/\tau_0$.  As $(\tau_0,\tau_1)$ runs through $\teichu \times \teich$, we note that $(\tau_0^2,r)$ runs through $\teichu \times \teich$, and as $r$ runs through $\teich$, we note that $r^2+r \pmod{2}$ runs through $S$, necessarily taking each value twice (since a quadratic polynomial cannot take any value more than twice, and there are $q/2$ values of trace zero).  Thus as $(\tau_0,\tau_1)$ runs through $\teichu \times \teich$, $(\tau_0+2 \tau_1)^2 \pmod{8}$ becomes congruent to each $\tau(1+4 s)$ with $(\tau,s) \in \teichu \times S$ two times.
\end{proof}

\subsection{Unimodular Quadratic Forms of Rank $2$}\label{Jules}

Let us first compute the $p$-adic generating function for hyperbolic and elliptic planes.
\begin{lemma}\label{Norman}
The $p$-adic generating function of the hyperbolic plane satisfies
\[
G_{\Hyp}(z) = \left(1-\frac{1}{q}\right) z^{2 \ring} + \frac{1}{q} G_{\Hyp}(z^\unif).
\]
The head of the $p$-adic generating function is
\[
H_{\Hyp}(z) = \left(1-\frac{1}{q}\right) \left(z^{2 \ring}+\frac{1}{q} z^{2\unif\ring}\right),
\]
and the $p$-adic generating function satisfies $G_{\Hyp}(z)=H_{\Hyp}(z)+\frac{1}{q^2} G_{\Hyp}(z^{\unif^2})$.
When $p$ is odd, then $2 \ring=\ring$, so the above instances of $2 R$ and $2 \pi R$ may be replaced with $R$ and $\pi R$, respectively.
\end{lemma}
\begin{proof}
Let us consider the partial generating function for $x y$ on the region $S=(\ring\smallsetminus \unif\ring)\times \ring$, that is, where $x$ is a unit.
Then the derivative of $x y$ always has valuation $0$, and so we may use Lemma \ref{Gretchen} to show that the partial $p$-adic generating function on $S$ is $\frac{q-1}{q^2} \sum_{\tau \in T} z^{\tau+\unif\ring} = \frac{q-1}{q} z^\ring$.

Now consider $x y$ on the region $S^c = \ring^2\smallsetminus S$, that is, where $x$ is a multiple of $\pi$, and suppose we are interested in the partial modulo $\pi^k$ generating function of $x y$ on $S^c$.
We write $x = \pi \widetilde{x}$, and count the number of times $(\pi \widetilde{x}) y$ takes each value in $\quotring{k}$ as $\widetilde{x}$ runs through $\quotring{k-1}$ and $y$ runs through $\quotring{k}$, which is $q$ times the number of times $\pi(\widetilde{x} y)$ takes each value in $\quotring{k}$ as $\widetilde{x}$ and $y$ run through $\quotring{k-1}$.
Thus if the modulo $\pi^{k-1}$ generating function for $\widetilde{x} y$ is $F_{k-1}(\gamma)$, then the partial modulo $\pi^k$ generating function for $x y$ on $S^c$ is $\frac{1}{q} F_{k-1}(\gamma^{\pi})$.  (The $\gamma^\pi$ means we should replace each term $\gamma^A$ where $A \in \groupring{k-1}$ with $\gamma^{\pi A}$.  And the $\frac{1}{q}$ comes about, since we get the counting factor of $q$ just mentioned, but also pick up two factors of $q$ in the denominator due to our choice of normalization of modular generating functions.)  Taking limits, we see if the $p$-adic generating function for $\widetilde{x} y$ is $G(z)$, then the partial $p$-adic generating function for $x y$ on $S^c$ is $\frac{1}{q} G(z^\pi)$.  But of course $\widetilde{x} y$ and $x y$ represent the same form $G(z)$, so if we add the partial $p$-adic generating functions of $x y$ on $S$ and on $S^c$, then we see that the $p$-adic generating function of $x y$ satisfies $G(z)=\frac{q-1}{q} z^\ring + \frac{1}{q} G(z^\unif)$.  Now scale by $2$ to get the first relation we were to prove for the $p$-adic generating function of $\Hyp$.

We apply this relation to its own second term to get
\[
G_{\Hyp}(z) = \left(1-\frac{1}{q}\right) \left(z^{2 \ring}+\frac{1}{q} z^{2\pi\ring}\right) + \frac{1}{q^2} G_{\Hyp}(z^{\unif^2}),
\]
which shows that the head of the $p$-adic generating is exactly what we claim it to be by Lemma \ref{Heinrich}.
\end{proof}
Recall that whenever $f(X)$ is a quadratic polynomial over $\ring$ whose reduction modulo $\pi$ is an irreducible quadratic polynomial over $\resfield$, the rank $2$ form $2 y^2 f(x/y)$ is called the elliptic plane $\Ell$.  All such forms are equivalent, regardless of the choice of $f$.  Some standard forms for $\Ell$ are $x^2-\alpha y^2$ when $p$ is odd and $2(x^2+x y - \xi y^2)$ when $p=2$, where $\alpha$ and $\xi$ are as defined at the beginning of Section \ref{Carlos}.
We now compute the $p$-adic generating function for an elliptic plane.
\begin{lemma}\label{Ellen}
The $p$-adic generating function of the elliptic plane satisfies
\[
G_{\Ell}(z) = \left(1+\frac{1}{q}\right) z^{2 \ring} - \frac{1}{q} G_{\Ell}(z^\unif).
\]
The head of the $p$-adic generating function is
\[
H_{\Ell}(z) = \left(1+\frac{1}{q}\right) \left(z^{2 \ring}-\frac{1}{q} z^{2\unif\ring}\right),
\]
and the $p$-adic generating function satisfies $G_{\Ell}(z)=H_{\Ell}(z)+\frac{1}{q^2} G_{\Ell}(z^{\unif^2})$.
When $p$ is odd, then $2 \ring=\ring$, so the above instances of $2 R$ and $2 \pi R$ may be replaced with $R$ and $\pi R$, respectively.
\end{lemma}
\begin{proof}
Let us compute the head of the $p$-adic generating function for $\frac{1}{2}\Ell$, whose derivative has valuation $0$ on $\ring^2\smallsetminus (\pi\ring)^2$ (see proof of Corollary \ref{Hortense}).
Suppose that $f(X)$ is a quadratic polynomial over $\ring$ whose reduction modulo $\pi$, say $g(X)$, is an irreducible quadratic polynomial over $\resfield$.
Then express $\frac{1}{2}\Ell$ as $Q(x,y)=y^2 f(x/y)$.
Let $\theta$ and $\bar{\theta}$ be the roots of $g(x)$ in $\F_{q^2}$, and then note that $y^2 g(x/y)=\beta(x-\theta y)(x-\bar{\theta} y)$ for some $\beta \in \resfieldunits$.
Now suppose that $d,e \in \ring$ reduce modulo $\pi$ to $\delta,\epsilon \in \resfield$ respectively.
Then $Q(d,e) \pmod{\pi} = \beta N(\delta-\epsilon\theta)$, where $N$ is the Galois-theoretic norm from $\F_{q^2}$ to $\F_q$.
Furthermore, as $(d,e)$ runs through a set of representatives modulo $\pi$ of $\ring^2\smallsetminus(\pi\ring)^2$, our $(\delta,\epsilon)$ runs through $\F_q^2\smallsetminus\{(0,0)\}$, and so $\delta-\epsilon\theta$ runs through $\F_{q^2}^*$, and so $N(\delta-\epsilon\theta)$ runs through $\resfieldunits$, taking each value $q+1$ times.
Thus, by Lemma \ref{Gretchen}, we have $H_{\frac{1}{2}\Ell}(z)=\frac{q+1}{q^2} \sum_{\tau \in \teichu} z^{\tau+\pi\ring}$, and then we coalesce cosets and scale by $2$ to get the claimed expression for $H_{\Ell}(z)$.

The recursive expression $G_{\Ell}(z)=H_{\Ell}(z)+\frac{1}{q^2} G_{\Ell}(z^{\unif^2})$ comes immediately from Lemma \ref{Heinrich}.
Now set $I(z)=G_{\Ell}(z)+\frac{1}{q}G_{\Ell}(z^\unif)-\left(1+\frac{1}{q}\right) z^{2\ring}$, and use the recursion to obtain
\[
I(z) = H_{\Ell}(z) +\frac{1}{q^2} G_{\Ell}(z^{\unif^2})+\frac{1}{q} H_{\Ell}(z^\unif) + \frac{1}{q^3} G_{\Ell}(z^{\unif^3}) - \left(1+\frac{1}{q}\right) z^{2\ring},
\]
which is seen to be $\frac{1}{q^2} I(z^{\unif^2})$ when one substitutes the value of $H_{\Ell}(z)$.
Thus $I(1)=0$, and furthermore, we can iterate this relation to see that $I(z)=\frac{1}{q^{2 k}} I(z^{\unif^{2 k}})$ for every $k \in \N$.
Now if we apply $\phi_{2 k}$ to $I(z)$ to get an element of $\groupring{2 k}$, the fact that $I(z)$ is equal to $\frac{1}{q^{2 k}} I(z^{\pi^{2 k}})$ shows that the only term $\gamma^{\coset{\tau}{2 k}}$ of $\phi_{2 k}(I(z))$ that could have a nonzero coefficient would be $\gamma^{\unif^{2 k}\ring}$.  But the coefficient for this must also be zero since $I(1)=0$.  So $\phi_{2 k}(I(z))=0$ for every $k \in\N$, so $I(z)=0$.
This proves the first claim of our lemma.
\end{proof}
\begin{remark}\label{Gabriel}
We can repeatedly apply $G_{\Ell}(z)=H_{\Ell}(z)+\frac{1}{q^2} G_{\Ell}(z^{\unif^2})$ and use the value of $H_{\Ell}(z)$ from Lemma \ref{Ellen} to see that for each $k \in \N$, we have
\[
G_{\Ell}(z)=\sum_{i=0}^{k-1} \frac{1}{q^{2 i}} \left(1+\frac{1}{q}\right) \left(z^{2\unif^{2 i} \ring}-\frac{1}{q} z^{2\unif^{2 i+1}\ring}\right) + \frac{1}{q^{2 k}} G_{\Ell}(z^{\unif^{2 k}}).
\]
This shows that $\Ell$ represents precisely the elements $r \in \ring$ whose valuation has the same parity as $\val(2)$ (and also $\Ell$ represents $0$, but only trivially).
\end{remark}

When $p$ is odd, there are only two unimodular forms of rank $2$ in $R[x,y]$ up to equivalence (see \cite[\S 92:1a]{O'Meara-1963}), and these are evidently $\Hyp$ and $\Ell$ (these having different discriminants).

For the rest of this section, we assume that $p=2$.
By \cite[\S 93:17]{O'Meara-1963}, an arbitrary rank $2$ form in $R[x,y]$ is equivalent to a form $\Pl(a,b)$ given by $a x^2 + 2 x y - b a^{-1} y^2$, where $a,b \in R$ with $\val(a) \leq \min\{\val(2),\val(b)/2\}$, and $b$ equal to either $0$, or $4 \xi$, or else $b$ is an element with $\val(b)$ odd and less than $\val(4)$.
Note that the discriminant of $\Pl(a,b)$ is $-(1+b)$ and its norm is the ideal $a \ring$.
The restrictions we place on $a$ and $b$ stem from $a$ being the {\it norm generator} and $b$ generating a {\it quadratic defect}, which are discussed in \cite[\S 93:3, \S 63A]{O'Meara-1963}.

In the special case where $b=0$ and $\val(a)=\val(2)$, a simple coordinate transform shows that $\Pl(a,b)$ is equivalent to the hyperbolic plane.
Similarly, when $b=4\xi$ and $\val(a)=\val(2)$, it is not hard to show that $\Pl(a,b)$ is an elliptic plane.
If $a$ is a unit, we may complete the square to show that $\Pl(a,b)$ is equivalent to $a x^2 + c y^2$ for some unit $c$.
Note that $\Hyp$ and $\Ell$ have norm $2 R$ and $a x^2+ c y^2$ with $\val(a)=0$ has norm $R$.  All other $\Pl(a,b)$ have norm strictly between $\ring$ and $2\ring$.

When $K$ is unramified, our constraints on $a$ and $b$ then show that $\Pl(a,b)$ is either $\Hyp$ (when $\val(a)=1$ and $b=0$), $\Ell$ (when $\val(a)=1$ and $b=4\xi$), or of the form $a x^2+c y^2$ with $a, c\in \ringunits$ (when $\val(a)=0$).  The generating functions of the first two forms have been calculated explicitly above, and it will be useful to have an explicit calculation of the third.
\begin{lemma}\label{Earl}
Suppose that $p=2$ and does not ramify in $\ring$.  Let $a, b \in \ringunits$ with $a\equiv b \equiv 1 \pmod{2}$.
  
If $4 \mid a+b$, let $\sigma =(-1)^{\Tr((a+b)/(4 a))}$, and then the head of the $p$-adic generating function for $a \Sq \oplus b \Sq$ is
\[
H_{a \Sq\oplus b \Sq}(z)=z^\ring-\frac{1}{q} z^{2\ring} + \frac{q-1-\sigma}{q^2} z^{4 \ring} + \frac{\sigma}{q^2} z^{8 \ring}.
\]

If $4 \nmid a+b$, then the head of the $p$-adic generating function for $a \Sq \oplus b \Sq$ is
\[
H_{a \Sq\oplus b \Sq}(z) = \frac{2}{q^2} \sums{\tau \in \teichu \\ s \in S} z^{\tau\left(a + \frac{4}{a+b} s\right) + 4 \ring} + \frac{2}{q^3} \sums{\tau\in\teichu \\ s \in S} z^{\tau \left(a+b+4 s\right)+8\ring}.
\]
In all cases, the $p$-adic generating function satisfies
\[
G_{a \Sq\oplus b\Sq}(z)=H_{a \Sq\oplus b\Sq}(z) + \frac{1}{q^2} G_{a \Sq\oplus b\Sq}(z^{\unif^2}).
\]
\end{lemma}
This is proved in Appendix \ref{Karen}.

\subsection{Unimodular Quadratic Forms of Arbitrary Rank}

As mentioned in Section \ref{Gordon}, every unimodular form is equivalent to a direct sum of forms of ranks $1$ and $2$.
We now look at this principle more specifically so that we can calculate the $p$-adic generating functions of forms of higher rank.
It turns out that any form of rank greater than $4$ is always equivalent to a form containing a hyperbolic plane as a direct summand.
Therefore, it will be useful to compute generating functions of direct sums of hyperbolic planes.
\begin{lemma}\label{Rebecca}
Let $\ell=\val(2)$.  Let $r \in \N$ be even, let $Q_+=\Hyp^{r/2}$, and let $Q_-=\Ell\oplus\Hyp^{(r-2)/2}$.
The $p$-adic generating function of $Q_\pm$ satisfies
\[
G_{Q_\pm}(z) = \left(1\mp \frac{1}{q^{r/2}}\right) z^{2 \ring} \pm \frac{1}{q^{r/2}} G_{Q_\pm}(z^\unif), \quad\,\,\text{and}\quad\,\, G_{Q_\pm}(z)=H_{Q_\pm}(z)+\frac{1}{q^r} G_{Q_\pm}(z^{\unif^2}),
\]
where the head of the $p$-adic generating function of $Q_\pm$ is
\[
H_{Q_\pm}(z) = \left(1\mp \frac{1}{q^{r/2}}\right) \left(z^{2 \ring}\pm\frac{1}{q^{r/2}} z^{2\unif\ring}\right),
\]
with
\[
\Ig(H_{Q_\pm}(z)) =  \left(1\mp \frac{1}{q^{r/2}}\right) \left(t^{\ell} \pm\frac{t^{\ell+1}}{q^{r/2}}\right) \igrp.
\]
The Igusa local zeta function of $Q_\pm$ is 
\[
\Ig(G_{Q_\pm}(z))=\frac{\Ig(H_{Q_\pm}(z))}{1-t^2/q^r}.
\]
When $p$ is odd (i.e., when $\ell=0$), then $2 \ring=\ring$, so the above instances of $2 R$ and $2 \pi R$ may be replaced with $R$ and $\pi R$, respectively.
\end{lemma}
\begin{proof}
We begin by proving the first relation for $G_{Q_\pm}$.  We proceed by induction on $r$, with the $r=0$ case trivial, and the $r=2$ cases given by Lemmata \ref{Norman} and \ref{Ellen}.
For $r > 2$, we set $\widetilde{Q}_+=\Hyp^{(r-2)/2}$ and $\widetilde{Q}_-=\Ell\oplus\Hyp^{(r-4)/2}$, and let $G_{\widetilde{Q}_\pm}(z)$ and $G_{\Hyp}(z)$ be the $p$-adic generating functions of $\widetilde{Q}_\pm$ and $\Hyp$.
Then the sum-product rule tells us that we can multiply these together to get the generating function for $Q_\pm$, so 
\[
G_{Q_\pm}(z) = \left[\left(1-\frac{1}{q}\right) z^{2 \ring} + \frac{1}{q} G_{\Hyp}(z^\unif)\right] \left[\left(1\mp \frac{1}{q^{(r-2)/2}}\right) z^{2 \ring} \pm \frac{1}{q^{(r-2)/2}} G_{\widetilde{Q}_\pm}(z^\unif)\right].
\]
Since hyperbolic and elliptic planes are of norm $2\ring$, any generating function like $G_{\Hyp}(z^\unif)$ or $G_{\widetilde{Q}_\pm}(z^\pi)$ is equal to a generating function of the form $G(z^{2 \unif})$.
Thus, when we multiply out the above expression and use Remark \ref{Zachary}, we obtain our first relation, $G_{Q_\pm}(z)= \left(1\mp\frac{1}{q^{r/2}}\right)z^{2\ring} \pm \frac{1}{q^{r/2}} G_{Q_\pm}(z^\unif)$.  We apply this relation to its own second term to get
\[
G_{Q_\pm}(z)= \left(1\mp\frac{1}{q^{r/2}}\right)\left(z^{2\ring} \pm \frac{1}{q^{r/2}} z^{2\pi\ring}\right) + \frac{1}{q^r} G_{Q_\pm}(z^{\unif^2}),
\]
and then use Lemma \ref{Heinrich} to show that the head of the $p$-adic generating function of $Q_\pm$ is exactly what we claim it to be, and to get $\Ig(G_{Q_\pm})$ in terms of $\Ig(H_{Q_\pm})$, which we calculate with Lemma \ref{Stevie}.
\end{proof}
When $p$ is odd, for any given rank there are only two unimodular forms of that rank up to equivalence, one of each discriminant.
\begin{proposition}\label{Jason}
Let $p$ be odd and $r$ be a positive integer.  Then up to equivalence there are two distinct unimodular forms of rank $r$.  If $r$ is even, these are $\Hyp^{r/2}$ of discriminant $(-1)^{r/2}$ and $\Ell\oplus \Hyp^{(r-2)/2}$ of discriminant $(-1)^{r/2} \alpha$.  If $r$ is odd, these are $\Sq \oplus \Hyp^{(r-1)/2}$ of discriminant $(-1)^{(r-1)/2}$ and $\alpha \Sq \oplus \Hyp^{(r-1)/2}$ of discriminant $(-1)^{(r-1)/2} \alpha$.
\end{proposition}
\begin{proof}
In \cite[\S 92:1a]{O'Meara-1963}, it is shown that (up to equivalence) there are precisely two unimodular quadratic forms of any given rank, and we have exhibited two that are inequivalent since they have different discriminants.
\end{proof}
We may now express the $p$-adic generating functions of these unimodular forms for odd $p$ in terms of their discriminants.  We also calculate the values of $I_a(r,d)$ in Table \ref{Gary} that are used in Theorem \ref{Odilia} to express the local zeta functions for quadratic polynomials over $\ring$ when $p$ is odd.  Recall from Section \ref{Winston} that $I_a(r,d)=\Ig(z^a H_Q(z))$ where $Q$ is a unimodular quadratic form of rank $r$ and discriminant $d$ and $H_Q(z)$ is the head of its $p$-adic generating function.
\begin{lemma}\label{Ophelia}
Let $p$ be odd, let $Q$ be a quadratic form over $R$ of rank $r$ and discriminant $d$, and let $a\in\ring$.
\renewcommand{\labelitemi}{$\blacktriangleright$}
\begin{itemize}
\item If $r$ is even, then the head of the $p$-adic generating function of $Q$ is
\[
H_Q(z) = \left(1 - \frac{\eta((-1)^{r/2} d)}{q^{r/2}}\right) \left(z^{\ring}+\frac{\eta((-1)^{r/2} d)}{q^{r/2}} z^{\unif\ring}\right),
\]
\renewcommand{\labelitemii}{$\bullet$}
\begin{itemize}
\item and if $\pi\mid a$, then
\[
\Ig(z^a H_Q(z)) = \left(1 - \frac{\eta((-1)^{r/2} d)}{q^{r/2}}\right) \left(1+\frac{\eta((-1)^{r/2} d)}{q^{r/2}} t\right) \igrp,
\]
\item but if $\pi\nmid a$, then
\[
\Ig(z^a H_Q(z))=\left(1 - \frac{\eta((-1)^{r/2} d)}{q^{r/2}}\right) \left(\igrp + \frac{\eta((-1)^{r/2} d)}{q^{r/2}}\right).
\]
\end{itemize}
\item If $r$ is odd, then the head of the $p$-adic generating function of $Q$ is
\[
H_Q(z)= z^\ring -\frac{1}{q^r} z^{\unif \ring} + \frac{\eta((-1)^{(r-1)/2} d)}{q^{(r+1)/2}} \sum_{t \in \teich} \eta(t) z^{t+\unif \ring},
\]
\begin{itemize}
\item and if $\pi\mid a$, then
\[
\Ig(z^ a H_Q(z)) = \left(1-\frac{t}{q^r} \right) \left( \igr\right),
\]
\item but if $\pi\nmid a$, then
\[
\Ig(z^a H_Q(z)) =\left(1 + \frac{\eta(a(-1)^{(r+1)/2} d)}{q^{(r+1)/2}} t\right) \igrp -\frac{1}{q^r} -\frac{\eta(a(-1)^{(r+1)/2} d)}{q^{(r+1)/2}}.
\]
\end{itemize}
\end{itemize}
\renewcommand{\labelitemi}{$\bullet$}
\renewcommand{\labelitemii}{$-$}
In any case, the $p$-adic generating function satisfies $G_Q(z)=H_Q(z)+\frac{1}{q^r} G_Q(z^{\unif^2})$, so that the Igusa local zeta function for $Q$ is
\[
\Ig(G_Q(z))=\frac{\Ig(H_Q(z))}{1-t^2/q^r}.
\]
\end{lemma}
\begin{proof}
The calculation of the head of the $p$-adic generating function for the even rank case follows directly from Proposition \ref{Jason} and Lemma \ref{Rebecca}.
If $Q$ has odd rank, then Proposition \ref{Jason} tells us that it equivalent to the direct sum of a quadratic form $Q_1$ of rank $r-1$ and discriminant $1$ with a quadratic form $Q_2$ of rank $1$ and discriminant $d$ (viz., the form $d \Sq$).  By the even rank case of this lemma, and by Lemma \ref{Beatrice}, these forms have generating functions
\begin{align*}
G_{Q_1}(z) & = \left(1-\frac{1}{\sigma}\right) \left(z^\ring+\frac{1}{\sigma} z^{\unif \ring}\right) + \frac{1}{\sigma^2} G_{Q_1}(z^{\unif^2}) \\
G_{Q_2}(z) & =z^\ring -\frac{1}{q} z^{\unif \ring} + \frac{1}{q} \sum_{t \in \teich} \eta(d t) z^{t+\unif \ring} + \frac{1}{q} G_{Q_2}(z^{\unif^2}),
\end{align*}
where $\sigma=\eta((-1)^{(r-1)/2})q^{(r-1)/2}$.
We use the sum-product rule to obtain $G_Q(z)$ as the product $G_{Q_1}(z) G_{Q_2}(z)$, and we apply Remark \ref{Zachary} when multiplying out the right hand sides to obtain
\[
G_Q(z) = z^\ring -\frac{1}{q \sigma^2} z^{\unif \ring} + \frac{1}{q \sigma} \sum_{t \in \teich} \eta(d t) z^{t+\unif \ring} + \frac{1}{q \sigma^2} G_Q(z^{\unif^2}),
\]
and in view of Lemma \ref{Heinrich}, this shows that the head of the $p$-adic generating is exactly what we claim it to be.  We use Lemma \ref{Stevie} to compute $\Ig(z^c H_Q(z))$ in each case, and the relation between $\Ig(G_Q(z))$ and $\Ig(H_Q(z))$ comes from Lemma \ref{Heinrich}.
\end{proof}
For the rest of this section, we assume that $p=2$, and examine the unimodular quadratic forms that can arise in $2$-adic fields.
\begin{proposition}\label{Kurt}
Let $p=2$.  If $r$ is even, then any unimodular quadratic form of rank $r$ is equivalent to one of the following:
\begin{enumerate}[(i).]
\item $\Hyp^{r/2}$,
\item $\Pl(a,b) \oplus \Hyp^{r/2-1}$,
\item\label{Vladimir} $\Pl(a,b) \oplus \Pl(\pi^j,0) \oplus \Hyp^{r/2-2}$, or
\item\label{Yuri} $\Pl(a,b) \oplus \Pl(\pi^j,4\xi) \oplus \Hyp^{r/2-2}$,
\end{enumerate}
where $\Pl(a,b)$ is as defined in Section \ref{Jules}, and $a,b \in R$ with $\val(a) \leq \min\{\val(2),\val(b)/2\}$, and $b$ equal to either $0$, or $4 \xi$, or else $b$ is an element with $\val(b)$ odd and less than $\val(4)$, and $\val(a) < j \leq \val(2)$ with $\val(a)+j$ odd in cases \eqref{Vladimir} and \eqref{Yuri}.

If $r$ is odd, then any unimodular quadratic form of rank $r$ is equivalent to one of the following:
\begin{enumerate}[(i).]
\item $a \Sq \oplus \Hyp^{(r-1)/2}$,
\item $a \Sq \oplus \Pl(\pi^j,0) \oplus \Hyp^{(r-3)/2}$, or
\item $a \Sq \oplus \Pl(\pi^j,4\xi) \oplus \Hyp^{(r-3)/2}$,
\end{enumerate}
where $a \in \ringunits$ and $0 < j \leq \val(2)$.
\end{proposition}
\begin{proof}
This is a consequence of \cite[\S 93:17--18]{O'Meara-1963}, where we note that $\xi$ can be always be replaced with $-\xi$, for the ratio $(1+4\xi)/(1-4\xi) \equiv 1 \pmod{8}$, and so is a square.  Thus $1-4\xi$ is a nonsquare because $1+4\xi$ is.
\end{proof}
\begin{corollary}\label{Frank}
Suppose that $p=2$ and does not ramify in $\ring$.
If $r$ is even, then any unimodular quadratic form of rank $r$ is equivalent to one of the following:
\begin{enumerate}[(i).]
\item $\Hyp^{r/2}$,
\item $\Ell \oplus \Hyp^{r/2-1}$,
\item $a \Sq \oplus b \Sq \oplus \Hyp^{r/2-1}$, or
\item $a \Sq \oplus b \Sq \oplus \Ell \oplus \Hyp^{r/2-2}$,
\end{enumerate}
where $a, b \in \ringunits$ and are both congruent to $1$ modulo $2$.

If $r$ is odd, then any unimodular quadratic form of rank $r$ is equivalent to one of the following:
\begin{enumerate}[(i).]
\item $a \Sq \oplus \Hyp^{(r-1)/2}$,
\item $a \Sq \oplus \Ell \oplus \Hyp^{(r-3)/2}$,
\end{enumerate}
where $a \in \ringunits$ and is congruent to $1$ modulo $2$.
\end{corollary}
\begin{proof}
This follows from Proposition \ref{Kurt}.
In Section \ref{Jules}, we noted that $\Pl(a,b)$ is always equivalent to $\Hyp$ (when $\val(a)=1$ and $b=0$), $\Ell$ (when $\val(a)=1$ and $b=4\xi$), or $u \Sq \oplus v \Sq$ with $u$ and $v$ units (when $\val(a)=0$).
Furthermore we may take the units $u$ and $v$ to be $1$ modulo $2$ by scaling $x$ and $y$, since every unit is a square modulo $2$.
\end{proof}
When $p=2$ and is not ramified in $\ring$, all unimodular forms have norm $\ring$, $2\ring$, or $0$.  Those with norm $2 \ring$ and $0$ are covered by Lemma \ref{Rebecca} above (one uses $Q_+$ with $r=0$ to get the form $0$).
So we now calculate the $p$-adic generating functions and local zeta functions for unimodular forms of norm $\ring$ over unramified $2$-adic fields.
\begin{lemma}\label{Christine}
Suppose that $p=2$ and does not ramify in $\ring$.  Let $a \in \ringunits$.
Let $r \in \N$ be odd, let $Q_+=a \Sq \oplus \Hyp^{(r-1)/2}$, and if $r \geq 3$, then let $Q_-=a \Sq \oplus \Ell \oplus \Hyp^{(r-3)/2}$.  The head of the $p$-adic generating function of $Q_\pm$ is
\[
H_{Q_\pm}(z)=\left(1\mp \frac{1}{q^{(r-1)/2}}\right) \left[z^\ring \pm \frac{1}{q^{(r+1)/2}} \sum_{\tau \in \teich} z^{a \tau+4\ring}\right] + \frac{2}{q^{r+1}} \sums{\tau \in \teichu \\ s \in S} z^{a\tau (1+4 s)+ 8 \ring},
\]
with
\[
\Ig(H_{Q_\pm}(z))=\left(1-\frac{t^2}{q^r} \pm \frac{t^2-t}{q^{(r+1)/2}}\right) \igrp.
\]
The $p$-adic generating function satisfies $G_{Q_\pm}(z)=H_{Q_\pm}(z)+\frac{1}{q^r} G_{Q_\pm}(z^4)$.
The Igusa local zeta function for $Q_\pm$ is
\[
\Ig(G_{Q_\pm}(z)) = \frac{\Ig(H_{Q_\pm}(z))}{1-t^2/q^r}.
\]
\end{lemma}
\begin{proof}
Let $\widetilde{Q}_+=\Hyp^{(r-1)/2}$ and when $r\geq 3$, let $\widetilde{Q}_-=\Ell\oplus\Hyp^{(r-3)/2}$.  Let $G_{\widetilde{Q}_\pm}(z)$ and $G_{a\Sq}(z)$ be the respective $p$-adic generating functions of $\widetilde{Q}_\pm$ and $a\Sq$.
By the sum-product rule, the $p$-adic generating function for $Q_\pm$ is $G_{Q_\pm}(z)=G_{a \Sq}(z) G_{\widetilde{Q}_\pm}(z)$. Taking the values of these latter two generating functions are described in Lemmata \ref{Diana} and \ref{Rebecca}, we see that $G_{Q_\pm}(z)$ is 
\[
\left(\frac{2}{q^2} \sums{\tau \in \teichu \\ s \in S} z^{a\tau (1+4 s)+ 8 \ring} + \frac{1}{q} G_{a \Sq}(z^4)\right)\!\! \left[\! \left(1\mp \frac{1}{q^{(r-1)/2}}\right)\!\!\left(z^{2 \ring}\pm \frac{1}{q^{(r-1)/2}} z^{4 \ring}\right)\!\!+\! \frac{1}{q^{r-1}} G_{\widetilde{Q}_\pm}(z^4)\right]\! .
\]
One can multiply out and simplify the products using Remark \ref{Zachary}, keeping in mind that $G_{\widetilde{Q}_{\pm}}(z^4)$ is equal to $G(z^8)$ for some $G(z) \in \limgroupring$ since $\widetilde{Q}_\pm$ is of norm $2\ring$ or $0$.
When one does this, and coalesces $\sum_{\tau \in \teich} z^{\tau+2\ring}$ into $q z^\ring$ (see Remark \ref{Anne}), one obtains
\[
\left(1\mp \frac{1}{q^{(r-1)/2}}\right) \left[z^\ring \pm \frac{1}{q^{(r+1)/2}} \sum_{\tau \in \teich} z^{a\tau+4\ring}\right] + \frac{2}{q^{r+1}} \sums{\tau \in \teichu \\ s \in S} z^{a\tau (1+4 s)+ 8 \ring} + \frac{1}{q^r} G_{Q_\pm}(z^4), 
\]
which, by Lemma \ref{Heinrich} makes the head what it is claimed to be.  We use Lemma \ref{Stevie} to calculate that $\Ig(H_{Q_\pm}(z))$ is
\[
\left(1\mp \frac{1}{q^{(r-1)/2}}\right) \left[\igrp \pm \frac{1}{q^{(r+1)/2}} \left(t^2\igrp+\sum_{\tau \in \teichu} 1\right)\right] + \frac{2}{q^{r+1}} \sums{\tau \in \teichu \\ s \in S} 1,
\]
and since $\card{\teichu}=q-1$ and $\card{S}=q/2$, this is
\[
\left(1\mp \frac{1}{q^{(r-1)/2}}\right) \left(1 \pm  \frac{t^2}{q^{(r+1)/2}} \right) \igrp  + \left(\pm \frac{1}{q^{(r+1)/2}} - \frac{1}{q^r}\right)(q-1) + \frac{q-1}{q^r},
\]
which is
\[
\left[\left(1\mp \frac{1}{q^{(r-1)/2}}\right) \left(1 \pm  \frac{t^2}{q^{(r+1)/2}} \right)\pm \frac{q-t}{q^{(r+1)/2}} \right] \igrp,
\]
which simplifies to give the value of $\Ig(H_{Q_\pm}(z))$ that we were to prove.  The relation between $\Ig(G_{Q_\pm}(z))$ and $\Ig(H_{Q_\pm}(z))$ comes from Lemma \ref{Heinrich}.
\end{proof}
\begin{lemma}\label{Eveline}
Suppose that $p=2$ and does not ramify in $\ring$  Let $a, b \in \ringunits$ with $a\equiv b \equiv 1 \pmod{2}$.
Let $r \in \N$ be even, let $Q_+=a \Sq \oplus b \Sq \oplus \Hyp^{(r-2)/2}$ for $r \geq 2$, and let $Q_-=a \Sq \oplus b \Sq \oplus \Ell\oplus\Hyp^{(r-4)/2}$ for $r \geq 4$.

If $4 \mid a+b$, let $\sigma =(-1)^{\Tr((a+b)/(4 a))}$, and then the head of the $p$-adic generating function for $Q_\pm$ is
\[
H_{Q_\pm}(z) = z^\ring \mp \frac{1}{q^{r/2}} z^{2 \ring} \pm \frac{1}{q^{r/2}}  z^{4 \ring} -\frac{\sigma+1}{q^r} z^{4\ring} + \frac{\sigma}{q^r} z^{8 \ring},
\]
with
\[
\Ig(H_{Q_\pm}(z)) = \left(1 -\frac{t^2}{q^r} \pm \frac{t^2-t}{q^{r/2}} + \frac{\sigma (t^3-t^2)}{q^r}\right) \igrp.
\]

If $4 \nmid a+b$, then the head of the $p$-adic generating function for $Q_\pm$ is
\begin{multline*}
H_{Q_\pm}(z) = \left(1\mp \frac{1}{q^{(r-2)/2}}\right) \left(z^\ring \pm \frac{1}{q^{r/2}} z^{2 \ring}\right) \\
\pm \frac{2}{q^{(r+2)/2}} \sums{\tau \in \teichu \\ s \in S} z^{\tau\left(a + \frac{4}{a+b} s\right) + 4 \ring}
+ \frac{2}{q^{r+1}} \sums{\tau\in\teichu \\ s \in S} z^{\tau \left(a+b+4 s\right)+8\ring},
\end{multline*}
with
\[
\Ig(H_{Q_\pm}(z)) = \left(1-\frac{t^2}{q^r}\right) \igrp.
\]

In all cases, the $p$-adic generating function satisfies $G_{Q_\pm}(z)=H_{Q_\pm}(z)+\frac{1}{q^r} G_{Q_\pm}(z^4)$.
The Igusa local zeta function for $Q_\pm$ is
\[
\Ig(G_{Q_\pm}(z)) = \frac{\Ig(H_{Q_\pm}(z))}{1-t^2/q^r}.
\]
\end{lemma}
\begin{proof}
Let $\widetilde{Q}_+=\Hyp^{(r-2)/2}$ when $r\geq 2$, and let $\widetilde{Q}_-=\Ell\oplus \Hyp^{(r-4)/2}$ when $r \geq 4$.
Let $G_{\widetilde{Q}_\pm}(z)$ and $G_{a\Sq\oplus b\Sq}(z)$ be the respective $p$-adic generating functions of $\widetilde{Q}_\pm$ and $a\Sq\oplus b\Sq$.
By the sum-product rule, the $p$-adic generating function for $Q_\pm$ is $G_{Q_\pm}(z)=G_{a \Sq\oplus b\Sq}(z) G_{\widetilde{Q}_\pm}(z)$.
The values of these latter two generating functions are described in Lemmata \ref{Earl} and \ref{Rebecca}.

First consider the case where $4 \mid a+b$, so that $G_{Q_\pm}(z)$ is
\[
z^\ring-\frac{1}{q} z^{2\ring} + \frac{q-1-\sigma}{q^2} z^{4 \ring} + \frac{\sigma}{q^2} z^{8 \ring} + \frac{1}{q^2} G_{a \Sq\oplus b \Sq}(z^4)
\]
times
\[
\left(1\mp \frac{1}{q^{(r-2)/2}}\right) \left(z^{2 \ring}\pm \frac{1}{q^{(r-2)/2}} z^{4 \ring}\right) + \frac{1}{q^{r-2}} G_{\widetilde{Q}_\pm}(z^4).
\]
One can multiply out and simplify the products using Remark \ref{Zachary}, keeping in mind that $G_{\widetilde{Q}_{\pm}}(z^4)$ is equal to $G(z^8)$ for some $G(z) \in \limgroupring$ since $\widetilde{Q}_\pm$ is of norm $2\ring$ or $0$.
When one does this, one obtains
\[
z^\ring \mp \frac{1}{q^{r/2}} z^{2 \ring} \pm \frac{1}{q^{r/2}}  z^{4 \ring} -\frac{\sigma+1}{q^r} z^{4\ring} + \frac{\sigma}{q^r} z^{8 \ring} +\frac{1}{q^r} G_{Q_\pm}(z^4),
\]
which, by Lemma \ref{Heinrich} makes the head what it is claimed to be.
We use Lemma \ref{Stevie} to calculate 
\[
\Ig(H_{Q_\pm}(z)) = \left(1 \mp \frac{t}{q^{r/2}}  \pm \frac{t^2}{q^{r/2}}  -\frac{(\sigma+1) t^2}{q^r} + \frac{\sigma t^3}{q^r} \right) \igrp,
\]
which simplifies to the claimed value.

In the case where $4 \nmid a+b$, we use the appropriate generating function for $a \Sq\oplus b\Sq$ from Lemma \ref{Earl} to see that $G_{Q_\pm}(z)$ is
\[
\frac{2}{q^2} \sums{\tau \in \teichu \\ s \in S} z^{\tau\left(a + \frac{4}{a+b} s\right) + 4 \ring} + \frac{2}{q^3} \sums{\tau\in\teichu \\ s \in S} z^{\tau \left(a+b+4 s\right)+8\ring} + \frac{1}{q^2} G_{a \Sq\oplus b\Sq}(z^4)
\]
times
\[
\left(1\mp \frac{1}{q^{(r-2)/2}}\right) \left(z^{2 \ring}\pm \frac{1}{q^{(r-2)/2}} z^{4 \ring}\right) + \frac{1}{q^{r-2}} G_{\widetilde{Q}_\pm}(z^4).
\]
One can multiply out and simplify the products using Remark \ref{Zachary}, keeping in mind that $G_{\widetilde{Q}_{\pm}}(z^4)$ is equal to $G(z^8)$ for some $G(z) \in \limgroupring$ since $\widetilde{Q}_\pm$ is of norm $2\ring$ or $0$.
When one does this, and coalesces cosets (see Remark \ref{Anne}), one obtains
\begin{multline*}
\left(1\mp \frac{1}{q^{(r-2)/2}}\right) \left(z^\ring \pm \frac{1}{q^{r/2}} z^{2 \ring}\right) 
\pm \frac{2}{q^{(r+2)/2}} \sums{\tau \in \teichu \\ s \in S} z^{\tau\left(a + \frac{4}{a+b} s\right) + 4 \ring}
\\+ \frac{2}{q^{r+1}} \sums{\tau\in\teichu \\ s \in S} z^{\tau \left(a+b+4 s\right)+8\ring} +\frac{1}{q^r} G_{Q_{\pm}}(z^4),
\end{multline*}
which, by Lemma \ref{Heinrich} makes the head what it is claimed to be.
We use Lemma \ref{Stevie} to calculate that $\Ig(H_{Q_\pm}(z))$ is
\[
\left(1\mp \frac{1}{q^{(r-2)/2}}\right) \left(1 \pm \frac{t}{q^{r/2}} \right) \igrp \pm \frac{2}{q^{(r+2)/2}} \sums{\tau \in \teichu \\ s \in S} 1 + \frac{2}{q^{r+1}} \sums{\tau\in\teichu \\ s \in S} t,
\]
and since $\card{\teichu}=q-1$ and $\card{S}=q/2$, this is
\[
\left(1\mp \frac{1}{q^{(r-2)/2}}\right) \left(1 \pm \frac{t}{q^{r/2}} \right) \igrp \pm \frac{q-1}{q^{r/2}} + \frac{(q-1)t}{q^r},
\]
which is
\[
\left[\left(1\mp \frac{1}{q^{(r-2)/2}}\right) \left(1 \pm \frac{t}{q^{r/2}} \right) \pm \frac{q-t}{q^{r/2}} + \frac{(q-t)t}{q^r} \right] \igrp,
\]
which simplifies to give the value of $\Ig(H_{Q_\pm}(z))$ that we were to prove.

In each case, the relation between $\Ig(G_{Q_\pm}(z))$ and $\Ig(H_{Q_\pm}(z))$ comes from Lemma \ref{Heinrich}.
\end{proof}

\subsection{Addition Rules for Unimodular Quadratic Forms over Unramified $2$-Adic Fields}\label{Leonard}

In order to use Theorem \ref{Ursula} to calculate the local zeta function of a quadratic polynomial over an unramified $2$-adic ring, one needs to know how to express the direct sum of any pair of unimodular forms on Table \ref{Melissa} as another form on Table \ref{Melissa}.  In Lemma \ref{Eric} we proposed ``addition rules'' that make this possible.  We prove Lemma \ref{Eric} here.

The first addition rule is actually true for all $p$-adic fields.
\begin{lemma}
The quadratic form $\Ell\oplus \Ell$ is equivalent to the quadratic form $\Hyp\oplus \Hyp$.
\end{lemma}
\begin{proof}
When $p$ is odd, Proposition \ref{Jason} shows that up to equivalence the only forms of rank $4$ are $\Hyp\oplus\Hyp$ and $\Ell\oplus\Hyp$.  When $p=2$, consult Proposition \ref{Kurt}, and recall from Section \ref{Jules} that if $\Pl(a,b)$ has norm $2\ring$, then it must be $\Hyp$ or $\Ell$, and so up to equivalence the only forms or rank $4$ and norm $2 \ring$ are $\Hyp\oplus\Hyp$ and $\Ell\oplus\Hyp$.  Of the two candidates, only $\Hyp\oplus\Hyp$ has the same discriminant as $\Ell\oplus\Ell$.
\end{proof}
We prove the second addition rule specifically for the unramified $2$-adic case.
\begin{lemma}
Suppose that $p=2$ and does not ramify in $\ring$.
Suppose that $Q(x,y,z)=a x^2+b y^2+c z^2$ with $a,b,c \in \ringunits$.  If there are $d,e,f \in \ring$ with $Q(d,e,f) \equiv -a b c \pmod{8}$, then $Q(x,y,z)$ is equivalent to  $(-a b c) \Sq \oplus \Hyp$.  Otherwise $Q(x,y,z)$ is equivalent to $(-a b c)(1+4\xi) \Sq \oplus \Ell$.
\end{lemma}
\begin{proof}
Since $Q$ has rank $3$ and norm $\ring$, Corollary \ref{Frank} shows that it must be equivalent either to $(-a b c) \Sq \oplus \Hyp$ or $(-a b c)(1+4 \xi)\Sq \oplus \Ell$, where the coefficient of the $\Sq$ term has been chosen in view of the fact that the discriminant must be preserved.
We note that $(-a b c) \Sq \oplus \Hyp$ obviously represents $-a b c$.
We claim that $(-a b c)(1+4\xi)\Sq\oplus\Ell$ does not represent $-a b c \pmod{8}$.  For if it did, then since $\Ell$ represents only zero and elements of odd valuation by Remark \ref{Gabriel}, we would need to have some $r \in \ring$ such that $f(r)=(-a b c)(1+4\xi) r^2+a b c$ is either zero or of odd valuation.  The valuation of $f(r)$ will be $0$ unless $r \equiv 1 \pmod{2}$.  Write $r=1+2 s$ with $s \in \ring$, so that $f(r) \equiv -4 a b c (\xi+s+s^2) \pmod{8}$, and then note that $\Tr(\xi+s+s^2) \equiv 1 \pmod{2}$, so that $\twoval(f(r))=2$.
\end{proof}

\appendix

\section{Proof of Lemma \ref{Earl}}\label{Karen}

Throughout this section, we assume that $p=2$ and does not ramify in $\ring$.
We let $\teich$ be a set of Teichm\"uller representatives for $\resfield$ in $\field$, that is, $T$ contains all the $(q-1)$th roots of unity and $0$.
We let $\teichu=T\smallsetminus\{0\}$, and we let $S=\{\tau \in \teich: \Tr(\tau) \equiv 0 \pmod{\pi}\}$.  Note that $\card{S}=q/2$.

In order to prove Lemma \ref{Earl}, we need some preliminary results concerning the arithmetic of Teichm\"uller representatives.  We present these as technical lemmata, some of which are used here, and some of which find use later in Appendix \ref{Mordecai}.
\begin{lemma}\label{Richard}
Let $a \in R$, and suppose that $a_0 \in T$ with $a \equiv a_0 \pmod{2}$.  Then $a^q \equiv a_0 \pmod{4}$.  If $a$ is a unit, or if the residue field for $\ring$ has order $q \geq 4$, then $a^q \equiv a_0 \pmod{8}$.
\end{lemma}
\begin{proof}
Write $a=a_0+2 r$ for some $r \in R$, and consider the binomial expansion of $(a_0+2 r)^q$ modulo $8$, which is $a_0^q + 2 q a_0^{q-1} r + 2 q(q-1) a_0^{q-2} r^2$.  The last two terms always vanish modulo $4$, and they vanish modulo $8$ when $q \geq 4$.  If $q=2$ and $a$ is a unit, then $a_0=1$, and it is easy to check that $a^2 \equiv 1 \pmod{8}$ for every unit in $\Z_2$.
\end{proof}
\begin{lemma}\label{Elijah}
Suppose that $a, b, c \in T$ with $a+b \equiv c \pmod{2}$.
Then
\begin{align*}
a + b & \equiv c + 2 (a b)^{q/2} \\
& \equiv c + 2 a + 2 (a c)^{q/2} \pmod{4}.
\end{align*}
If the residue field for $\ring$ is of order $q \geq 4$, then
\begin{align*}
a + b & \equiv c + 2 (a b)^{q/2} + 4 (a b)^{q/4} (a^{q/2}+b^{q/2}) \\
& \equiv c + 2 a - 2 (a c)^{q/2} + 4 (a c)^{q/4} (a^{q/2}+c^{q/2}) \pmod{8}.
\end{align*}
\end{lemma}
\begin{proof}
We prove the asserted congruences modulo $8$ for the case when $q \geq 4$: these imply the asserted congruences modulo $4$, which are easy to show when $q=2$, since $T=\{0,1\}$.
By Lemma \ref{Richard}, $(a+b)^q \equiv c \pmod{8}$, so then $a + b  \equiv c +\left((a+b)- (a+b)^q \right) \pmod{8}$, and we use binomial expansion.  From the theorems of Kummer \cite{Kummer-1852} and Anton \cite{Anton-1869}, we know that the only binomial coefficients $\binom{q}{j}$ that do not vanish modulo $8$ are those with $j=0$, $q/4$, $q/2$, $3 q/4$, and $q$, and the values of these are $1$, $4$, $-2$, $4$, and $1$ modulo $8$, respectively.  This (and the fact that $a^q=a$ and $b^q=b$ since $a, b \in T$), suffice to prove the first congruence.

To prove the second congruence, we replace $b$ with $c-a$ and use the same techniques.  This does not change the value of our expression modulo $8$, for $b \equiv c-a \pmod{2}$, whence one easily shows that $b^2 \equiv (c-a)^2 \pmod{4}$, so that $b^{q/2} \equiv (c-a)^{q/2} \pmod{4}$.
\end{proof}
\begin{corollary}\label{Martin}
Suppose that $a \in \ring$ and $b, c \in T$ with $a+b \equiv c \pmod{2}$.
Then
\begin{align*}
a + b & \equiv c + a + a^q + 2 (a c)^{q/2} \pmod{4}.
\end{align*}
\end{corollary}
\begin{proof}
Write $a=a_0+2 a_1$ with $a_0 \in \teich$ and $a_1 \in \ring$.
Then $a_0+b\equiv c\pmod{2}$, so that Lemma \ref{Elijah} tells us that $a_0 + b  \equiv c + 2 a_0 + 2 (a_0 c)^{q/2} \pmod{4}$.  So $a+b \equiv c+ 2 a_0 + 2 a_1 + 2 (a_0 c)^{q/2} \pmod{4}$, which is the same as $c+a+a_0 + 2(a c)^{q/2} \pmod{4}$, which is the same as $c+a+a^q +2 (ac)^{q/2} \pmod{4}$ by Lemma \ref{Richard}.
\end{proof}
\begin{lemma}\label{Bartholomew}
Let $a \in R$.  If $a\not\equiv 0 \pmod{2}$, then as $b$ runs through $T$, the quantity $b^2+a b \pmod{2}$ runs through the values of $\{a^2 s \pmod{2}: s \in S\}$, taking each value twice.  If $a \equiv 0 \pmod{2}$, then as $b$ runs through $T$, the quantity $b^2+ a b \pmod{2}$ runs through the values of $\resfield$, taking each value once.
\end{lemma}
\begin{proof}
The claim when $a \equiv 0 \pmod{2}$ is clear because $b^2 + a b \equiv b^2 \pmod{2}$, and $x \mapsto x^2$ is a permutation of $\resfield$.  So we assume $a \not\equiv 0 \pmod{2}$ henceforth.  Write $c=a^{-1} b$, so that we are looking at the quantity $a^2 (c^2+c) \pmod{2}$ as $c$ runs through $a^{-1} T$, which is the same as $T$ modulo $2$.
As $c$ runs through $T$, the quantity $c^2+c \pmod{2}$ takes each value in $\resfield$ with zero trace twice since it cannot take any such value more than twice.
\end{proof}
\begin{lemma}\label{Eustace}
For $c \in \ring$, write $c S \pmod{2}$ to mean the set $\{c s \pmod{2}: s\in S\}$.  Then $c S \pmod{2}$ is a subgroup of the additive group of $\resfield$, and if $a, b\in \ringunits$ with $a\not\equiv b \pmod{2}$, we have $a S \pmod{2}+ b S\pmod{2} =\resfield$.
\end{lemma}
\begin{proof}
Let $\bar{c}$ be the reduction of $c$ modulo $2$.  If $\bar{c}=0$, then $c S \pmod{2}=\{0\}$, and if $\bar{c}\not=0$, then $c S \pmod{2}$ is the kernel of the $\F_2$-linear functional $x \mapsto \Tr(x/\bar{c})$, hence an additive subgroup of $\resfield$ of index $2$.  Changing $c$ modulo $2$ changes the functional, hence its kernel, and so $a S \pmod{2} + b S \pmod{2}$ must be the entire group $\resfield$.
\end{proof}
\begin{corollary}\label{Hildegard}
Let $a, b \in R$ with $a\not\equiv b\pmod{2}$.  Then as $(c,d)$ runs through $T^2$, the quantity $c^2+ a c + d^2 + b d \pmod{2}$ runs through the values of $\resfield$, with each value being taken $q$ times.
\end{corollary}
\begin{proof}
By Lemma \ref{Bartholomew}, if either $a$ or $b$ vanishes modulo $2$, then this is clear.
Otherwise, $(c^2 + a c,d^2+ b d) \pmod{2}$ runs through $a^2 S \pmod{2} \times b^2 S \pmod{2}$ with each value taken four times, and then Lemma \ref{Eustace} completes the proof.
\end{proof}

{\it Proof of Lemma \ref{Earl}:}  Once we determine the head, the rest follows immediately from Lemma \ref{Heinrich}.
It suffices to determine the head for $\Sq\oplus w \Sq$ for $w\in\ringunits$ with $w\equiv 1 \pmod{2}$, for we can use $w=b/a$ and then scale the generating functions by $a$ to get the desired forms, once we check that when $4\nmid a+b$ we have $4 a/(1+(b/a)) = 4 a^2 /(a+b) \equiv 4/(a+b) \pmod{8}$ and $4 a \equiv 4 \pmod{8}$ because $a \equiv 1 \pmod{2}$.

We can use use Corollary \ref{Hortense} to determine the head as 
\[
H(z)=\frac{1}{q^4} \sums{s_0,s_1,\tau_0,\tau_1 \in \teich \\ (s_0,\tau_0)\not=(0,0)} z^{(s_0+2 s_1)^2 + w (\tau_0+2 \tau_1)^2 + 8 \ring}.
\]
When $q=2$, we have $\teich=\{0,1\}$ and $S=\{0\}$, and the square of any unit is always $1 \pmod{8}$, so if we break the sum into the cases where $(s_0,\tau_0)=(1,0)$, $(0,1)$, and $(1,1)$, respectively, we get
\[
\frac{1}{8} \sum_{\tau_1 \in \teich} z^{1+4 w \tau_1^2+8\ring} + \frac{1}{8} \sum_{s_1 \in \teich} z^{4 s_1^2 +w+8\ring} + \frac{1}{4} z^{1+w+8\ring}.
\]
Then note that $4 w \equiv 4 \pmod{8}$, and that $\tau\mapsto\tau^2$ is a permutation of $\teich$ to get $(z^{1+4\ring} + z^{w+4\ring} + z^{1+w+8\ring})/4$, and it is not hard to show that this matches the two general forms to be proved (under their respective hypotheses) when one sets $q$ equal to $2$.

We assume $q \geq 4$ henceforth.
Write $w\equiv 1+2 w_1+4 w_2 \pmod{8}$ with $w_0, w_1, w_2 \in T$, and then note that
\[
(s_0+2 s_1)^2+w(\tau_0+2 \tau_1)^2 \equiv s_0^2 + \tau_0^2 + (w-1) \tau_0^2 + 4 (s_1^2+s_0 s_1) + 4 (\tau_1^2+\tau_0 \tau_1) \pmod{8}.
\]
Let us concern ourselves with the case when this value is $1 \pmod{2}$, that is, when $s_0^2+\tau_0^2 \equiv 1 \pmod{2}$.
Then by Lemma \ref{Elijah}, we have
\[
s_0^2 + \tau_0^2  \equiv 1 + 2 (\tau_0^2-\tau_0) + 4 \tau_0^{q/2} (\tau_0+1) \pmod{8},
\]
and since $s_0^2+\tau_0^2 \equiv 1 \pmod{2}$, we must have $s_0+\tau_0 \equiv 1 \pmod{2}$, so that $(s_0+2 s_1)^2+w(\tau_0+2 \tau_1)^2 \pmod{8}$ is
\[
1 + 2 \left(\left(\frac{w+1}{2}\right) \tau_0^2 - \tau_0\right) + 4 \tau_0^{q/2}(\tau_0+1) + 4 (s_1^2+(1-\tau_0) s_1) + 4 (\tau_1^2+\tau_0 \tau_1).
\]
As we let $(s_0,\tau_0)$ run through $(\teich)^2\smallsetminus\{(0,0)\}$, we see that $(s_0^2,\tau_0^2)$ runs through $(\teich)^2\smallsetminus\{(0,0)\}$, and we obtain $s_0^2+\tau_0^2 \equiv 1 \pmod{2}$ in $q$ different ways.
If $4\mid 1+w$, then $(w+1)/2 \equiv 0 \pmod{2}$, so then $(w+1)\tau_0^2/2+\tau_0 \pmod{2}$ runs through $\resfield$, taking each value once.
But if $4\nmid 1+w$, then Lemma \ref{Bartholomew} shows that $(w+1) \tau_0^2/2+\tau_0 \pmod{2}$ runs through the values of $\frac{2}{w+1} S \pmod{2}$, taking each value twice.
In either case, Corollary \ref{Hildegard} shows that for any given values of $s_0$ and $\tau_0$, the term $(s_1^2+(1-\tau_0) s_1) + (\tau_1^2+\tau_0 \tau_1)$ taken modulo $2$ runs through $\resfield$, taking each value $q$ times as $(s_1,\tau_1)$ runs through $T^2$.
Thus, if $4\mid 1+w$, we see that $(s_0+2 s_1)^2+w(\tau_0+2 \tau_1)^2 \pmod{8}$ runs through the values of the form $1+2 s+4 c$ with $s,c \in \teich$, taking each value $q$ times.
And if $4 \nmid 1+w$, we see that $(s_0+2 s_1)^2+w(\tau_0+2 \tau_1)^2 \pmod{8}$ runs through the values of the form $1+4 s/(w+1) + 4 c$ with $s \in S$ and $c \in T$, taking each value $2 q$ times.

Of course the value of $(s_0+2 s_1)^2+w(\tau_0+2 \tau_1)^2$ can be a unit not congruent to $1$ modulo $2$.  If $u \in T$ and we want to count the instances where $(s_0+2 s_1)^2+ w(\tau_0+2 \tau_1)^2 \equiv u \pmod{2}$, then just pick the unique $v \in T$ with $v^2 = u$, and then note there is bijection between the quadruples $(s_0,\tau_0,s_1,\tau_1)$ in our summation that make $(s_0+2 s_1)^2+w(\tau_0+2 \tau_1)^2$ congruent to $1$ modulo $2$ and the quadruples that make it congruent to $u$ modulo $2$: just scale the quadruple by $v$.  So we just multiply all the outputs that are $1$ modulo $2$ by $v^2=u$ to get the outputs that are $u$ modulo $2$.

Finally, we must consider the outputs of $(s_0+2 s_1)^2+w(\tau_0+2 \tau_1)^2$ that vanish modulo $2$, that is, when $s_0=\tau_0$ (which runs through $\teichu$ since our summation prohibits both from vanishing simultaneously).  Meanwhile $(s_1,\tau_1)$ runs through $\teich^2$.  Then we have
\begin{align*}
(s_0+2 s_1)^2+w(\tau_0+2 \tau_1)^2 \equiv (1+w) \tau_0^2 + 4 (s_1^2+\tau_0 s_1) + 4 (\tau_1^2+\tau_0 \tau_1) \\
\equiv (1+w) \tau_0^2 + 4 ((s_1+\tau_1)^2+\tau_0 (s_1+\tau_1)) \pmod{8}.
\end{align*}
For each value of $\tau_0 \in \teichu$, Lemma \ref{Bartholomew} shows that the term $(s_1+\tau_1)^2+\tau_0 (s_1+\tau_1) \pmod{2}$ runs through $\tau_0^2 S \pmod{2}$, taking each value $2 q$ times as $(s_1,\tau_1)$ runs through $T^2$.
And $\tau_0^2$ runs through $\teichu$ as $\tau_0$ runs through $\teichu$, so we get $2 q$ instances of each element of the form $\tau(1+w+4 s) \pmod{8}$ for $\tau \in \teichu$ and $s \in S$.
There are several cases to consider.
If $4\mid 1+w$ and $\Tr((1+w)/4) \equiv 0 \pmod{2}$, then our expression furnishes $2 q$ instances of each value that $4 \tau s \pmod{8}$ attains as $(\tau,s)$ runs through $\teichu \times S$, and since $0 \in S$, this means we get $(2 q)(q-1)=2 q^2-2 q$ instances of $0 \pmod{8}$ and $(2 q)(q/2-1)=q^2-2 q$ instances of $4 \tau \pmod{8}$ for each $\tau \in \teichu$.
If $4 \mid 1+w$ and $\Tr((1+w)/4) \equiv 1 \pmod{2}$, then our expression furnishes $2 q$ instances of each value that $4 \tau s \pmod{8}$ attains as $(\tau,s)$ runs through $\teichu \times (T\smallsetminus S)$, and so we get $q^2$ instances of each element of the form $4 \tau \pmod{8}$ with $\tau \in \teichu$.
If $4 \nmid 1+w$, then we still have $2 \mid 1+w$, and then our expression furnishes $2 q$ instances for each value that $\tau(1+w+4 s) \pmod{8}$ attains as $(\tau,s)$ runs through $\teichu \times S$.

When we put together all these counts and coalesce cosets (see Remark \ref{Anne}), we see that if $4 \mid 1+w$ and $\Tr((1+w)/4) \equiv 0 \pmod{2}$, then
the head of the $p$-adic generating function for $\Sq\oplus w \Sq$ is
\[
H_{\Sq\oplus w \Sq}(z)=z^\ring-\frac{1}{q} z^{2\ring} + \frac{q-2}{q^2} z^{4 \ring} + \frac{1}{q^2} z^{8 \ring},
\]
and if $4 \mid 1+w$ with $\Tr((1+w)/4) \equiv 1 \pmod{2}$, then
\[
H_{\Sq\oplus w \Sq}(z)=z^\ring-\frac{1}{q} z^{2\ring} + \frac{1}{q} z^{4 \ring} - \frac{1}{q^2} z^{8 \ring},
\]
and if $4 \nmid 1+w$, then
\[
H_{\Sq\oplus w \Sq}(z)=\frac{2}{q^2} \sums{u \in \teichu \\ s \in S} z^{u(1+4 s/(1+w))+4 \ring} + \frac{2}{q^3} \sums{\tau\in\teichu \\ s \in S} z^{\tau (1+w+4 s)+8\ring}. \qedhere
\]

\section{Calculations for Table \ref{Violet}}\label{Mordecai}

In this section we calculate the values $I(Q_0,Q_1,Q_2)$ in Table \ref{Violet}.
Recall from Section \ref{Wilbur} that if $Q_0$, $Q_1$, and $Q_2$ are unimodular quadratic forms, and if we let $G_{Q_i}(z)$ and $H_{Q_i}(z)$ denote respectively the $p$-adic generating generating function of $Q_i$ and the head of said generating function, then $I(Q_0,Q_1,Q_2)=\Ig(H_{Q_0}(z) G_{Q_1}(z^2) G_{Q_2}(z^4))$.

Throughout this section, we assume that $p=2$ and does not ramify in $\ring$.
We let $\teich$ be a set of Teichm\"uller representatives for $\resfield$ in $\field$, that is, $T$ contains all the $(q-1)$th roots of unity and $0$.
We let $\teichu=T\smallsetminus\{0\}$, and we let $S=\{\tau \in \teich: \Tr(\tau) \equiv 0 \pmod{\pi}\}$.  Note that $\card{S}=q/2$.
We always use $Q_0$, $Q_1$, and $Q_2$ to denote unimodular quadratic forms, and if $Q$ is a quadratic form over $\ring$, then $G_Q(z)$ will denote the $p$-adic generating function for $Q$, and $H_Q(z)$ will denote the head of $G_Q(z)$.

We use the same shorthand for unimodular quadratic forms that is used in Table \ref{Violet}: $\Planes(+)$ means a direct sum of hyperbolic planes (with the correct number needed to achieve a particular rank, if specified) and $\Planes(-)$ is the direct sum of a single elliptic plane and some number of hyperbolic planes (again, achieving a particular rank, if specified).  For example, if we say that $\rank(Q_i)=r_i$ and $Q_i=a\Sq\oplus b\Sq\oplus\Planar{i}$, then we mean that either $Q_i=a\Sq\oplus b\Sq\oplus \Hyp^{(r_i-2)/2}$ (if $\pm_i=+$) or $Q_i=a\Sq\oplus b\Sq\oplus \Ell\oplus\Hyp^{(r_i-4)/2}$ (if $\pm_i=-$).

When calculating $\Ig(H_{Q_0}(z) G_{Q_1}(z^2) G_{Q_2}(z^4))$ for unimodular forms $Q_0$, $Q_1$, and $Q_2$, the first thing to note from Corollary \ref{Hortense} is that $H_{Q_0}(z)$ is $8$-uniform.
Therefore we may make extensive use of Remark \ref{Zachary} and Lemma \ref{Una} in our calculations.
Indeed, throughout this section, when $F(z) \in \limgroupring$, we use $\widehat{F}(z)$ as shorthand for the $2$-uniformization of $F(z)$ and $\widetilde{F}(z)$ for the $4$-uniformization.
Note that by $\widehat{F}(z^{2^j})$, we mean that one should first $2$-uniformize $F$, and then scale by replacing every instance of $z$ with $z^{2^j}$.
If one wants to perform the operations in the opposite order, then one arrives at the same function if one first scales $F(z)$ to obtain $F(z^{2^j})$, and then uniformizes $F(z^{2^j})$ modulo $2^{j+1}$.
The same convention and principle holds for $\widetilde{F}(z)$, and so $\widetilde{F}(z^2)$ and $\widehat{F}(z^4)$ are both $8$-uniform.
\begin{lemma}\label{Samantha}
We have
\[H_{Q_0}(z) G_{Q_1}(z^2) G_{Q_2}(z^4)=H_{Q_0}(z) G_{Q_1}(z^2) \widehat{H}_{Q_2}(z^4)=H_{Q_0}(z) \widetilde{G}_{Q_1}(z^2) \widehat{G}_{Q_2}(z^4).\]
\end{lemma}
\begin{proof}
Since Corollary \ref{Hortense} shows that $H_{Q_0}(z)$ is $8$-uniform, this follows from Lemma \ref{Una}.
\end{proof}
The various uniformizations and related quantities on Tables \ref{Rachel} and \ref{Edith} will be useful.
They are easy to calculate from $G_Q$ and $H_Q$ as given in Lemmata \ref{Rebecca}, \ref{Christine}, and \ref{Eveline}.
Now we are ready to examine $H_{Q_0}(z) G_{Q_1}(z^2) G_{Q_2}(z^4)$ according to various cases.
\begin{table}[ht]
\caption{Uniformizations of Generating Functions for a Unimodular Form $Q$}\label{Rachel}
\begin{center}
\vspace{-3mm}
\begin{tabular}{|c|c|}
\multicolumn{2}{c}{$r=\rank(Q)$} \\
\multicolumn{2}{c}{$a \equiv 1 \pmod{2}$}\\
\hline
\multicolumn{2}{|c|}{$Q = \Planar{0}$} \\
\hline
$\norm(Q)$ & $\begin{cases} 2 R & \text{if $r > 0$} \\ 0 & \text{if $r=0$}\end{cases}$ \\
$\widetilde{G}_Q(z)$ & $\left(1\mp \frac{1}{q^{r/2}}\right) z^{2 \ring} \pm\frac{1}{q^{r/2}} z^{4\ring}$ \\
$\widetilde{H}_Q(z)$ & $\left(1\mp \frac{1}{q^{r/2}}\right) \left(z^{2 \ring}\pm\frac{1}{q^{r/2}} z^{4\ring}\right)$ \\
$\Ig(\widetilde{H}_Q(z))$ & $\left(1\mp \frac{1}{q^{r/2}}\right) \left(t \pm\frac{t^2}{q^{r/2}}\right) \igrp$ \\ 
$H_Q(z)-\widetilde{H}_Q(z)$ & $0$ \\
$\Ig(H_Q(z)-\widetilde{H}_Q(z))$ & $0$ \\
$\widehat{G}_Q(z)$ & $z^{2 \ring}$ \\
$\widehat{H}_Q(z)$ & $\left(1- \frac{1}{q^r}\right) z^{2\ring}$ \\
$\Ig(\widehat{H}_Q(z))$ & $\left(1- \frac{1}{q^r}\right) t \igrp$ \\
\hline
\hline
\multicolumn{2}{|c|}{$Q = a\Sq\oplus \Planar{0}$} \\
\hline
$\norm(Q)$ & $R$ \\
$\widetilde{G}_Q(z)$ & $\left(1\mp \frac{1}{q^{(r-1)/2}}\right) z^\ring \pm \frac{1}{q^{(r+1)/2}} \sum_{\tau \in \teich} z^{a \tau+4\ring}$ \\
$\widetilde{H}_Q(z)$ & $\left(1\mp \frac{1}{q^{(r-1)/2}}\right) z^\ring -\frac{1}{q^r} z^{4\ring} \pm \frac{1}{q^{(r+1)/2}} \sum_{\tau \in \teich} z^{a \tau+4\ring}$ \\
$\Ig(\widetilde{H}_Q(z))$ &  $\left(1-\frac{t^2}{q^r} \pm \frac{t^2-t}{q^{(r+1)/2}}\right) \igrp$ \\
$H_Q(z)-\widetilde{H}_Q(z)$ & $\frac{2}{q^{r+1}} \sums{\tau \in \teichu \\ s \in S} z^{a \tau (1+4 s)+ 8 \ring} - \frac{1}{q^r} \sums{\tau \in \teichu} z^{a \tau+ 4 \ring}$ \\
$\Ig(H_Q(z)-\widetilde{H}_Q(z))$ & $0$ \\
$\widehat{G}_Q(z)$ & $z^\ring$ \\
$\widehat{H}_Q(z)$ & $z^\ring - \frac{1}{q^r} z^{2\ring}$ \\
$\Ig(\widehat{H}_Q(z))$ & $\left(1-\frac{t}{q^r}\right) \igrp$ \\
\hline
\end{tabular}
\end{center}
\end{table}
\begin{table}[ht]
\caption{Uniformizations of Generating Functions for a Unimodular Form $Q$}\label{Edith}
\begin{center}
\vspace{-3mm}
\begin{tabular}{|c|c|}
\multicolumn{2}{c}{$r=\rank(Q)$} \\
\multicolumn{2}{c}{$a \equiv b \equiv 1 \pmod{2}$}\\
\hline
\multicolumn{2}{|c|}{$Q = a\Sq\oplus b\Sq \Planar{0}$ with $4 \mid a+b$ and $\sigma=(-1)^{\Tr((a+b)/(4 a))}$} \\
\hline
\hline
$\norm(Q)$ & $R$ \\
$\widetilde{G}_Q(z)$ & $z^\ring \mp \frac{1}{q^{r/2}} z^{2 \ring} \pm \frac{1}{q^{r/2}}  z^{4 \ring}$ \\
$\widetilde{H}_Q(z)$ & $ z^\ring \mp \frac{1}{q^{r/2}} z^{2 \ring} \pm \frac{1}{q^{r/2}}  z^{4 \ring} -\frac{1}{q^r} z^{4\ring}$ \\
$\Ig(\widetilde{H}_Q(z))$ & $\left(1 -\frac{t^2}{q^r} \pm \frac{t^2-t}{q^{r/2}}\right) \igrp$ \\
$H_Q(z)-\widetilde{H}_Q(z)$ & $\frac{\sigma}{q^r} \left(z^{8\ring}-z^{4\ring}\right)$\\
$\Ig(H_Q(z)-\widetilde{H}_Q(z))$ & $\frac{\sigma}{q^r} (t^3-t^2) \igrp$ \\
$\widehat{G}_Q(z)$ & $z^{\ring}$ \\
$\widehat{H}_Q(z)$ & $ z^{\ring} - \frac{1}{q^r} z^{2\ring}$ \\
$\Ig(\widehat{H}_Q(z))$ & $\left(1- \frac{t}{q^r} \right) \igrp$ \\
\hline
\hline
\multicolumn{2}{|c|}{$Q = a\Sq\oplus b\Sq \Planar{0}$ with $4 \nmid a+b$} \\
\hline
$\norm(Q)$ & $R$ \\
$\widetilde{G}_Q(z)$ & $\left(1\mp \frac{1}{q^{(r-2)/2}}\right) z^\ring \pm \frac{1}{q^{r/2}} z^{2 \ring} \pm \frac{2}{q^{(r+2)/2}} \sums{\tau \in \teichu \\ s \in S} z^{\tau\left(a + \frac{4}{a+b} s\right) + 4 \ring}$ \\
$\widetilde{H}_Q(z)$ & $\left(1\mp \frac{1}{q^{(r-2)/2}}\right) z^\ring \pm \frac{1}{q^{r/2}} z^{2 \ring} \pm \frac{2}{q^{(r+2)/2}} \sums{\tau \in \teichu \\ s \in S} z^{\tau\left(a + \frac{4}{a+b} s\right) + 4 \ring} -\frac{1}{q^r} z^{4 \ring}$ \\
$\Ig(\widetilde{H}_Q(z))$ & $\left(1-\frac{t^2}{q^r}\right) \igrp$ \\
$H_Q(z)-\widetilde{H}_Q(z)$ & $-\frac{1}{q^{r-1}} z^{2\ring} + \frac{1}{q^r} z^{4\ring} + \frac{2}{q^{r+1}} \sums{\tau\in\teichu \\ s \in S} z^{\tau \left(a+b+4 s\right)+8\ring}$ \\
$\Ig(H_Q(z)-\widetilde{H}_Q(z))$ & $0$ \\
$\widehat{G}_Q(z)$ & $z^{\ring}$ \\
$\widehat{H}_Q(z)$ & $ z^{\ring} - \frac{1}{q^r} z^{2\ring}$ \\
$\Ig(\widehat{H}_Q(z))$ & $\left(1- \frac{t}{q^r} \right) \igrp$ \\
\hline
\end{tabular}
\end{center}
\end{table}
\FloatBarrier
\begin{lemma}\label{Robert}
We have
\[
H_{Q_0}(z) G_{Q_1}(z^2) G_{Q_2}(z^4) z^{4\ring} = H_{Q_0}(z) G_{Q_1}(z^2) G_{\Sq}(z^4).
\]
\end{lemma}
\begin{proof}
We have $H_{Q_0}(z) G_{Q_1}(z^2) G_{Q_2}(z^4) z^{4\ring} = H_{Q_0}(z) G_{Q_1}(z^2) z^{4\ring}$, which in turn equals $H_{Q_0}(z) G_{Q_1}(z^2) \widehat{G}_{\Sq}(z^4)$ by Table \ref{Rachel}, and so equals $H_{Q_0}(z) G_{Q_1}(z^2) G_{\Sq}(z^4)$ by Lemma \ref{Samantha}.
\end{proof}
\begin{lemma}\label{Anthony}
If $\norm(Q_1)=\norm(Q_2)=\ring$, then $H_{Q_0}(z) G_{Q_1}(z^2) G_{Q_2}(z^4)$ is just $\widehat{H}_{Q_0}(z)$, which can be obtained from Table \ref{Rachel} or \ref{Edith}.
\end{lemma}
\begin{proof}
Lemma \ref{Samantha} allows us to replace $H_{Q_0}(z) G_{Q_1}(z^2) G_{Q_2}(z^4)$ by $H_{Q_0}(z) \widetilde{G}_{Q_1}(z^2) \widehat{G}_{Q_2}(z^4)$, which equals $H_{Q_0}(z) \widetilde{G}_{Q_1}(z^2) z^{4\ring}$ by Tables \ref{Rachel} and \ref{Edith}.
Then we may use Lemma \ref{Una} to replace $\widetilde{G}_{Q_1}(z^2) z^{4\ring}$ by $\widehat{G}_{Q_1}(z^2) z^{4 \ring}$, which equals $z^{2\ring} z^{4 \ring}=z^{2 \ring}$ by Tables \ref{Rachel} and \ref{Edith}.  So we need only calculate $H_{Q_0}(z) z^{2 \ring}$, which is $\widehat{H}_{Q_0}$ by Lemma \ref{Una}.
\end{proof}
\begin{lemma}\label{Andrew}
We have
\[
H_{Q_0}(z) G_{Q_1}(z^2) G_{Q_2}(z^4) z^{2\ring} = H_{Q_0}(z) G_{\Sq}(z^2) G_{\Sq}(z^4).
\]
\end{lemma}
\begin{proof}
We see that $H_{Q_0}(z) G_{Q_1}(z^2) G_{Q_2}(z^4) z^{2\ring}=H_{Q_0}(z) z^{2\ring}=\widehat{H}_{Q_0}(z)$, which in turn equals $ H_{Q_0}(z) G_{\Sq}(z^2) G_{\Sq}(z^4)$ by Lemma \ref{Anthony}.
\end{proof}
\begin{lemma}
If $\norm(Q_2)=\ring$ and $\norm(Q_1)\not=\ring$, then $H_{Q_0}(z) G_{Q_1}(z^2) G_{Q_2}(z^4)$ is just $\widetilde{H}_{Q_0}(z)$, which can be obtained from Table \ref{Rachel} or \ref{Edith}.
\end{lemma}
\begin{proof}
Lemma \ref{Samantha} allows us to replace $H_{Q_0}(z) G_{Q_1}(z^2) G_{Q_2}(z^4)$ by $H_{Q_0}(z) \widetilde{G}_{Q_1}(z^2) \widehat{G}_{Q_2}(z^4)$, which equals $H_{Q_0}(z) \widetilde{G}_{Q_1}(z^2) z^{4\ring}$ by Tables \ref{Rachel} and \ref{Edith}.
Then we may replace $\widetilde{G}_{Q_1}(z^2) z^{4\ring}$ by $\widehat{G}_{Q_1}(z^2) z^{4 \ring}$, which equals $z^{4\ring} z^{4 \ring}=z^{4 \ring}$ by Table \ref{Rachel}.  So we need only calculate $H_{Q_0}(z) z^{4 \ring}=\widetilde{H}_{Q_0}(z)$.
\end{proof}
\begin{lemma}\label{Clarence}
If $\norm(Q_2)\not=\ring$, then $H_{Q_0}(z) G_{Q_1}(z^2) G_{Q_2}(z^4) = H_{Q_0}(z) \widetilde{G}_{Q_1}(z^2)$.
\end{lemma}
\begin{proof}
Lemma \ref{Samantha} allows us to replace $H_{Q_0}(z) G_{Q_1}(z^2) G_{Q_2}(z^4)$ by $H_{Q_0}(z) \widetilde{G}_{Q_1}(z^2) \widehat{G}_{Q_2}(z^4)$, which equals $H_{Q_0}(z) \widetilde{G}_{Q_1}(z^2) z^{8\ring}$ by Table \ref{Rachel}.  Since $\widetilde{G}_{Q_1}(z^2)$ is a $8$-uniform, we see that $\widetilde{G}_{Q_1}(z^2) z^{8\ring}= \widetilde{G}_{Q_1}(z^2)$.
\end{proof}
\begin{lemma}\label{David}
If $\norm(Q_2)\not=\ring$ and $Q_1=\Planar{1}$, then
\[
H_{Q_0}(z) G_{Q_1}(z^2) G_{Q_2}(z^4)=\widetilde{H}_{Q_0}(z) \pm_1 \frac{1}{q^{r_1/2}} \left(H_{Q_0}(z)-\widetilde{H}_{Q_0}(z)\right).
\]
\end{lemma}
\begin{proof}
Lemma \ref{Clarence} allows us to replace $H_{Q_0}(z) G_{Q_1}(z^2) G_{Q_2}(z^4)$ with $H_{Q_0}(z) \widetilde{G}_{Q_1}(z^2)$, and we use the value of $\widetilde{G}_{Q_1}$ from Table \ref{Rachel} and apply Lemma \ref{Una}.
\end{proof}
\begin{lemma}
If $\norm(Q_0)$, $\norm(Q_1)$, and $\norm(Q_2)\not=\ring$, then $H_{Q_0}(z) G_{Q_1}(z^2) G_{Q_2}(z^4)=\widetilde{H}_{Q_0}(z)=H_{Q_0}(z)$, which can be obtained from Lemma \ref{Rebecca} or Table \ref{Rachel}.
\end{lemma}
\begin{proof}
Apply Lemma \ref{David}, and note from Table \ref{Rachel} that $H_{Q_0}(z)-\widetilde{H}_{Q_0}(z)=0$.
\end{proof}
\begin{lemma}
Suppose that $\norm(Q_2)\not=\ring$, that $Q_1=\Planar{1}$ is of rank $r_1$, and that $Q_0 = a \Sq\oplus\Planar{0}$ is of rank $r_0$.  Then
\[
\Ig(H_{Q_0}(z) G_{Q_1}(z^2) G_{Q_2}(z^4)) =\Ig(\widetilde{H}_{Q_0}(z)) = \Ig(H_{Q_0}(z)),
\]
which can be obtained from Lemma \ref{Christine} or Table \ref{Rachel}.
\end{lemma}
\begin{proof}
Apply Lemma \ref{David}, and note from Table \ref{Rachel} that $\Ig$ annihilates $H_{Q_0}(z)-\widetilde{H}_{Q_0}(z)$.
\end{proof}
\begin{lemma}\label{Felicia}
Suppose that $\norm(Q_2)\not=\ring$, that $Q_1=\Planar{1}$ is of rank $r_1$, and that $Q_0=a\Sq \oplus b\Sq\oplus\Planar{0}$ is of rank $r_0$ with $a\equiv b \equiv 1 \pmod{2}$.

If $4 \mid a+b$, let $\sigma =(-1)^{\Tr((a+b)/(4 a))}$, and then
\[
\Ig(H_{Q_0}(z) G_{Q_1}(z^2) G_{Q_2}(z^4)) = \left(1 -\frac{t^2}{q^{r_0}} \pm_0 \frac{t^2-t}{q^{r_0/2}} \pm_1 \frac{\sigma (t^3-t^2)}{q^{r_0+r_1/2}}\right) \igrp.
\]

If $4 \nmid a+b$, then
\[
\Ig(H_{Q_0}(z) G_{Q_1}(z^2) G_{Q_2}(z^4)) = \Ig(\widetilde{H}_{Q_0}(z))=\Ig(H_{Q_0}(z)),
\]
which can be obtained from Lemma \ref{Eveline} or Table \ref{Edith}.
\end{lemma}
\begin{proof}
Apply Lemma \ref{David}, and use Table \ref{Edith} to obtain $\Ig(\widetilde{H}_{Q_0}(z))$ and $\Ig(H_{Q_0}(z)-\widetilde{H}_{Q_0}(z))$.
\end{proof}
\begin{remark}\label{Vera}
Note that if $s \in R$ with $s \equiv 1 \pmod{2}$ then since scaling by $s$ does not change valuation, we should have
\[
\Ig(H_{s Q_0}(z) G_{s Q_1}(z^2) G_{s Q_2}(z^4)) = \Ig(H_{Q_0}(z) G_{Q_1}(z^2) G_{Q_2}(z^4)),
\]
and so it makes sense that the sign $\sigma$ in Lemma \ref{Felicia} should not change when we replace $(a,b)$ with $(s a, s b)$.  Note that we insist that $s \equiv 1 \pmod{2}$ to preserve our assumption that $a$ and $b$ should be $1$ modulo $2$.
\end{remark}
\begin{lemma}\label{Horace}
Suppose that $\norm(Q_2)\not=\ring$, and that $Q_1=c\Sq\oplus\Planar{1}$ is of rank $r_1$.  Then
\[
H_{Q_0}(z) G_{Q_1}(z^2) G_{Q_2}(z^4)= \widehat{H}_{Q_0}(z) \pm_1 \frac{1}{q^{(r_1+1)/2}} (H_{Q_0}(z)-\widetilde{H}_{Q_0}(z)) \sum_{v \in \teich} z^{2 c v}.
\]
\end{lemma}
\begin{proof}
Lemma \ref{Clarence} allows us to replace $H_{Q_0}(z) G_{Q_1}(z^2) G_{Q_2}(z^4)$ with $H_{Q_0}(z) \widetilde{G}_{Q_1}(z^2)$, and we use the value of $\widetilde{G}_{Q_1}$ from Table \ref{Rachel}.
We multiply this by $H_{Q_0}(z)$, and note that $H_{Q_0}(z) z^{2\ring}=\widehat{H}_{Q_0}(z)$ and $H_{Q_0}(z) z^{2 c \tau+8\ring}=H_{Q_0}(z) z^{2 c \tau}$ by the $8$-uniformity of $H_{Q_0}(z)$, and so obtain
\[
\left(1 \mp_1 \frac{1}{q^{(r_1-1)/2}}\right) \widehat{H}_{Q_0}(z) \pm_1 \frac{1}{q^{(r_1+1)/2}} H_{Q_0}(z) \sum_{v \in \teich} z^{2 c v},
\]
which is equal to
\[
\widehat{H}_{Q_0}(z) \pm_1 \frac{1}{q^{(r_1+1)/2}} (\widetilde{H}_{Q_0}(z)-\widehat{H}_{Q_0}(z)) \sum_{v \in \teich} z^{2 c v} 
\pm_1 \frac{1}{q^{(r_1+1)/2}} (H_{Q_0}(z)-\widetilde{H}_{Q_0}(z)) \sum_{v \in \teich} z^{2 c v}.
\]
Now note that $\widetilde{H}_{Q_0}(z)-\widehat{H}_{Q_0}(z)$ is necessarily a linear combination of terms of the form $z^{a+4\ring}-z^{a+2\ring}$, so that by coalescence of cosets (see Remark \ref{Anne}), we have
\[
\left(z^{a+4\ring}-z^{a+2\ring}\right) \sum_{v \in \teich} z^{2 c v} = 0.\qedhere
\]
\end{proof}
\begin{lemma}
Suppose that $\norm(Q_2)\not=\ring$, that $Q_1=c\Sq\oplus\Planar{1}$ is of rank $r_1$, that $Q_0=\Planar{0}$ is of rank $r_0$.  Then
\[
H_{Q_0}(z) G_{Q_1}(z^2) G_{Q_2}(z^4)=\widehat{H}_{Q_0}(z),
\]
which can be obtained from Table \ref{Rachel}.
\end{lemma}
\begin{proof}
Use Lemma \ref{Horace} and note from Table \ref{Rachel} that $H_{Q_0}(z)-\widetilde{H}_{Q_0}(z)=0$.
\end{proof}
\begin{lemma}
Suppose that $\norm(Q_2)\not=\ring$, that $Q_1=c\Sq\oplus\Planar{1}$ is of rank $r_1$, and that $Q_0=a\Sq\oplus\Planar{0}$ is of rank $r_0$.  Then
\[
\Ig(H_{Q_0}(z) G_{Q_1}(z^2) G_{Q_2}(z^4))=\Ig(\widehat{H}_{Q_0}(z)),
\]
which can be obtained from Table \ref{Rachel}.
\end{lemma}
\begin{proof}
We apply Lemma \ref{Horace}, and note from Table \ref{Rachel} that
\[
(H_{Q_0}(z)-\widetilde{H}_{Q_0}(z)) \sum_{v \in \teich} z^{2 c v} = 
- \frac{1}{q^{r_0}} \sums{\tau \in \teichu \\ v \in \teich} z^{a \tau+ 2 cv + 4 \ring} +
\frac{2}{q^{r_0+1}} \sums{\tau \in \teichu \\ s \in S \\ v \in \teich} z^{a \tau (1+4 s)+ 2 c v + 8 \ring},
\]
which is annihilated by $\Ig$.
\end{proof}
\begin{lemma}\label{Veronica}
Suppose that $\norm(Q_2)\not=\ring$, that $Q_1=c\Sq\oplus\Planar{1}$ is of rank $r_1$, and that $Q_0=a\Sq\oplus b\Sq\oplus\Planar{0}$ is of rank $r_0$ with $a \equiv b \equiv 1 \pmod{2}$ and $4 \mid a+b$, and let $\sigma =(-1)^{\Tr((a+b)/(4 a))}$.  Then
\[
\Ig(H_{Q_0}(z) G_{Q_1}(z^2) G_{Q_2}(z^4))= \left(1- \frac{t}{q^{r_0}} \pm_1 \frac{\sigma (t^3-t^2)}{q^{r_0+(r_1+1)/2}} \right) \igrp.\]
\end{lemma}
\begin{proof}
We apply Lemma \ref{Horace}, and use the values of $\widehat{H}_{Q_0}$ and $H_{Q_0}(z)-\widetilde{H}_{Q_0}(z)=\sigma (z^{8\ring}-z^{4\ring})/q^{r_0}$ from Table \ref{Edith}.  Note that $\Ig$ annihilates $(H_{Q_0}(z)-\widetilde{H}_{Q_0}(z)) z^{2 c v}$ when $v\in\teichu$ and gives $\sigma q^{-r_0} (t^3-t^2)(1-1/q)/(1-t/q)$ when $v=0$.
\end{proof}
\begin{remark}
By the same principle alluded to in Remark \ref{Vera}, the sign $\sigma$ in Lemma \ref{Veronica} is unchanged when we replace $(a,b,c)$ with $(s a, s b, s c)$ for some $s \in R$ with $s\equiv 1 \pmod{2}$.
\end{remark}
\begin{lemma}\label{Xavier}
Suppose that $\norm(Q_2)\not=\ring$, that $Q_1=c\Sq\oplus\Planar{1}$ is of rank $r_1$, and that $Q_0=a\Sq\oplus b\Sq\oplus\Planar{0}$ is of rank $r_0$ with $a \equiv b \equiv 1 \pmod{2}$ and $4 \nmid a+b$, and let $\phi=(-1)^{\Tr\left(\frac{c}{2} \left[\left(\frac{a+b}{2 c}\right)^q+\frac{a+b}{2 c}\right]\right)}$.  Then
\[
\Ig(H_{Q_0}(z) G_{Q_1}(z^2) G_{Q_2}(z^4))= \left(1 - \frac{t}{q^{r_0}} \pm_1 \frac{\phi(t^3-t^2)}{q^{r_0+(r_1+1)/2}} \right)\igrp 
\]
\end{lemma}
\begin{proof}
We apply Lemma \ref{Horace}, and use the value of $H_{Q_0}-\widetilde{H}_{Q_0}$ from Table \ref{Edith} to see that 
\[
(H_{Q_0}(z)-\widetilde{H}_{Q_0}(z)) \sum_{v \in \teich} z^{2 c v}=-\frac{q-1}{q^{r_0-1}} z^{2\ring} + \frac{2}{q^{r_0+1}} \sum_{\tau\in\teichu, s \in S, v \in \teich} z^{\tau \left(a+b+4 s\right)+2 c v + 8\ring}.
\]
Let us examine the triple sum.  First, reparameterize it with $v=w \tau$ with $w \in \teich$ to obtain
\[
\sum_{\tau\in\teichu, s \in S, w \in \teich} z^{2 c \tau \left(w+e+2 s/c\right)+8\ring},
\]
where we set $e=\frac{a+b}{2 c}$.  For each $w \in \teich$, there is a unique $x\in \teich$ such that $w+e\equiv x\pmod{2}$, and $x$ runs through $\teich$ as $w$ runs through $\teich$.
By Corollary \ref{Martin}, we have $w+e \equiv x + e^q + e + 2(x e)^{q/2} \pmod{4}$, and so we can again reparameterize our sum as
\[
\sum_{\tau\in\teichu, s \in S, x \in \teich} z^{2 c \tau \left(x+e^q+e+2(x e)^{q/2}+2 s/c\right)+8\ring}.
\]
where we note that $2\mid e^q+e$ by Lemma \ref{Richard}.  The only part of our expression for $H_{Q_0}(z) G_{Q_1}(z^2) G_{Q_2}(z^4)$ to which it is difficult to apply $\Ig$ is this sum.
To do this, we need to understand the valuation of the term $2 c \tau\left(x+e^q+e+2(x e)^{q/2}+2 s/c\right)$.  The valuation is $1$ when $x\not=0$, i.e., for $(q-1)^2 q/2$ of the triple sum's terms.

When $x=0$, we have $4 \tau \left(c(e^q+e)/2 + s\right)$, so the valuation is at least $2$.  It is exactly $2$ if and only if $c(e^q+e)/2 \not\equiv s \pmod{2}$.  If $\Tr(c(e^q+e)/2)\equiv 1 \pmod{2}$, this invariably happens, so we get a valuation of $2$ for $q(q-1)/2$ of our summation terms.  If $\Tr(c(e^q+e)/2) \equiv 0 \pmod{2}$, then we get a valuation of $2$ for $(q-1)(q/2-1)$ of our triple sum's terms, and a valuation of $3$ or higher for $q-1$ of our triple sum's terms.  Thus, if we apply $\Ig$ to the triple sum, when $\Tr(c(e^q+e)) \equiv 0\pmod{4}$, we get
\[
\frac{q(q-1)^2}{2} t + \frac{(q-1)(q-2)}{2} t^2 + (q-1) t^3 \igrp,
\]
and when $\Tr(c(e^q+e))\equiv 2 \pmod{4}$, we get.
\[
\frac{q(q-1)^2}{2} t + \frac{q(q-1)}{2} t^2.
\]
Using these calculations to apply $\Ig$ to $(H_{Q_0}(z)-\widetilde{H}_{Q_0}(z)) \sum_{v \in \teich} z^{2 c v}$ and the value of $\Ig(\widehat{H}_{Q_0}(z))$ from Table \ref{Edith}, we get the desired value of $\Ig(H_{Q_0}(z) G_{Q_1}(z^2) G_{Q_2}(z^4))$ via Lemma \ref{Horace}.
\end{proof}
\begin{remark}
By the same principle alluded to in Remark \ref{Vera}, the sign $\phi$ in Lemma \ref{Xavier} is unchanged when we replace $(a,b,c)$ with $(s a, s b, s c)$ for some $s \in R$ with $s\equiv 1 \pmod{2}$.  For $\frac{1}{2} \left[\left(\frac{a+b}{2 c}\right)^q+\frac{a+b}{2 c}\right]$ is invariant under this scaling, and is an element of $R$ since $u^q\equiv u \pmod{2}$ for any $u \in R$.  Thus when we scale by $s$, the element $\frac{c}{2} \left[\left(\frac{a+b}{2 c}\right)^q+\frac{a+b}{2 c}\right] \in R$ is scaled by a factor of $s$, so it is unchanged modulo $2$, and thus $\phi$ is unchanged.
\end{remark}
\begin{lemma}\label{Nelson}
Suppose that $\norm(Q_2)\not=\ring$ and that $Q_1=c\Sq\oplus d\Sq\oplus\Planar{1}$ is of rank $r_1$ with $c\equiv d \equiv 1\pmod{2}$ and $4 \mid c+d$.
Then
\[
H_{Q_0}(z) G_{Q_1}(z^2) G_{Q_2}(z^4)= \widehat{H}_{Q_0}(z) \pm_1 \frac{1}{q^{r_1/2}} \left(H_{Q_0}(z)-\widetilde{H}_{Q_0}(z)\right).
\]
\end{lemma}
\begin{proof}
Lemma \ref{Clarence} allows us to replace $H_{Q_0}(z) G_{Q_1}(z^2) G_{Q_2}(z^4)$ with $H_{Q_0}(z) \widetilde{G}_{Q_1}(z^2)$, and we use the value of $\widetilde{G}_{Q_1}(z)$ from Table \ref{Edith}.  We multiply this by $H_{Q_0}(z)$, and note that $H_{Q_0}(z) z^{2^j \ring}$ is just the $2^j$-uniformization of $H_{Q_0}(z)$.
\end{proof}
\begin{lemma}
Suppose that $\norm(Q_2)\not=\ring$ and that $Q_1=c\Sq\oplus d\Sq\oplus\Planar{1}$ is of rank $r_1$ with $c\equiv d \equiv 1\pmod{2}$ and $4 \mid c+d$.
Suppose that $Q_0=\Planar{0}$.
Then
\[
H_{Q_0}(z) G_{Q_1}(z^2) G_{Q_2}(z^4)= \widehat{H}_{Q_0}(z),
\]
which can be obtained from Table \ref{Rachel}.
\end{lemma}
\begin{proof}
Use Lemma \ref{Nelson} and note from Table \ref{Rachel} that $\widetilde{H}_{Q_0}(z)-H_{Q_0}(z)=0$.
\end{proof}
\begin{lemma}
Suppose that $\norm(Q_2)\not=\ring$ and that $Q_1=c\Sq\oplus d\Sq\oplus\Planar{1}$ is of rank $r_1$ with $c\equiv d \equiv 1\pmod{2}$ and $4 \mid c+d$.
Suppose that $Q_0=a\Sq\oplus\Planar{0}$.  Then
\[
\Ig(H_{Q_0}(z) G_{Q_1}(z^2) G_{Q_2}(z^4))= \Ig(\widehat{H}_{Q_0}(z)),
\]
which can be obtained from Table \ref{Rachel}.
\end{lemma}
\begin{proof}
Apply Lemma \ref{Nelson}, and note from Table \ref{Rachel} that $\Ig$ annihilates $H_{Q_0}(z)-\widetilde{H}_{Q_0}(z)$.
\end{proof}
\begin{lemma}\label{Ronald}
Suppose that $\norm(Q_2)\not=\ring$ and that $Q_1=c\Sq\oplus d\Sq\oplus\Planar{1}$ is of rank $r_1$ with $c\equiv d \equiv 1\pmod{2}$ and $4 \mid c+d$.
Suppose that $Q_0=a\Sq\oplus b\Sq\oplus\Planar{0}$ is of rank $r_0$.

If $4 \mid a+b$, let $\sigma =(-1)^{\Tr((a+b)/(4 a))}$, and then
\[
\Ig(H_{Q_0}(z) G_{Q_1}(z^2) G_{Q_2}(z^4))= \left(1- \frac{t}{q^{r_0}} \pm_1 \frac{\sigma(t^3-t^2)}{q^{r_0+r_1/2}}\right) \igrp.
\]
If $4\nmid a+b$, then
\[
\Ig(H_{Q_0}(z) G_{Q_1}(z^2) G_{Q_2}(z^4))= \Ig(\widehat{H}_{Q_0}(z)),
\]
which can be obtained from Table \ref{Edith}.
\end{lemma}
\begin{proof}
Apply Lemma \ref{Nelson}, and use Table \ref{Edith} to obtain $\Ig(\widehat{H}_{Q_0}(z))$ and $\Ig(H_{Q_0}(z)-\widetilde{H}_{Q_0}(z))$.
\end{proof}
\begin{remark}
By the same principle alluded to in Remark \ref{Vera}, the sign $\sigma$ in Lemma \ref{Ronald} is unchanged when we replace $(a,b,c,d)$ with $(s a, s b, s c, s d)$ for some $s \in R$ with $s\equiv 1 \pmod{2}$.
\end{remark}
\begin{lemma}\label{Roger}
Suppose that $\norm(Q_2)\not=\ring$ and that $Q_1=c\Sq\oplus d\Sq\oplus\Planar{1}$ is of rank $r_1$ with $c\equiv d \equiv 1\pmod{2}$ and $4 \nmid c+d$.
Then
\[
H_{Q_0}(z) G_{Q_1}(z^2) G_{Q_2}(z^4) = \widehat{H}_{Q_0}(z) \pm_1 \frac{2}{q^{(r_1+2)/2}}  \left(H_{Q_0}(z)-\widetilde{H}_{Q_0}(z)\right) \sums{v \in \teichu \\ u \in S} z^{2 v\left(c + \frac{4}{c+d} u\right)}.
\]
\end{lemma}
\begin{proof}
Lemma \ref{Clarence} allows us to replace $H_{Q_0}(z) G_{Q_1}(z^2) G_{Q_2}(z^4)$ with $H_{Q_0}(z) \widetilde{G}_{Q_1}(z^2)$, and we use the value of $\widetilde{G}_{Q_1}$ from Table \ref{Edith}.
We multiply this by $H_{Q_0}(z)$, and note that $H_{Q_0}(z) z^{2^j\ring}$ is the $2^j$-uniformization of $H_{Q_0}(z)$ and $H_{Q_0}(z) z^{2 v\left(c+\frac{4}{c+d}u\right)+8\ring}=H_{Q_0}(z) z^{2 v\left(c+\frac{4}{c+d}u\right)}$ because $H_{Q_0}(z)$ is $8$-uniform by Corollary \ref{Hortense}.  So we obtain
\begin{multline*}
H_{Q_0}(z) G_{Q_1}(z^2) G_{Q_2}(z^4)=\left(1\mp_1 \frac{1}{q^{(r_1-2)/2}}\right) \widehat{H}_{Q_0}(z) \pm_1 \frac{1}{q^{r_1/2}} \widetilde{H}_{Q_0}(z) \\ \pm_1 \frac{2}{q^{(r_1+2)/2}}  H_{Q_0}(z) \sums{\tau \in \teichu \\ s \in S} z^{2 \tau\left(c + \frac{4}{c+d} s\right)},
\end{multline*}
which is
\begin{multline*}
\left(1\mp_1 \frac{1}{q^{(r_1-2)/2}}\right) \widehat{H}_{Q_0}(z) \pm_1 \frac{1}{q^{r_1/2}} \widetilde{H}_{Q_0}(z)  \pm_1 \frac{1}{q^{r_1/2}}  \widetilde{H}_{Q_0}(z) \sums{\tau \in \teichu} z^{2 \tau} \\
\pm_1 \frac{2}{q^{(r_1+2)/2}}  \left(H_{Q_0}(z)-\widetilde{H}_{Q_0}(z)\right) \sums{\tau \in \teichu \\ s \in S} z^{2 \tau\left(c + \frac{4}{c+d} s\right)},
\end{multline*}
which in turn is
\[
\widehat{H}_{Q_0}(z) \pm_1 \frac{\widetilde{H}_{Q_0}(z)-\widehat{H}_{Q_0}(z)}{q^{r_1/2}}\sums{\tau \in \teich} z^{2 \tau c} 
\pm_1 \frac{2\left(H_{Q_0}(z)-\widetilde{H}_{Q_0}(z)\right)}{q^{(r_1+2)/2}} \sums{\tau \in \teichu \\ s \in S} z^{2 \tau\left(c + \frac{4}{c+d} s\right)}.
\]
Now note that $\widetilde{H}_{Q_0}(z)-\widehat{H}_{Q_0}(z)$ is necessarily a linear combination of terms of the form $z^{a+4\ring}-z^{a+2\ring}$, so that by coalescence of cosets (see Remark \ref{Anne}), we have
\[
\left(z^{a+4\ring}-z^{a+2\ring}\right) \sum_{\tau \in \teich} z^{2 \tau c} = 0.\qedhere
\]
\end{proof}
\begin{lemma}
Suppose that $\norm(Q_2)\not=\ring$ and that $Q_1=c\Sq\oplus d\Sq\oplus\Planar{1}$ is of rank $r_1$ with $c\equiv d \equiv 1\pmod{2}$ and $4 \nmid c+d$.
Suppose that $Q_0=\Planar{0}$.  Then
\[
H_{Q_0}(z) G_{Q_1}(z^2) G_{Q_2}(z^4)=\widehat{H}_{Q_0}(z),
\]
which can be obtained from Table \ref{Rachel}.
\end{lemma}
\begin{proof}
Apply Lemma \ref{Roger}, and note from Table \ref{Rachel} that $\widetilde{H}_{Q_0}=H_{Q_0}$.
\end{proof}
\begin{lemma}
Suppose that $\norm(Q_2)\not=\ring$ and that $Q_1=c\Sq\oplus d\Sq\oplus\Planar{1}$ is of rank $r_1$ with $c\equiv d \equiv 1\pmod{2}$ and $4 \nmid c+d$.
Suppose that $Q_0=a\Sq\oplus\Planar{0}$ is of rank $r_0$.  Then
\[
\Ig(H_{Q_0}(z) G_{Q_1}(z^2) G_{Q_2}(z^4))=\Ig(\widehat{H}_{Q_0}(z)),
\]
which can be obtained from Table \ref{Rachel}.
\end{lemma}
\begin{proof}
We apply Lemma \ref{Roger}, and use the value of $H_{Q_0}(z)-\widetilde{H}_{Q_0}(z)$ from Table \ref{Rachel}.  When it is multiplied by $\sums{v \in \teichu \\ u \in S} z^{2 v\left(c + \frac{4}{c+d} u\right)}$, one gets
\[
\frac{2}{q^{r_0+1}} \sums{t,v \in \teichu \\ s, u \in S} z^{a t (1+4 s)+ 2 v\left(c + \frac{4}{c+d} u\right)+8\ring}
-\frac{1}{q^{r_0}} \sums{t, v \in \teichu \\ u \in S} z^{a t + 2 v\left(c + \frac{4}{c+d} u\right) + 4\ring},
\]
which is clearly annihilated by $\Ig$.
\end{proof}
\begin{lemma}
Suppose that $\norm(Q_2)\not=\ring$ and that $Q_1=c\Sq\oplus d\Sq\oplus\Planar{1}$ is of rank $r_1$ with $c\equiv d \equiv 1\pmod{2}$ and $4 \nmid c+d$.
Suppose that $Q_0=a\Sq\oplus b\Sq\oplus\Planar{0}$ is of rank $r_0$ with $a\equiv b \equiv 1\pmod{2}$ and $4 \mid a+b$.
Then
\[
\Ig(H_{Q_0}(z) G_{Q_1}(z^2) G_{Q_2}(z^4)) = \Ig(\widehat{H}_{Q_0}(z)),
\]
which can be obtained from Table \ref{Edith}.
\end{lemma}
\begin{proof}
We apply Lemma \ref{Roger}, and use Table \ref{Edith} to obtain the value $H_{Q_0}(z)-\widetilde{H}_{Q_0}(z)=\sigma (z^{8\ring}-z^{4\ring})/q^{r_0}$, where $\sigma =(-1)^{\Tr((a+b)/(4 a))}$.
Note that $\Ig$ always annihilates $(H_{Q_0}(z)-\widetilde{H}_{Q_0}(z)) z^{2 v\left(c + \frac{4}{c+d} u\right)}$ when $v\in\teichu$.
\end{proof}
\begin{lemma}\label{Gertrude}
Suppose that $\norm(Q_2)\not=\ring$ and that $Q_1=c\Sq\oplus d\Sq\oplus\Planar{1}$ is of rank $r_1$ with $c\equiv d \equiv 1\pmod{2}$ and $4 \nmid c+d$.
Suppose that $Q_0=a\Sq\oplus b\Sq\oplus\Planar{0}$ is of rank $r_0$ with $a\equiv b \equiv 1\pmod{2}$ and $4 \nmid a+b$.

If $4\nmid a+b+c+d$, then 
\[
\Ig(H_{Q_0}(z) G_{Q_1}(z^2) G_{Q_2}(z^4)) = \Ig(\widehat{H}_{Q_0}(z)),
\]
which can be obtained from Table \ref{Edith}.

If $4\mid a+b+c+d$, let $\psi=(-1)^{\Tr\left(\frac{c}{2}\left[\left(\frac{c+d}{2}\right)^q+\frac{a+b}{2 c}\right]\right)}$, and then
\[
\Ig(H_{Q_0}(z) G_{Q_1}(z^2) G_{Q_2}(z^4)) = \left(1-\frac{t}{q^{r_0}} \pm_1 \frac{\psi(t^3-t^2)}{q^{r_0+r_1/2}}\right) \igrp.
\]
\end{lemma}
\begin{proof}
We apply Lemma \ref{Roger}, and use the value of $H_{Q_0}-\widetilde{H}_{Q_0}$ from Table \ref{Edith} to see that 
\[
\left(H_{Q_0}(z)-\widetilde{H}_{Q_0}(z)\right) \sums{v \in \teichu \\ u \in S} z^{2 v\left(c + \frac{4}{c+d} u\right)}
\]
is
\[
\frac{2-q}{2 q^{r_0-2}} z^{2\ring} - \frac{1}{2 q^{r_0-1}} z^{4\ring} + \frac{2}{q^{r_0+1}} \sums{\tau, v\in\teichu \\ s, u \in S} z^{\tau \left(a+b+4 s\right) + 2 v\left(c+\frac{4}{c+d} u\right) +8\ring}.
\]
Let us examine the quadruple sum.  First, reparameterize with $v = \tau w$ with $w \in \teichu$ to obtain
\[
\sums{\tau, w\in\teichu \\ s, u \in S} z^{2 \tau \left(c(w+e)+ 2 s  + 2 (w/g)  u\right)+8\ring},
\]
where $e=\frac{a+b}{2 c}$ and $g=\frac{c+d}{2}$.
Note that if $w \not\equiv g \pmod{2}$, then Lemma \ref{Eustace} tells us that $s + (w/g)u \pmod{2}$ runs through $\resfield$, taking each value $q^2/4$ times, as $(s, u)$ runs through $S^2$.  If $w \equiv g \pmod{2}$, then $s + (w/g)u \pmod{2}$ runs through $S$, taking each value $q/2$ times.
Lemma \ref{Richard} tells us that $w\equiv g^q \pmod{4}$ when $w \equiv g \pmod{2}$, so our quadruple sum becomes
\[
\frac{q^2}{4} \sums{\tau, w\in\teichu \\ w\not\equiv g \!\!\!\! \pmod{2}} z^{2 \tau c(w+e)+4\ring} + \frac{q}{2} \sums{\tau \in\teichu \\ s \in S} z^{2 \tau \left(c(g^q+e)+ 2 s\right)+8\ring},
\]
and we can sum over all $w$, and then deduct the $w=0$ and $w=g$ terms to get
\[
\frac{q^3(q-2)}{4} z^{2\ring} + \frac{q^2}{4} z^{4\ring} -\frac{q^2}{4} \sum_{\tau \in \teichu} z^{\tau(a+b+c+d)+4\ring} + \frac{q}{2} \sums{\tau \in\teichu \\ s \in S} z^{2 \tau \left(c(g^q+e)+ 2 s\right)+8\ring},
\]
where we have used the fact that $c(c+d) \equiv c+d \pmod{4}$ because $c\equiv d \equiv 1 \pmod{2}$.
When we combine this with the other terms, we get that
\begin{equation}\label{Colin}
\left(H_{Q_0}(z)-\widetilde{H}_{Q_0}(z)\right) \sums{v \in \teichu \\ u \in S} z^{2 v\left(c + \frac{4}{c+d} u\right)} =
-\frac{1}{2 q^{r_0-1}} \left(M-\frac{2}{q} N\right)
\end{equation}
where
\[
M = \sum_{\tau \in \teichu} z^{\tau(a+b+c+d)+4\ring} \text{\quad and \quad } N = \sums{\tau \in\teichu \\ s \in S} z^{2 \tau \left(c(g^q+e)+ 2 s\right)+8\ring}.
\]
We get different values for these sums depending on $\twoval(a+b+c+d)$, $\twoval(g^q+e)$, and in the case where $\twoval(g^q+e) \geq 1$ whether $\Tr(c(g^q+e)/2)$ is $0$ or $1$ modulo $2$.
The important cases to delineate for $M$ are (I) $\twoval(a+b+c+d)=1$ and (II) $\twoval(a+b+c+d) > 1$.  Note that case (I) is not achievable when $\ring=\Z_2$, but can be achieved in any unramified extension thereof.
Since $g^q \equiv g \pmod{2}$ and $c\equiv 1 \pmod{2}$ (making $e \equiv \frac{a+b}{2} \pmod{2}$), we see that $g^q+e \equiv \frac{a+b+c+d}{2} \pmod{2}$, so we see that $\twoval(g^q+e)=0$ in case (I) and $\twoval(g^q+e) > 0$ in case (II).
We then divide case (II) into case (IIA) where $\Tr(c(g^q+e)/2)\equiv 1 \pmod{2}$, and case (IIB) where $\Tr(c(g^q+e)/2) \equiv 0 \pmod{2}$.

In case (I), we have $M(z) = q z^{2\ring}-z^{4\ring}$,
and
\[
\Ig(M(z))=(q t- t^2) \igrp \text{\quad and \quad} \Ig(N(z)) =\frac{(q-1) q t}{2}.
\]
In case (IIA) or (IIB), we have $M(z) = (q-1) z^{4\ring}$, and
\[
\Ig(M(z))=(q-1) t^2 \igrp.
\]
Then in case (IIA), we obtain
\[
\Ig(N(z)) =\frac{(q-1) q t^2}{2},
\]
and in case (IIB), we have
\[
\Ig(N(z)) =\frac{(q-1)(q-2)t^2}{2} +(q-1) t^3 \igrp=\left(\frac{q t^3+q(q-2)t^2}{2}\right) \igrp. 
\]
Using these calculations to apply $\Ig$ to \eqref{Colin}, and also using value of $\Ig(\widehat{H}_{Q_0}(z))$ from Table \ref{Edith}, we get the desired value of $\Ig(H_{Q_0}(z) G_{Q_1}(z^2) G_{Q_2}(z^4))$ via Lemma \ref{Roger}.
\end{proof}
\begin{remark}
By the same principle alluded to in Remark \ref{Vera}, the sign $\psi$ in Lemma \ref{Gertrude} is unchanged when we replace $(a,b,c,d)$ with $(s a, s b, s c, s d)$ for some $s \in R$ with $s\equiv 1 \pmod{2}$.  To see this, recall that Lemma \ref{Gertrude} assumes that $a \equiv b \equiv c \equiv d \equiv 1 \pmod{2}$ and let us note that $\psi$ arises only in the case where $4 \mid a+b+c+d$, so we assume that these assumptions are always in force.  (Indeed, we insist that $s \equiv 1 \pmod{2}$ to preserve these assumptions.)  Now $u^q \equiv u \pmod{2}$ for any $u \in R$, and by our assumptions, we have $\left(\frac{c+d}{2}\right)^q+\frac{a+b}{2 c} \equiv \frac{a+b+c+d}{2} \equiv 0 \pmod{2}$, so that $\frac{1}{2}\left[\left(\frac{c+d}{2}\right)^q+\frac{a+b}{2 c}\right]$ is always an element of $R$ under our assumptions.
Furthermore, note that the term $(a+b)/(2 c)$ is unchanged when we apply our scaling by $s$.  The term $((c+d)/2)^q \in R$ is scaled by $s^q$, which is $1$ modulo $4$ because $s \equiv 1 \pmod{2}$.  Thus the term $\frac{1}{2}\left[\left(\frac{c+d}{2}\right)^q+\frac{a+b}{2 c}\right]$ is unchanged modulo $2$ under this scaling.  So $\frac{c}{2}\left[\left(\frac{c+d}{2}\right)^q+\frac{a+b}{2 c}\right]$ is unchanged modulo $2$ when we scale $(a,b,c,d)$ by $s$, and thus $\psi$ is unchanged under this scaling.
\end{remark}

\section*{Acknowledgements}

The authors thank Prof.~M.~Helena Noronha, who decided to fund this research project through National Science Foundation grant DMS 1247679.
The authors thank an anonymous referee for a very careful reading of the manuscript, with many helpful corrections and suggestions that have improved this paper.

\def\cprime{$'$}

\end{document}